\documentclass[12pt,a4paper,reqno]{amsart}
\usepackage{amssymb}
\usepackage{amscd}
\numberwithin{equation}{section}

\theoremstyle{plain}

\newtheorem{theorem}[subsection]{Theorem}
\newtheorem{proposition}[subsection]{Proposition}
\newtheorem{lemma}[subsection]{Lemma}
\newtheorem{corollary}[subsection]{Corollary}
\newtheorem{conjecture}[subsection]{Conjecture}

\newtheorem{claim}[subsection]{Claim}

\theoremstyle{definition}

\newtheorem{definition}[subsection]{Definition}

\newtheorem{remark}[subsection]{Remark}
\newtheorem{example}[subsection]{Example}

\renewcommand{\leq}{\leqslant}
\renewcommand{\geq}{\geqslant}

\newsavebox{\proofbox}
\savebox{\proofbox}{\begin{picture}(7,7)%
  \put(0,0){\framebox(7,7){}}\end{picture}}

\newcommand\E{\mathbb{E}}
\newcommand\Z{\mathbb{Z}}
\newcommand\R{\mathbb{R}}
\newcommand\T{\mathbb{T}}
\newcommand\C{\mathbb{C}}
\newcommand\N{\mathbb{N}}

\newcommand\F{\mathbb{F}}

\newcommand\eps{\varepsilon}
\newcommand\rank{\operatorname{rank}}
\newcommand\arank{\operatorname{arank}}
\newcommand\GI{\operatorname{GI}}
\newcommand\GIP{\operatorname{GIP}}

\newcommand\Poly{\operatorname{Poly}}

\newcommand\Fw{{\mathbb{F}^\omega}}

\newcommand\psd{{\operatorname{psd}}}
\newcommand\str{{\operatorname{str}}}

\newcommand\Diff{{\operatorname{Diff}}}
\newcommand\st{{\operatorname{st}}}
\newcommand\id{{\operatorname{id}}}

\renewcommand\mod{{\ \operatorname{mod}\ }}
\newcommand\ultra{{{}^*}}
\newcommand\charac{\operatorname{char}}
\newcommand\BAR{{\operatorname{BAR}}}
\newcommand\BR{{\operatorname{BR}}}
\newcommand\LR{{\operatorname{BR}}}
\newcommand\ER{{\operatorname{ER}}}
\newcommand\HK{{\operatorname{HK}}}
\newcommand\Sym{\operatorname{Sym}}
\newcommand\CSM{\operatorname{CSM}}
\renewcommand\th{{\operatorname{th}}}
\newcommand\ader{{\Delta}}
\newcommand\mder{{\Delta\!\!\!\!\!\hbox{\raisebox{0.2ex}{\tiny\ \textbullet}}\ \!}}

\newcommand\n{{\mathbf{n}}}

\begin{document}

\title[Finite fields in low characteristic]{The inverse conjecture for the Gowers norm over finite fields in low characteristic}

\author{Terence Tao}
\address{UCLA Department of Mathematics, Los Angeles, CA 90095-1596.
}
\email{tao@math.ucla.edu}

\author{Tamar Ziegler}
\address{Department of Mathematics, Technion, Haifa, Israel 32000.}
\email{tamarzr@tx.technion.ac.il}

\thanks{The first author is supported by a grant from the MacArthur Foundation, and by NSF grant CCF-0649473. 
The second author is supported by ISF grant 557/08, and by an Alon fellowship.  }

\subjclass{11B30, 11T06}

\begin{abstract}  We establish the \emph{inverse conjecture for the Gowers norm over finite fields}, which asserts (roughly speaking) that if a bounded function $f: V \to \C$ on a finite-dimensional vector space $V$ over a finite field $\F$ has large Gowers uniformity norm $\|f\|_{U^{s+1}(V)}$, then there exists a (non-classical) polynomial $P: V \to \T$ of degree at most $s$ such that $f$ correlates with the phase $e(P) = e^{2\pi i P}$.  This conjecture had already been established in the ``high characteristic case'', when the characteristic of $\F$ is at least as large as $s$.  Our proof relies on the weak form of the inverse conjecture established earlier by the authors and Bergelson \cite{berg}, together with new results on the structure and equidistribution of non-classical polynomials, in the spirit of the work of Green and the first author \cite{finrat} and of Kaufman and Lovett \cite{kauf}.
\end{abstract}

\maketitle

\section{Introduction}

\subsection{The inverse conjecture}

Let $\F = \F_p$ be a finite field of prime order $\charac(\F) = p$.  Throughout this paper, $\F$ will be considered fixed (e.g. $\F = \F_2$ or $\F = \F_3$), and the term ``vector space'' will be shorthand for ``vector space over $\F$'', and more generally any linear algebra term (e.g. span, independence, basis, subspace, linear transformation, etc.) will be understood to be over the field $\F$ unless otherwise stated.  

If $V$ is a vector space, $f: V \to \C$ is a function, and $h \in V$ is a shift, we define the \emph{multiplicative derivative} $\mder_h f: V \to \C$ of $f$ by the formula
$$ \mder_h f := (T_h f) \overline{f}$$
where the shift operator $T_h$ with shift $h$ is defined by $T_h f(x) := f(x+h)$.  If $V$ is finite, and $d \geq 1$ is an integer, we define the \emph{Gowers uniformity norm} $\|f\|_{U^d(V)}$ by the formula
$$ \|f\|_{U^{d}(V)} := |\E_{h_1,\ldots,h_d, n \in V} \mder_{h_1} \ldots \mder_{h_{d}} f(n)|^{1/2^{d}},$$
where we use the expectation notation $\E_{a \in A} f(a) := \frac{1}{|A|} \sum_{a \in A} f(a)$ for any finite non-empty set $A$, with $|A|$ denoting the cardinality of $A$.   We review some basic properties of the Gowers uniformity norms in Appendix \ref{gowapp}. 

The \emph{inverse conjecture} for the Gowers norm in finite characteristic addresses the question of determining those bounded functions $f: V \to \C$ with large Gowers norm.  To phrase this conjecture correctly in the low characteristic case, we need the notion of a \emph{non-classical polynomial}\footnote{Strictly speaking, ``not necessarily classical polynomial'' would be a more accurate terminology than ``non-classical polynomial''.}:

\begin{definition}[Polynomials]\label{polydef}  Let $V$ be a finite-dimensional vector space, let $d \geq 0$ be an integer, and let $G$ be an additive group.  A function $P: V \to G$ from $V$ to $G$ is said to be a \emph{(non-classical) polynomial} of degree $\leq d$ if one has
$$ \ader_{h_1} \ldots \ader_{h_{d+1}} f(n) = 0$$
for all $h_1,\ldots,h_{d+1},n \in V$, where $\ader_h f(n) := (T_h - 1) f(n) = f(n+h)-f(n)$ is the additive derivative of $f$ in the direction $h$.  We adopt the convention that the zero polynomial $0$ has degree $-\infty$.  We denote the space of all polynomials of degree $\leq d$ as $\Poly_{\leq d}(V \to G)$; this is clearly an additive group, with $\Poly_{\leq d}(V \to G) = \{0\}$ for $d<0$.  
\end{definition}

\begin{remark} In practice the group $G$ will usually be the unit circle $\T := \R/\Z$, the finite field $\F$, or the embedded copy $\iota(\F)$ of $\F$ in $\T$, where $\iota: \F \to \T$ is the additive homomorphism
$$ \iota(j) := \frac{j}{p} \mod 1.$$
In particular, $\iota(\F) = \frac{1}{p}\Z/\Z$ is the group of $p^{\th}$ roots of unity in $\R/\Z$.
\end{remark}

\begin{remark}
Polynomials that take values in $\F$ (and by abuse of notation, $\iota(\F)$) will be referred to as \emph{classical polynomials}; but in this paper, the term ``polynomial'' will be understood to encompass the non-classical case unless otherwise stated.  Clearly $\iota$ induces an isomorphism $\iota_*: \Poly_{\leq d}(V \to \F) \to \Poly_{\leq d}(V \to \iota(\F))$.   We will take advantage of this isomorphism whenever we need to use the multiplicative structure on $\F$, since $\T$ has no multiplicative structure, save for the fact that it is a $\Z$-module (i.e. one can define $n\alpha$ when $\alpha \in \T$ and $n \in \Z$).    
\end{remark}

\begin{remark} These notions of polynomials are part of a larger theory of \emph{polynomial algebra} between (filtered) groups that are not necessarily abelian; see Appendix \ref{poly-alg}.  
\end{remark}

\begin{example}\label{mother}  The map $P: \F_2 \to \T$ with $P(0) := 0$ and $P(1) := 1/2$ is a classical polynomial of degree $\leq 1$.  The map $Q: \F_2 \to \T$ with $Q(0) = 0$ and $Q(1) = 1/4$ is a (non-classical) polynomial of degree $\leq 2$; note that $\ader_1 Q = \frac{1}{4} - P$.  

We can generalise these examples to higher dimensions.  If $n \geq 1$ is an integer, we let $L: \F_2^n \to \Z$ be the function $L(x_1,\ldots,x_n) := |x_1|+\ldots+|x_n|$, where $x \mapsto |x|$ is the obvious map from $\F_2$ to the fundamental domain $\{0,1\}$.  This map is not a polynomial (either in the classical or non-classical sense); however, the function $\frac{L}{2} \mod 1$ is a classical polynomial of degree $\leq 1$ from $\F_2^n$ to $\T$, $\frac{L}{4} \mod 1$ is a (non-classical) polynomial of degree $\leq 2$ from $\F_2^n$ to $\T$, and more generally, for any $k \geq 0$, $\frac{L}{2^{k+1}} \mod 1$ is a (non-classical) polynomial of degree $\leq k$ from $\F_2^n$ to $\T$.  For further generalisation of these examples, see Lemma \ref{polybasic}(ii), (iii) below.
\end{example}

The relevance of (non-classical) polynomials to the Gowers norms can be seen from the easily verified fact that if $f: V \to \C$ is a function on a finite-dimensional vector space $V$ that is bounded in magnitude by $1$ (thus $|f(x)| \leq 1$ for all $x \in V$), and $d \geq 0$ be an integer, then $\|f\|_{U^d(V)} \leq 1$, with equality if and only if $f$ is of the form $f = e(P)$ for some polynomial $P: V \to \T$ of degree at most $d$, where $e: \T \to \C$ is the standard character $e(x) := e^{2\pi ix}$.

We collect the following standard facts about polynomials, setting $V$ to equal the standard finite-dimensional vector space $\F^n$ in order to use coordinates:

\begin{lemma}[Basic facts about polynomials]\label{polybasic}  Let $V = \F^n$ for some natural number $n$, $G$ be an additive group, and $d$ be an integer.
\begin{itemize}
\item[(i)] If $P: V \to G$ is a function and $d \geq 0$, then $P \in \Poly_{\leq d}(V \to G)$ if and only if $\ader_h P \in \Poly_{\leq d-1}(V \to G)$ for all $h \in V$.  In fact, we may replace ``for all $h \in V$'' by ``for all $h$ in a set that generates $V$''.
\item[(ii)] If $d \geq 0$, then a function $P: V \to \F$ is a polynomial of degree $\leq d$ if and only if it has a representation of the form
\begin{equation}\label{iot}
 P(x_1,\ldots,x_n) = \sum_{0 \leq i_1,\ldots,i_n < p: i_1+\ldots+i_n \leq d} c_{i_1,\ldots,i_n} x_1^{i_1} \ldots x_n^{i_n}
\end{equation}
for some coefficients $c_{i_1,\ldots,i_n} \in \F$, and furthermore these coefficients are unique.
\item[(iii)]  If $d \geq 0$, then a function $P: V \to \T$ is a polynomial of degree $\leq d$ if and only if it has a representation of the form
\begin{equation}\label{plan}
 P(x_1,\ldots,x_n) = \alpha + 
 \sum_{\text{\tiny $\begin{array}{ll} & 0 \leq i_1, \ldots,i_n < p; j \geq 0: \\  & 0 < i_1+\ldots+i_n \leq d-j(p-1) \end{array}$}} \frac{c_{i_1,\ldots,i_n,j} |x_1|^{i_1} \ldots |x_n|^{i_n}}{p^{j+1}} \mod 1
\end{equation}
for some coefficients $c_{i_1,\ldots,i_n,j} \in \{0,1,\ldots,p-1\}$ and $\alpha \in \T$, where $x \mapsto |x|$ is the map from $\F$ to the fundamental domain $\{0,1,\ldots,p-1\} \in \Z$.  Furthermore the coefficients $c_{i_1,\ldots,i_n,j}$ and $\alpha$ are unique.
\item[(iv)]  If $R$ is a commutative ring, and $P, Q: V \to R$ are polynomials of degree $\leq d$ and $\leq d'$ respectively, then $PQ: V \to R$ is a polynomial of degree $\leq d+d'$.
\item[(v)]  The map $p: P \mapsto pP$ is a homomorphism from 
$$\Poly_{\leq d}(V \to G) \to \Poly_{\leq \max(d-p+1,0)}(V \to G).$$  If $G=\T$ and $d \geq 0$, then this homomorphism is surjective.
\item[(vi)]  If $P \in \Poly_{\leq d}(V \to \T)$ and $d \geq 0$, then there exists $\alpha \in \T$ such that $P$ takes values in the coset $\alpha + \frac{1}{p^{\lfloor \frac{d-1}{p-1}\rfloor+1}} \Z/\Z$ of the $(p^{\lfloor \frac{d-1}{p-1}\rfloor+1})^{th}$ roots of unity.  In particular, $P$ takes on at most $p^{\lfloor \frac{d-1}{p-1}\rfloor+1}$ distinct values.
%\item[(vii)] If $P \in \Poly_{\leq p-2}(V \to G)$ and $x, v \in V$, then $pP=0$ if and only if $$\sum_{i \in \F_p} P(x+iv) = 0.$$
\end{itemize}
\end{lemma}

\begin{proof} See Appendix \ref{nonclass}.
\end{proof}

\begin{remark}
If follows from part (vi) that  if $d<p$ then the set of polynomials of degree $\le d$ in $\T$ coincides (up to constants) with the set of classical polynomials of degree $\le d$.  However, this statement is false in the low characteristic case $p \leq d$.
\end{remark}

\begin{remark}\label{exact} We isolate one of the claims in the above lemma for special comment, namely the surjectivity claim in part (v).  This claim implies that every polynomial $P$ of degree $\leq d$ for some $d \geq 0$ has a $p^{\th}$ root $Q$ (i.e. $pQ = P$) which is a (non-classical) polynomial of degree $\leq d+p-1$.  We refer to the ability to take $p^\th$ roots while losing exactly a factor of $p-1$ in the degree as the \emph{exact roots property}.  The exact roots property plays a crucial role in our proof of the inverse conjecture with the correct degree of polynomials involved.  Unfortunately, this property does not hold in the ergodic theory setting, which is one reason why our arguments here do not proceed via the ergodic theoretic approach; see Appendix \ref{root-counter} for further discussion.
\end{remark}

We now connect non-classical polynomials to the Gowers norms.  From the monotonicity of the Gowers norms and the Cauchy-Schwarz-Gowers inequality (see Lemma \ref{gow}(ii), (vi)) we see that if $f: V \to \C$ correlates with a polynomial $P \in \Poly_{\leq s+1}(V \to \T)$ in the sense that
$$ |\E_{x \in V} f(x) e(-P(x))| \geq \delta$$
for some $\delta > 0$, then we have $\|f\|_{U^{s+1}(V)} \geq \delta$.  

The \emph{inverse conjecture for the Gowers norm} over $\F$ is a converse of this statement:

\begin{conjecture}[Inverse conjecture $\GI(s)$]\label{conj-main}  Let $\delta > 0$ and $s \geq 0$.  Then there exists an $\eps = \eps_{\delta,s,\F} > 0$ such that for every finite-dimensional vector space $V$ and any $1$-bounded function $f: V \to \C$ with $\|f\|_{U^{s+1}(V)} \geq \delta$, there exists $P \in \Poly_{\leq s}(V \to \T)$ such that $$|\E_{x \in V} f(x) e(-P(x))| \geq \epsilon.$$
\end{conjecture}

For a fixed $s$, we denote the above conjecture as $\GI(s)$.  We now briefly review the history of progress on this conjecture.  The case $\GI(0)$ is trivial, while the case $\GI(1)$ follows easily from Plancherel's theorem.  The result was established for $\GI(2)$ in \cite{gt:inverse-u3} (for odd characteristic) and \cite{sam} (for even characteristic), and a formulation of Conjecture \ref{conj-main} was then conjectured in both papers.  In that formulation, the polynomial $P$ was assumed to be a classical polynomial rather than a non-classical one.  In subsequent work \cite{finrat}, \cite{lms}, it was shown that this ``classical'' formulation of the conjecture could fail in the low characteristic regime $p \leq s$; however the counterexamples in these papers did not prevent the ``non-classical'' formulation of Conjecture \ref{conj-main} given above from holding in those cases.

The case when $\delta$ is sufficiently close to $1$ (depending on $s$) was treated in \cite{akklr}, while the case when the characteristic $p$ is large compared to $s$ and $\delta$ was established in \cite{stv}.  In \cite{finrat}, Conjecture \ref{conj-main} was also established in the case when $f = e(P)$ for some $P \in \Poly_{< p}(V \to \T)$.  Finally, in \cite{tao-ziegler}, \cite{berg}, Conjecture \ref{conj-main} was established in the high-characteristic case $s < p$, and a weaker version of this conjecture established in the low-characteristic case (see Theorem \ref{gis2}).

The first main result of this paper is to extend the high characteristic result from \cite{tao-ziegler}, \cite{berg} to the low characteristic case also:

\begin{theorem}\label{main}  $\GI(s)$ is true for all choices of $\F$ and $s$.
\end{theorem}

\begin{remark}  As stated at the beginning of the introduction, we are restricting $\F$ to be a field of prime order.  But finite fields $\F_{p^j}$ of prime power order are also covered by this theorem, since any vector space over $\F_{p^j}$ can also be viewed as a vector space over $\F_p$, with no change in the definition of the Gowers norm or the definition of a (non-classical) polynomial.
\end{remark}

\begin{remark} In \cite{gowers-wolf-3}, Gowers and Wolf used the high characteristic case of Conjecture \ref{conj-main} to compute the true complexity of a system (a concept introduced in \cite{gowers-wolf}) of linear equations in sufficiently high characteristic.  In principle, Theorem \ref{main} would allow the ``sufficiently high characteristic'' condition to be weakened or dropped entirely.  However, this would require adapting the arguments in \cite{gowers-wolf-3} from classical polynomials to non-classical polynomials, and furthermore the high characteristic hypothesis is also used elsewhere in the arguments (in particular, the characteristic was assumed to exceed the ``Cauchy-Schwarz complexity'' of the system being studied).
\end{remark}

\subsection{Rank and analytic rank}

Theorem \ref{main} is established as a consequence of a related result (Theorem \ref{main-p} below), which is the main technical result of the paper.  Before we can state that result, we first need to recall the useful notions of rank and analytic rank, following \cite{gowers-wolf-3}.

\begin{definition}[Rank and analytic rank]  Let $s \geq 0$ be an integer, and let $P \in \Poly_{\leq s+1}(V \to \T)$.
\begin{itemize}
\item The \emph{rank} $\rank(P)= \rank_s(P)$ of $P$ is the least number $m$ of polynomials $Q_1,\ldots,Q_m \in \Poly_{\leq s}(V \to \T)$ of degree $\leq s$ such that $P$ is a function of $Q_1,\ldots,Q_m$, i.e. there exists a function $F: \T^m \to \T$ such that $P = F(Q_1,\ldots,Q_m)$.  (We adopt the convention that $\rank_0(P)$ is infinite if the linear polynomial $P$ is non-constant.)  
\item The \emph{analytic rank} $\arank(P) = \arank_{s}(P)$ of $P$ is defined to be the quantity $\arank(P) := -\log_p \|e(P)\|_{U^{s+1}(V)}^{1/2^{s+1}}$.
\end{itemize}
We define the rank and analytic rank of polynomials $P \in \Poly_{\leq s+1}(V \to \F)$ by using the homomorphism $\iota_*$, thus for instance $\arank(P) := \arank(\iota_* P)$.  (In particular, when defining the rank of a classical polynomial $P$, we allow for $P$ to be represented by \emph{non-classical} polynomials $Q_1,\ldots,Q_m$ of the required degree.)
\end{definition}

The analytic rank is closely related to the \emph{derivative} $d^{s+1} P: V^{s+1} \to \T$ of $P$, defined as
\begin{equation}\label{dsp-def}
d^{s+1} P(h_1,\ldots,h_{s+1}) := \ader_{h_1} \ldots \ader_{h_{s+1}} P(x)
\end{equation}
for any $h_1,\ldots,h_{s+1},x \in V$ (note that the right-hand side is independent of $x$ when $P$ has degree at most $s+1$).  Indeed, a short calculation shows that
\begin{equation}\label{arank-eq}
\E_{h_1,\ldots,h_{s+1} \in V} e(d^{s+1} P(h_1,\ldots,h_{s+1})) = p^{-\arank(P)}.
\end{equation}

Thus for instance $\arank(P)$ vanishes if and only if $P$ is of degree $\leq s$.

For a classical quadratic form $Q: V \to \F$, the rank $\rank(Q) = \rank_1(Q)$ and the analytic rank $\arank(Q) = \arank_1(Q)$ are both equal (at least when $p$ is odd) to the usual concept of the rank of a quadratic form in linear algebra; see \cite{gowers-wolf-3} for further discussion.  

\begin{example}\label{exam}  For each $k \geq 1$, let $S_k: \F_2^n \to \F_2$ be the degree $k$ classical symmetric polynomial
$$ S_k(x_1,\ldots,x_n) := \sum_{1 \leq i_1 < \ldots < i_k \leq n} x_{i_1} \cdots x_{i_k}$$
or equivalently, using the notation $L(x_1,\ldots,x_n) := |x_1|+\ldots+|x_n|$ from Example \ref{mother}, we have
$$ S_k = \binom{L}{k} \mod 2.$$
A classical theorem of Lucas (reflecting the self-similar fractal nature of Pascal's triangle modulo $2$) then gives the identity
$$ S_k = S_{2^{a_1}} \cdots S_{2^{a_m}}$$
whenever $k = 2^{a_1} + \ldots + 2^{a_m}$ is the binary expansion of $k$; thus for instance $S_3 = S_2 S_1$, $S_5 = S_4 S_1$, $S_6 = S_4 S_2$, $S_7 = S_4 S_2 S_1$, etc.  An easy induction on $m$ then shows that each $S^{2^m}$ is a binary coefficient of $L$, or more precisely that
\begin{equation}\label{lam}
L = \sum_{m=0}^\infty |S_{2^m}| 2^m
\end{equation}
(note that only finitely many of the summands are non-zero).

In \cite{lms}, \cite{finrat} it was computed that
$$ \| e(\iota(S_4)) \|_{U^4(\F_2^n)}^4 = \frac{1}{8} + o(1)$$
as $n \to \infty$, and thus $\arank(S_4) = \arank_3(S_4) = 3 + o(1)$.  As for the rank of $S_4$, it can be shown that as $n \to \infty$, $S_4$ cannot be expressed as a function of a bounded number of \emph{classical} cubics; see \cite{lms}, \cite{finrat}.  However, from \eqref{lam} we see that $S_4$ is a function of the expression $\frac{L}{8} \hbox{ mod 1}$, which is a cubic by Lemma \ref{polybasic}(iii).  We conclude that $\rank(S_4)=\rank_3(S_4)=1$.
\end{example}

We collect here some basic observations regarding rank and analytic rank:

\begin{lemma}[Basic properties of rank]\label{basic-rank} \ 
\cite{gowers-wolf-3} Let $s \geq 0$ be an integer, and let $P, Q \in \Poly_{\leq s+1}(V \to \T)$ be polynomials.  
\begin{itemize}
\item[(i)] $\rank(P) = \rank(-P)$ and $\arank(P) = \arank(-P)$.
\item[(ii)] $\rank(P+Q) \leq \rank(P) + \rank(Q)$ and $\arank(P+Q) \leq 2^{s+1} (\arank(P) + \arank(Q))$.
\item[(iii)] $\arank(P) \leq C_s \rank(P)$, for some constant $C_s > 0$ depending only on $s$.
\item[(iv)]  We have the inequality
$$ |\E_{h_1,\ldots,h_{s+1} \in V} e(d^{s+1} P(h_1,\ldots,h_{s+1})) \prod_{j=1}^{s+1} f_j(h_1,\ldots,h_{s+1})| \leq p^{-\arank(P)/2^{s}}$$
whenever $f_j: V^{s+1} \to \C$ are $1$-bounded functions, with each $f_j(h_1,\ldots,h_{s+1})$ independent of the $h_j$ variable.
\end{itemize}
\end{lemma}

\begin{proof} 
The claim (i) is trivial, as is the first part of claim (ii).  The second part of (ii) is established in \cite[Lemma 5.9]{gowers-wolf-3}.   Now we turn to (iii).  If we write $\rank(P) = m$, then $P$ is a function of polynomials $Q_1,\ldots,Q_m$ of degree $\leq s$.  By Lemma \ref{polybasic}(vi), we may assume that the $Q_1,\ldots,Q_m$ all take values in the $(p^C)^{\th}$ roots of unity for some $C = C(s)$.  By Fourier analysis, we may thus decompose
$$ e(P) = \sum_{0 \leq a_1,\ldots,a_m < p^C} c_{a_1,\ldots,a_m} e(a_1 Q_1 + \ldots + a_m Q_m)$$
where the Fourier coefficients $c_{a_1,\ldots,a_m}$ are bounded in magnitude by $1$ (in fact their $\ell^2$ norm is bounded by $1$).  By the pigeonhole principle, one can thus find $a_1,\ldots,a_m$ such that
$$ |\E_{x \in V} e(P(x)) e( - a_1 Q_1(x) - \ldots - a_m Q_m(x) )| \geq p^{-mC}$$
which by the monotonicity of Gowers norms implies that
$$ \| e(P) e( - a_1 Q_1 - \ldots - a_m Q_m )\|_{U^{s+1}(V)} \geq p^{-mC}.$$
Since the $U^{s+1}(V)$ norm is invariant with respect to modulation by polynomials of degree $\leq s$, we conclude that
$$ \| e(P) \|_{U^{s+1}(V)} \geq p^{-mC}.$$
and thus $\arank(P) \leq 2^s C m$, and the claim follows.

Finally, the claim (iv) is established in \cite[Lemma 5.4]{gowers-wolf-3}.
\end{proof}

Lemma \ref{basic-rank}(iii) asserts that the rank controls the analytic rank.  For each $s \geq 0$, we let $\GIP(s)$ denote the following converse:

\begin{conjecture}[Inverse conjecture $\GIP(s)$ for polynomials]\label{poly-cut}  Let $\delta > 0$.  Then there exists an integer $K = K_{\delta,s,\F} > 0$ such that for every finite-dimensional vector space $V$ and any $Q \in \Poly_{\leq s+1}(V \to \C)$ with $\|e(Q)\|_{U^{s+1}(V)} \geq \delta$, there exists $P_1,\ldots,P_K \in \Poly_{\leq s}(V \to \T)$ such that $Q$ is a function of $P_1,\ldots,P_K$ (i.e. $Q = F(P_1,\ldots,P_K)$ for some function $F: \T^K \to \T$.
\end{conjecture}

\begin{remark} This type of conjecture was introduced by Bogdanov and Viola \cite{bv}.  The case $\GIP(0)$ is trivial, and the case $\GIP(1)$ follows from Fourier analysis.  In the high characteristic case $p > s+1$, the result was established in \cite[Proposition 6.1]{finrat}.  
\end{remark}

Our main technical theorem is then

\begin{theorem}\label{main-p}  $\GIP(s)$ is true for all choices of $\F$ and $s$.
\end{theorem}

Theorem \ref{main} can be deduced from Theorem \ref{main-p} and the ``weak'' form of the inverse conjecture for the Gowers norm established in \cite{tao-ziegler}, \cite{berg}, together with an argument from \cite{gowers-wolf-3}; we give this (standard) argument in Section \ref{deduce}.

We remark that a different approach using ultrafilters to the structural theory of the Gowers norms is in the process of being carried out in \cite{szeg}, \cite{szeg-finite}, \cite{cs}.

We are indebted to the anonymous referee for a careful reading of the paper and many useful suggestions.

\section{An outline of the argument}\label{outline}

In this section we give an informal outline of how we will prove the inverse conjecture for polynomials $\GIP(s)$ (and hence the full inverse conjecture $\GI(s)$), suppressing many technical details (for instance, we will leave terms such as ``bounded'' vague for now, but such concepts will be made rigorous shortly with the assistance of nonstandard analysis).  We will also identify $\F$ with $\iota(\F) = \frac{1}{p}\Z/\Z$ for the purposes of this discussion.

As with many other arguments in this subject, we will induct on the degree parameter $s$, and assume that the conjecture $\GIP(s')$ has already been proven for all $s' < s$.  For sake of exposition we will work in the low characteristic case $p \leq s$, which is the hardest case.

Informally, the induction hypothesis allows one to obtain a completely satisfactory equidistribution theory for all (non-classical) polynomials of degree less than or equal to $s$, at least in principle.  For instance, if $P: V \to \T$ is a polynomial of degree $s'+1$ for some $s'<s$ that is of high rank, one can use $\GIP(s')$ to describe the equidistribution of the tuples $(P(x+h_1 \omega_1 + \ldots + h_d \omega_d ))_{\omega_1,\ldots,\omega_d \in \{0,1\}} \in \T^{\{0,1\}^d}$ for fixed $d$, as $x,h_1,\ldots,h_d$ ranges uniformly over $V$.  The precise equidistribution results we will need are rather technical to state, and will be formalised in Lemma \ref{symprod} and Lemma \ref{equil}.

The conjecture $\GIP(s)$ asserts, informally, that any non-classical polynomial of degree $\leq s+1$ of bounded analytic rank, also has bounded rank.  This conjecture can be established from three sub-claims, which we informally state as follows:

\begin{enumerate}
\item (Multiplication by $p$) If $P$ is a non-classical polynomial of degree $\leq s+1$ of bounded analytic rank, then $pP$ is a non-classical polynomial of degree $\leq s-p+2$ of bounded analytic rank.
\item (Division by $p$) If $Q$ is a non-classical polynomial of degree $\leq s-p+2$ of bounded rank, then there exists a non-classical polynomial $P'$ of degree $\leq s+1$ of bounded rank such that $pP' = Q$.
\item (Classical case) If $P: V \to \F$ is a classical polynomial of degree $\leq s+1$ of bounded analytic rank, then there exists a classical polynomial $Q:V \to \F$ of degree $\leq s+1$ and bounded rank such that $d^{s+1} P = d^{s+1} Q$, where the top order derivative $d^{s+1} P: V^{s+1} \to \F$ of $P$ was defined in \eqref{dsp-def}.  (In other words, $P$ and $Q$ differ by a polynomial of degree strictly less than $s+1$.) 
% Note that $\Delta_{h_1} \ldots \Delta_{h_{s+1}} P$ is of degree $\leq 0$, hence is constant, and so the definition \eqref{dsp-def} is well-defined.
\end{enumerate}

Indeed, assume these three claims hold.  Then if $P$ has bounded analytic rank, then by Claim (i), $pP$ has bounded analytic rank, and hence is of bounded rank by induction hypothesis.  By Claim (ii), we can thus find a polynomial $P'$ of bounded rank (and therefore also of bounded analytic rank) such that $pP' = pP$, so that $P-P'$ is classical while still having bounded analytic rank.  By Claim (iii), $P-P'$ differs from a classical bounded rank polynomial $Q$ by a polynomial of degree strictly less than 
$s+1$, and so $P$ has bounded rank also, as required.

It remains to verify the three claims.  After performing a Fourier expansion, Claim (i) will be an easy consequence of the multidimensional Szemer\'edi theorem for vector spaces, first proven by Bergelson, Leibman, and McCutcheon \cite{blm}; we will establish it in Section \ref{pmult-sec}.

We defer discussion of Claim (ii) for now, and move on to Claim (iii), which is somewhat easier to establish.  The strategy here is to obtain as much structural information on the expression $d^{s+1} P$ as possible, so that the bounded rank polynomial $Q$ can be constructed more or less explicitly.  It is easy to see that $d^{s+1} P: V^{s+1} \to \F$ is a symmetric multilinear form.  The fact that $P$ is classical gives an additional cancellation property, namely that $d^{s+1} P(h_1,\ldots,h_{s+1})$ necessarily vanishes whenever at least $p$ of the $h_1,\ldots,h_{s+1}$ are equal to each other.  For instance, if $p=2$ and $s+1=4$, we have
$$ d^{s+1} P(a,a,b,c) = 0,$$
as can be deduced from the identity
$$ \ader_a \ader_a P = (T_{2a} - 2T_a + 1) P = 2P - T_a (2P) = 0$$
since one has $2P=0$ when $P$ is classical.  We will refer to symmetric multilinear forms with this cancellation property as \emph{classical symmetric multilinear forms}.

The fact that $P$ has bounded analytic rank implies (from \eqref{arank-eq}) that the form $d^{s+1} P$ is \emph{biased}, in the sense that $\E_{h_1,\ldots,h_{s+1} \in V} e(d^{s+1} P(h_1,\ldots,h_{s+1}))$ is large.  To exploit this, we use a general equidistribution result of Kaufman and Lovett\cite{kauf} to conclude that $d^{s+1} P$ must be expressible in terms of lower degree classical symmetric multilinear forms.  We then apply a standard ``regularity lemma'' (analogous to those in \cite{finrat}, \cite{kauf}) to make these forms ``independent'' of each other, which makes them jointly equidistributed in a certain technical sense (see Lemma \ref{joint} for a precise statement).  With this equidistribution, one can control the precise manner in which $d^{s+1} P$ is a function of the lower degree forms, and we will end up showing that $d^{s+1} P$ is a certain symmetrised combination of such forms (on a bounded index subspace).  This will be made more precise in Section \ref{concat-sec}, but a typical example occurs when $s+1=4$, in which the expression
$$ d^4 P(a,b,c,d)$$
will be expressed as a linear combination of terms such as
\begin{equation}\label{bab}
 B(a,b) B(c,d) + B(a,c) B(b,d) + B(a,d) B(b,c)
\end{equation}
where $B: V \times V \to \F$ is a (classical) symmetric bilinear form.

To conclude Claim (iii), we thus need to rewrite expressions such as \eqref{bab} in the form $d^{s+1} Q$, where $Q$ is of bounded rank.  For sake of argument let us work specifically with the example \eqref{bab}.  As $B$ is itself classical, it can be expressed as $B = d^2 R$ for some classical quadratic polynomial $R: V \to \F$.  In the high characteristic case $p>2$, one can then proceed simply by setting $Q := R^2/2!$, as the claim can be verified from the discrete Leibniz rule \eqref{leibniz}.  However one cannot proceed so easily in the low characteristic case $p=2$, as one can no longer divide by $2!$ in this case.  Instead, we lift the polynomial $R$ (which takes values in $\F \equiv \Z/2\Z$) to the larger cyclic group $\Z/4\Z$, obtaining a cubic polynomial $\tilde R: V \to \Z/4\Z$ which projects back down to $R$ in the sense that $\tilde R = R \mod 2$.  We then set $Q := \binom{\tilde R}{2} \mod 2$ (instead of $R^2/2!$); the point is that the sequence $n \mapsto \binom{n}{2} \mod 2$ is periodic with period $4$ and so this expression is well defined.  One can show that $Q$ is a quartic polynomial, which is clearly of bounded rank as it depends only on the cubic polynomial $\tilde R$, and one can also compute that
$$ d^4 Q(a,b,c,d) =  B(a,b) B(c,d) + B(a,c) B(b,d) + B(a,d) B(b,c)$$
which gives the desired representation of \eqref{bab}.  The same arguments work in more general degrees and characteristics to give Claim (iii) in general; see Section \ref{concat-sec}.  Note that these constructions rely heavily on the classical nature of $R$, and hence of $B$ and $d^4 P$, and ultimately exploit the multiplicative structure of the classical range $\F$ that is not present in the non-classical range $\R/\Z$.  

Finally, we return to Claim (ii).  Let $Q$ be a non-classical polynomial of degree $\leq s-p+2$ of bounded rank, thus $Q$ is some combination of polynomials of degree strictly less than $s-p+2$.  By a ``regularity lemma'' argument, it will turn out to be possible to express $Q$ as the combination of ``independent'' polynomials $P_1,\ldots,P_k$ of various degrees between $1$ and $s-p+2$, with each $P_j$ taking values in some cyclic group $\Z/p^{J_j}\Z$, thus
$$ Q = F( P_1,\ldots, P_k )$$
for some function $F: (\Z/p^{J_1}\Z) \times \ldots \times (\Z/p^{J_k}\Z) \to \T$.  The precise nature of this independence is somewhat technical to state, but it implies good joint equidistribution properties on the $P_1,\ldots,P_k$ (see Proposition \ref{equil} for a formal statement of this).  The fact that $Q$ has degree $\leq s-p+2$ will imply that a suitable ``weighted degree'' of $F$ is also at most $\leq s-p+2$.

For technical reasons, the polynomials $P_i$ of degree $1$ will cause some difficulty (their $p^{th}$ roots will not have high rank).  But one can trivially eliminate all such polynomials by passing to a finite index subspace on which these polynomials are constant.  As such, one can easily reduce to the case where all polynomials $P_i$ have degree at least $2$.

It would be convenient if we could then find another function $G: (\Z/p^{J_1}\Z) \times \ldots \times (\Z/p^{J_k}\Z) \to \T$ of ``weighted degree'' $\leq s+1$ such that $pG = F$, as the function $P' := G(P_1,\ldots,P_k)$ would then obey the necessary requirements for Claim 2.  Unfortunately, this claim turns out to be false in general.  However, what one can do is first use the exact roots property from Lemma \ref{polybasic}(v) to obtain a $p^{th}$ root $P'_j$ for each $P_j$ with the ``right'' degree, thus $pP'_j = P_j$.  Because the $P_j$ have degree at least $2$, it turns out that the independence properties of the $P_j$ are inherited by the $P'_j$.  One can then rewrite $Q$ as $Q = F'(P'_1,\ldots,P'_k)$ where $F': (\Z/p^{J_1+1}\Z) \times \ldots \times (\Z/p^{J_k+1}\Z) \to \T$ is the pullback of $F$.  By analysing the concept of weighted degree for periodic functions on $\Z^k$ (and in particular by breaking such functions down into multinomials), we will be able to find a function $G': (\Z/p^{J_1+1}\Z) \times \ldots \times (\Z/p^{J_k+1}\Z) \to \T$ of the right weighted degree such that $pG' = P$, and then the function $P' := G'(P'_1,\ldots,P'_k)$ will obey the properties required for Claim (ii).  

The detailed proof of Claim (ii) will occupy Sections \ref{regsec}-\ref{erpsec}.

\section{Taking ultralimits}\label{ultra}

To prove Theorem \ref{main-p} it will be convenient to pass from the finitary setting to an infinitary one, in order to eliminate a profusion of epsilons, deltas, and growth functions (such as the growth functions ${\mathcal F}$ that appear for instance in \cite{finrat}, \cite{kauf}); it also allows us to conveniently make rigorous such phrases as ``the polynomials $P_1,\ldots,P_m$ are linearly independent modulo bounded rank errors'', which would otherwise only make sense heuristically or would need to be quantified with additional parameters.  In \cite{tao-ziegler} the Furstenberg correspondence principle was used to convert the inverse conjecture to a statement in an infinitary branch of mathematics, namely ergodic theory.  Unfortunately, the ergodic theory framework has a drawback in the low characteristic setting, namely that there does not appear to be an easy way to take roots of polynomials in this setting in a degree-efficient manner (see Appendix \ref{root-counter}).  To overcome this technical obstacle we shall take a different infinitary formulation of the problem, namely an \emph{ultralimit} (or \emph{nonstandard analysis}) formulation.  (Such formulations have also appeared in other recent work on the inverse conjecture \cite{gtz}, \cite{szeg}.)

The basic machinery of ultralimits and nonstandard analysis is recalled in Appendix \ref{nsa-app}, as is the asymptotic notation (such as $X \ll Y$ or $X=O(Y)$) associated with this machinery.  

We now translate all of the terminology used for Theorem \ref{main-p} to the ultralimit setting.

Let $V$ be a non-empty limit finite set (i.e. an ultralimit $V = \lim_{\n \to \alpha} V_\n$ of standard non-empty finite sets $V_\n$) and let $f: V \to \ultra \C$ be a limit function on $V$ (thus $f$ is an ultralimit $f = \lim_{\n \to \alpha} f_\n$ of standard functions $f_n: V_\n \to \C$).  Then we can define the expectation $\E_{x \in V} f(x)$ of $f$  on $V$ in the usual fashion by the formula
$$ \E_{x \in V} f(x) := \lim_{\n \to \alpha} \E_{x_\n \in V_\n} f_\n(x_\n).$$
This will be a limit complex number.  If further $V$ is a limit finite-dimensional vector space (i.e. each $V_n$ is a standard finite-dimensional space), and $d \geq 1$ is a standard natural number, then we can similarly define the uniformity norm
$$ \|f\|_{U^d(V)} := \lim_{\n \to \alpha} \|f_\n\|_{U^d(V_\n)}$$
which will be a limit non-negative real number.

Given a limit finite-dimensional space $V$ and a standard integer $d \in \Z$, we let $\Poly_{\leq d}(V \to \ultra \T)$ denote the space of all limit polynomials $P$ of degree $\leq d$ from $V$ to $\ultra \T$, i.e. all ultralimits $P = \lim_{\n \to \alpha} P_\n$ of polynomials $P_\n \in \Poly_{\leq d}(V_\n \to \T)$.  We define $\Poly_{\leq d}(V \to \F)$ and $\Poly_{\leq d}(V \to \iota(\F))$ for limit finite-dimensional $V$ similarly.  

We make the simple but important remark that all the claims in Lemma \ref{polybasic}, that were established for finite-dimensional vector spaces, extend to the limit finite-dimensional setting (replacing all operations by their limit counterparts) by taking ultralimits.  In particular, the surjectivity claim in Lemma \ref{polybasic}(v) extends to the limit finite-dimensional setting.

If $s \geq 0$ is standard and $P \in \Poly_{\leq s+1}(V \to \ultra \T)$, we can define the rank, $\rank(P) = \rank_s(P)$, and analytic rank, $\arank(P) = \arank_s(P)$, which are now non-negative limit integers and non-negative limit real numbers respectively.  We let $\Poly_{\leq s+1, \BR}(V \to \ultra \T)$ and $\Poly_{\leq s+1, \BAR}(V \to \ultra \T)$ denote the limit polynomials $P \in \Poly_{\leq s+1}(V \to \ultra \T)$ of  degree $\le s+1$ that are of bounded rank  and bounded analytic rank respectively.  From Lemma \ref{basic-rank} we see that $\Poly_{\leq s+1, \BR}(V \to \ultra \T)$ and $\Poly_{\leq s+1, \BAR}(V \to \ultra \T)$ are vector spaces with the inclusions
$$ \Poly_{\leq s}(V \to \ultra \T) \leq \Poly_{\leq s+1, \BR}(V \to \ultra \T) \leq \Poly_{\leq s+1, \BAR}(V \to \ultra \T) \leq \Poly_{\leq s+1}(V \to \ultra \T).$$
We similarly define $\Poly_{\leq s+1, \BR}(V \to \F)$ and $\Poly_{\leq s+1, \BAR}(V \to \F)$.

We can now give the ultralimit formulation of $\GIP(s)$:

\begin{theorem}[Ultralimit equivalence of $\GIP(s)$]\label{main-q} Let $s \geq 0$ be standard.  Then $\GIP(s)$ holds if and only if, for every limit finite-dimensional vector space $V$, $\Poly_{\leq s+1, \BR}(V \to \ultra \T) = \Poly_{\leq s+1, \BAR}(V \to \ultra \T)$ (i.e. bounded analytic rank and bounded rank are equivalent).
\end{theorem}

\begin{proof}
We first assume $\GIP(s)$ and verify that $\Poly_{\leq s+1, \BR}(V \to \ultra \T) = \Poly_{\leq s+1, \BAR}(V \to \ultra \T)$.  Accordingly, we let $P \in \Poly_{\leq s+1, \BAR}(V \to \ultra \T)$ and need to show that $P$ has bounded rank.

By construction, we can write $V = \prod_{\n \to \alpha} V_\n$ as the ultraproduct of finite-dimensional spaces $V_\n$, and similarly write $P = \lim_{\n \to \alpha} P_\n$ as the ultralimit of polynomials $P_\n: V_\n \to \T$.  Since $P$ has bounded analytic rank, the $P_\n$ have bounded analytic rank uniformly in $\n$ (at least for $\n$ sufficiently close to $\alpha$).  Applying $\GIP(s)$, we conclude that the $P_\n$ have bounded rank uniformly in $\n$, thus we can find polynomials $Q_{\n,1},\ldots,Q_{\n,K}: V_\n \to \T$ of degree $\leq s$ and a function $F_\n: \T^K \to \T$ such that $P_\n = F_\n(Q_{\n,1},\ldots,Q_{\n,K})$.  Writing $Q_i := \lim_{\n \to \alpha} Q_{\n,i}$ and $F := \lim_{\n \to \alpha} F_\n$, we conclude that $Q_1,\ldots,Q_K: V \to \ultra \T$ are limit polynomials of degree $\leq s$ and $P = F(Q_1,\ldots,Q_K)$, and so $P$ has bounded rank as desired.

Conversely, suppose that $\Poly_{\leq s+1, \BR}(V \to \ultra \T) = \Poly_{\leq s+1, \BAR}(V \to \ultra \T)$.  We assume for contradiction that $\GIP(s)$ failed.  Thus, there exists a sequence $V_\n$ of finite-dimensional vector spaces and polynomials $P_\n: V_\n \to \T$ of degree $\leq s+1$ whose analytic rank is bounded uniformly in $\n$, but whose rank goes to infinity as $\n \to \infty$.  Setting $V := \prod_{\n \to \alpha} V_\n$ and $P := \lim_{\n \to \alpha} P_\n$, we see that $P: V \to \ultra \T$ is a limit polynomial of degree $\leq s+1$ of bounded analytic rank, and hence of bounded rank by hypothesis, thus $P = F(Q_1,\ldots,Q_K)$ for some bounded $m$, some limit polynomials $Q_1,\ldots,Q_K: V \to \ultra \T$ of degree $\leq s$, and some function $F: \ultra \T^K \to \ultra \T$.  Writing $Q_i := \lim_{\n \to \alpha} Q_{\n,i}$ and $F := \lim_{\n \to \alpha} F_\n$, we conclude that $P_\n = F_\n(Q_{\n,1},\ldots,Q_{\n,m})$ for all $\n$ sufficiently close to $\alpha$, and thus the rank of $P_\n$ is bounded uniformly for such $\n$, which gives the desired contradiction.
\end{proof}

A similar argument (which we omit) gives the ultralimit formulation of $\GI(s)$:

\begin{theorem}\label{gis} Let $s \geq 0$ be standard.  Then $\GI(s)$ holds if and only if, for every limit finite-dimensional vector space $V$ and every bounded limit function $f: V \to \ultra \C$ (where ``bounded'' means that $\sup_{x \in V} |f(x)|$ is bounded), $\|f\|_{U^{s+1}(V)} \gg 1$ if and only if there exists $P \in \Poly_{\leq s}(V \to \ultra \T)$ such that $|\E_{x \in V} f(x) e(-P(x))| \gg 1$. 
\end{theorem}

It remains to establish that $\Poly_{\leq s+1, \BR}(V \to \ultra \T) = \Poly_{\leq s+1, \BAR}(V \to \ultra \T)$.  This is the purpose of the remaining sections of the paper.

\section{Splitting into three subclaims}

To prove the claim $\Poly_{\leq s+1, \BR}(V \to \ultra \T) = \Poly_{\leq s+1, \BAR}(V \to \ultra \T)$ (and hence $\GIP(s)$), we split this theorem into three subclaims as outlined in Section \ref{outline}.  Claim (i) is easily formalised:

\begin{theorem}[Multiplication by $p$]\label{ailp}  Let $V$ be a limit finite-dimensional vector space.  If $k>p$ is a standard integer, then the map $P \mapsto pP$ maps $\Poly_{\leq k,BAR}(V \to \ultra \T)$ to $\Poly_{\leq k-p+1,BAR}(V \to \ultra \T)$.
\end{theorem}

We prove this theorem in Section \ref{pmult-sec}.  Claim (ii) is also easily formalised, though for technical reasons (having to do with the need to eliminate all linear polynomials that arise in the regularity lemma) it is convenient to weaken the claim to a ``virtual'' version in which one only obtains roots on a bounded index subspace:

\begin{theorem}[Exact roots]\label{er}  Let $s \geq 1$ be a standard integer such that $\GIP(s')$ is true for all $0 \leq s' \leq s$.
Then for every $0 \leq s' \leq s$, every limit finite-dimensional $V$, and every $P \in \Poly_{\leq s'+1,\BR}(V \to \ultra \T)$, there exists a bounded index (limit) subspace $V'$ of $V$ and $Q \in \Poly_{\leq s'+p,\BR}(V' \to \ultra \T)$ such that $pQ=P$ on $V'$.
\end{theorem}

We will prove this theorem in Sections \ref{regsec}-\ref{erpsec}.

To formalise Claim (iii) properly, we will need some additional notation.

\begin{definition}[Multilinear maps]  Let $V$ be a limit finite-dimensional vector space and $k \geq 1$.
\begin{itemize}
\item A \emph{$k$-linear map} is a limit map $T: V^k \to \F$ such that if one fixes all but one variable $h_i$ of the $k$ variables $h_1,\ldots,h_k \in V$, the remaining map $h_i \mapsto T(h_1,\ldots,h_k)$ is linear.  
\item A $k$-linear map is \emph{symmetric} if it is invariant under permutations of the $k$ variables $h_1,\ldots,h_k$.  
\item A symmetric $k$-linear map is \emph{classical} if $T(h_1,\ldots,h_k)$ vanishes whenever at least $p$ of the $h_1,\ldots,h_k$ agree (this condition is of course vacuous for $k<p$).  Thus, for instance, if $p=2$ and $k=3$, a symmetric trilinear form $T$ is classical if $T(a,a,b) = 0$ for all $a,b \in V$.  
\end{itemize}
We abbreviate ``classical symmetric multilinear'' as $\CSM$.  We denote the space of classical symmetric $k$-linear maps $T: V^k \to \F$ as $\CSM^k(V)$; this is clearly a vector space.
\end{definition}

\begin{example}  Let $V = \F^n$ and $k=2$.  The map 
$$ T( h, h' ) := \sum_{1 \leq i < j \leq n} h_i h'_j$$
where $h = (h_1,\ldots,h_n), h' = (h'_1,\ldots,h'_n)$ is multilinear, but not symmetric or classical.  The map
$$ T'( h, h' ) := \sum_{1 \leq i, j \leq n} h_i h'_j$$
is symmetric and multilinear, but not classical.  The map
$$ T''( h, h' ) := \sum_{1 \leq i, j \leq n: i \neq j} h_i h'_j$$
is symmetric, multilinear, and classical, and thus lies in $\CSM^2(V)$.
\end{example}

Multilinear maps are naturally associated to derivatives of polynomials.  Indeed, from the cocycle equation
$$ \ader_{h+h'} = \ader_h + \ader_{h'} + \ader_h \ader_{h'}$$
and the commutativity identity
$$ \ader_h \ader_{h'} = \ader_{h'} \ader_h$$
we see that for any $P \in \Poly_{\leq k}(V \to \ultra\T)$, $d^k P$ is multilinear and symmetric.  In particular, from multilinearity $d^k P$ must take values in $\iota(\F)$.

From Lemma \ref{pmult} we see that $\ader_h^p$ and $p \ader_h$ differ (multiplicatively) by an invertible formal differential operator, which is equal to $-1$ plus higher order terms.  Applying this to a polynomial $P \in \Poly_{\leq k}(V \to \ultra\T)$ with $k \geq p$, we conclude the identity
\begin{equation}\label{dkp}
 d^k P(h_1,\ldots,h_1,h_2,\ldots,h_{k-p+1}) = - d^{k-p+1} (pP)(h_1,h_2,\ldots,h_{k-p+1})
\end{equation}
for any $h_1,\ldots,h_{k-p+1} \in V$, where $h_1$ appears $p$ times on the left-hand side (recall that $pP$ is of degree $\leq k-p+1$, by Lemma \ref{polybasic}).

This identity has a number of consequences.  For instance, it gives some additional constraints on $d^k P$ beyond symmetry, such as
$$ d^k P(h_1,\ldots,h_1,h_2,\ldots,h_{k-p+1}) = d^k P( h_1, h_2, \ldots, h_2, h_3, \ldots, h_{k-p+1}),$$
where $p$ copies of $h_1$ appear on the left, and $p$ copies of $h_2$ appear on the right.

Another consequence will be important for us:

\begin{lemma}[Derivative of classical polynomials]\label{never}  Let $k \geq 1$ and $V$ be a limit finite-dimensional vector space, then $d^k$ maps $\Poly_{\leq k}(V \to \F)$ to $\CSM^k(V \to \F)$, and furthermore this map is surjective.  In other words, we have the short exact sequence
$$ 0 \to \Poly_{\leq k-1}(V \to \F) \to \Poly_{\leq k}(V \to \F) \stackrel{d^k}{\to} \CSM^k(V) \to 0.$$
\end{lemma}

To state the above lemma loosely, classical symmetric multilinear forms are nothing more than the derivatives of classical polynomials.  

\begin{proof}  
Applying \eqref{dkp} to a classical polynomial $P \in \Poly_{\leq k}(V \to \F)$, we have $pP=0$, and thus $d^k P$ is classical, which gives the first claim.

Conversely, suppose that $T \in \CSM^k(V)$.  It suffices to verify the claim when $V$ is finite dimensional, as the limit finite-dimensional case then follows by taking ultralimits.  We take advantage of the finite-dimensionality to write $V = \F^n$ (without loss of generality), then the multilinear form $T$ can be expressed in coordinates as
$$ T(h_1,\ldots,h_k) = \iota\left( \sum_{1 \leq i_1,\ldots,i_k \leq n} c_{i_1,\ldots,i_k} h_{1,i_1} \ldots h_{k,i_k} \right)$$
where $h_{i,1},\ldots,h_{i,n} \in \F$ are the coordinates of $h_i$, and $c_{i_1,\ldots,i_k} \in \F$ are coefficients.  From the symmetric nature of $T$ we know that the $c_{i_1,\ldots,i_k}$ are symmetric with respect to permutations of the indices $i_1,\ldots,i_k$; from the classical nature of $T$ we know that the $c_{i_1,\ldots,i_k}$ vanish whenever $p$ or more of the $i_j$ are equal.  We thus see that $T$ is an integer linear combination of expressions of the form
$$ \iota\left( \sum_{\{i_1,\ldots,i_k\} = A} h_{1,i_1} \ldots h_{k,i_k} \right)$$
where $\{i_1,\ldots,i_k\}$ is the multiset formed by $i_1,\ldots,i_k$, and $A$ is a multiset of $k$ elements taking values in $\{1,\ldots,n\}$, with the multiplicity $a_j$ of each $1 \leq j \leq n$ in $A$ being less than $p$.  A short computation then shows that each such expression can be expressed as $d^k P$ for a polynomial $P \in \Poly_{\leq k}(V \to \F)$, indeed we may take
$$ P(x_1,\ldots,x_n) := \iota\left( \prod_{j=1}^n \frac{x_j^{a_j}}{a_j!} \right).$$
Note how the multiplicity bound $a_j < p$ allows for the factorial $a_j!$ to be inverted in $\F$.  (One could also use the binomial coefficient $\binom{x_j}{a_j}$ in place of $\frac{x_j^{a_j}}{a_j!}$ here if desired.) The claim now follows from linearity of $d^k$.
\end{proof}

Claim (iii) will be a variant of Lemma \ref{never} in the biased case:

\begin{theorem}[Inverse Gowers for classical symmetric multilinear forms]\label{GICSM} Let $s \geq 2$ be such that $\GIP(s')$ is true for all $0 \leq s' < s$, let $V$ be limit finite-dimensional, and let $T \in \CSM^{s+1}(V \to \F)$ be such that $|\E_{h_1,\ldots,h_{s+1} \in V} e(\iota(T(h_1,\ldots,h_{s+1})))| \gg 1$.  Then there exists a bounded index (limit) subspace $V'$ of $V$ and $P \in \Poly_{\leq s+1,\BR}(V' \to \F)$ such that $d^{s+1} P = T$ on $(V')^{s+1}$.
\end{theorem}

This theorem will be proven in Sections \ref{concat-sec}-\ref{multiconc}.
To state the theorem loosely, classical \emph{biased} symmetric multilinear forms are nothing more than the derivatives of classical \emph{bounded rank} polynomials.

In the remainder of this section, we show how Theorems \ref{ailp}, \ref{er}, \ref{GICSM} imply $\GIP(s)$.  We first need a technical lemma to handle the passage to bounded index limit subspaces.

\begin{lemma}\label{subspace}  Let $k \geq 2$ and $P \in \Poly_{\leq k}(V \to \ultra \T)$, and let $V'$ be a bounded index subspace of $V$.  (Note that such subspaces are automatically limit subspaces.)  Let $P' \in \Poly_{\leq k}(V' \to \ultra \T)$ be the restriction of $P$ to $V'$.
\begin{itemize}
\item[(i)] $P$ has bounded analytic rank if and only if $P'$ has bounded analytic rank.
\item[(ii)] $P$ has bounded rank if and only if $P'$ has bounded rank.
\end{itemize}
If $k=1$, then the ``only if'' portions of the claim continue to hold.
\end{lemma}

\begin{proof} We begin with (i).  If $P'$ has bounded analytic rank, then by \eqref{arank-eq}
$$
|\E_{h_1,\ldots,h_k \in V} 1_{V'}(h_1) \ldots 1_{V'}(h_k) e( d^k P(h_1,\ldots,h_k) )| \gg 1.$$
Applying Lemma \ref{basic-rank}(iv) we conclude that $P$ has bounded analytic rank, as required.

Conversely, suppose that $P$ has bounded analytic rank, then by \eqref{arank-eq} we have
$$
\E_{h_1,\ldots,h_k \in V} e( d^k P(h_1,\ldots,h_k) ) \gg 1.$$
Applying Fourier analysis, we conclude that for $\gg |V|^{k-1}$ $k-1$-tuples $(h_1,\ldots,h_{k-1})$ in $V^{k-1}$, one has $d^k P(h_1,\ldots,h_{k-1},\cdot) \equiv 0$.  Applying the pigeonhole principle, one can restrict $h_{k-1}$ to a single coset of $V'$ and still have the above claim.  Using multilinearity in $h_{k-1}$, one can in fact restrict $h_{k-1}$ to $V'$.  Iterating this argument one may restrict all of $h_1,\ldots,h_{k-1}$ to $V'$.  Using Fourier analysis we conclude that
$$
\E_{h_1,\ldots,h_k \in V'} e( d^k P'(h_1,\ldots,h_k) ) \gg 1,$$
and hence by \eqref{arank-eq}, $P$ has bounded analytic rank as desired.

Now we show (ii).  The ``only if'' portion is trivial, so we focus on the ``if'' part.  By induction we may assume that $V'$ is a hyperplane in $V$, and we may then write $V = V' \times \F$ without loss of generality.  Suppose that $P'$ is of bounded rank.  Letting $\pi: V \to V'$ be the coordinate projection from $V' \times \F$ to $V'$, it is easy to verify that $P' \circ \pi$ is a bounded rank polynomial of degree $\leq k$.  Subtracting this from $P$, we may assume without loss of generality that $P$ vanishes on $V'$.

On each coset of $h+V'$ of $V'$, $P$ is of degree $\leq k-1$ (since $\ader_h P$ is of degree $\leq k-1$).  Thus $1_{h+V'} P$ is a function of finitely many polynomials of degree $\leq k-1$.  Summing over a set of coset representatives $h$ we obtain the claim.
\end{proof}

Now we can prove $\GIP(s)$ assuming Theorems \ref{ailp}, \ref{er}, \ref{GICSM}.

\begin{proof}[Proof of $\GIP(s)$]
By induction we may assume that $\GIP(s')$ holds for all $0 \leq s' < s$; our task is now to show that $\GIP(s)$ holds.  We will assume $s \geq 2$, as the $s = 0, 1$ cases are well known.  

Let $P \in \Poly_{\leq s+1, \BAR}(V \to \ultra \T)$.  By Theorem \ref{main-q}, it will suffice to show that $P$ has bounded rank.

Let us first suppose that $s+1 > p$ (which is the most difficult case).  By Theorem \ref{ailp}, $p P \in \Poly_{\leq s-p+2, \BAR}(V \to \ultra \T)$, so by $\GIP(s-p+1)$, $pP$ is of bounded rank.  By Theorem \ref{er} (with $s$ replaced by $s-1$), there thus exists $Q \in \Poly_{\leq s+1, \BR}(V' \to \ultra \T) \subset \Poly_{\leq s+1,\BAR}(V' \to \ultra \T)$ such that $pP=pQ$ on a bounded index subspace $V'$ of $V$.  In particular, $P-Q$ takes values in $\iota(\F)$ on $V'$, so by Lemma \ref{never}, $d^{s+1}(P-Q) \in \CSM^{s+1}(V' \to \ultra \T)$.

By Lemma \ref{subspace} $P, Q$ both lie in $\Poly_{\leq s+1,\BAR}(V \to \ultra \T)$, and so $P-Q$ does also.  Thus
$$ \E_{h_1,\ldots,h_{s+1} \in V} e( d^{s+1}(P-Q)(h_1,\ldots,h_{s+1})) \gg 1.$$
Applying Theorem \ref{GICSM} (pulling back by $\iota$) we conclude that $d^{s+1}(P-Q) = d^{s+1}(W)$ on $(V'')^{s+1}$ for some $W \in \Poly_{\leq s+1,\BR}(V'' \to \iota(\F))$ on some bounded index subspace $V''$ of $V$, thus $P-Q-W$ is a polynomial of degree $\leq s+1$ on $V' \cap V''$.  Since $Q, W$ are of bounded rank on $V' \cap V''$, we conclude that $P$ is of bounded rank on $V' \cap V''$, and hence on $V$ by Lemma \ref{subspace}.

The case $s < p$ is similar.  Here, $pP \in \Poly_{\leq 1}(V \to \ultra \T)$, hence on passing to a subspace $pP$ is constant, and in particular we can find a constant $Q$ such that $pP = pQ$ on this subspace.  One then argues as before.
\end{proof}

It remains to prove Theorems \ref{ailp}, \ref{er}, \ref{GICSM}.  This is the purpose of the remaining sections of the paper.

\section{Multiplying by $p$}\label{pmult-sec}

In this section we prove Theorem \ref{ailp}.  The main tool will be the multidimensional Szemer\'edi theorem over finite fields\footnote{It is also possible to prove this proposition using the density Hales-Jewett theorem\cite{fk} instead.} from \cite{blm}, which we can formulate as follows:

\begin{proposition}[Multidimensional Szemer\'edi theorem]\label{mst}  Let $V$ be a limit finite-dimensional vector space, let $k \geq 1$ be a standard integer, and let $A \subset V^k$ be such that $|A| \gg |V|^k$.  Then there exist $\gg |V|^{k+1}$ tuples $(h_1,\ldots,h_k,h) \in V$ such that
$$ \{ (h_1 + h \omega_1, h_2 + h \omega_2, \ldots, h_k + h \omega_k): \omega_1,\ldots,\omega_k \in \F \} \subset A.$$
\end{proposition}

\begin{proof}  We combine the results from \cite{blm} with an averaging argument of Varnavides \cite{varnavides}.

The claim is trivial if $V$ has bounded dimension, so we may assume that $V$ has unbounded dimension.

Since $|A| \gg |V|^k$, if we write $A = \lim_{\n \to \alpha} A_\n$ and $V = \lim_{\n \to \alpha} V_\n$, then there is a standard $\delta > 0$ such that $|A_\n| \geq \delta |V_\n|^k$ for all $\n$ sufficiently close to $\alpha$.

Fix $\n$ with this property.  Let $M$ be a large standard integer depending on $k, \delta$ to be chosen later.  As $V$ has unbounded dimension, we may assume that $V_\n$ has dimension at least $M$ by taking $\n$ sufficiently close to $\alpha$.

Let $W_\n$ be an arbitrary linear subspace of $V_\n$ of dimension $M$.  We can then foliate $V_\n^k$ into $(|V_\n|/|W_\n|)^k$ cosets $x + W_\n^k$ of $W_\n^k$.  The average value of $|A_\n \cap (x+W_\n^k)|/|W_\n|^k$ in these cosets is at least $\delta$.  Thus, for at least $\delta (|V_\n|/|W_\n|)^k/2$ of these cosets, we have $|A_\n \cap (x+W_\n^k)|/|W_\n|^k \geq \delta/2$.

Consider one of these cosets $x+W_\n^k$ with $|A_\n \cap (x+W_\n^k)|/|W_\n|^k \geq \delta/2$. Applying \cite[Corollary 5.4]{blm}, we see that if $M$ is large enough depending on $k, \delta$, we can thus find a non-zero $h \in W_\n$ and $(h_1,\ldots,h_k) \in x + W_\n^k$ such that
\begin{equation}\label{silly}
\{ (h_1 + h \omega_1, h_2 + h \omega_2, \ldots, h_k + h \omega_k): \omega_1,\ldots,\omega_k \in \F \} \subset A_\n.
\end{equation}
Summing over all such cosets (and using the pigeonhole principle to fix $h$), we can thus find a non-zero $h \in W_\n$ such that \eqref{silly} holds
for at least $c_{k,M,\delta} |V_\n|^k$ tuples $(h_1,\ldots,h_k) \in V_\n^k$, where $c_{k,M,\delta} > 0$ is standard.  Averaging over all possible linear subspaces $W_\n$ of $V_\n$ of dimension $M$, we conclude from a routine double counting argument that in fact \eqref{silly} holds for at least $c'_{k,M,\delta} |V_\n|^{k+1}$ tuples $(h,h_1,\ldots,h_k) \in V_\n^{k+1}$, where $c'_{k,M,\delta}>0$ is standard.  Taking ultralimits as $\n \to \alpha$ we obtain the claim.
\end{proof}

Now we can prove Theorem \ref{ailp}.

\begin{proof}[Proof of Theorem \ref{ailp}]  Let $P \in \Poly_{\leq k,BAR}(V \to \ultra \T)$.  From \eqref{arank-eq} one has
\begin{equation}\label{ailp-eq}
 \E_{h_1,\ldots,h_k \in V} e( d^k P(h_1,\ldots,h_k) ) \gg 1.
\end{equation}
For each tuple $(h_1,\ldots,h_{k-1}) \in V^{k-1}$, the function $e(d^k P(h_1,\ldots,h_{k-1},\cdot))$ is a character on $V$. Thus, by Fourier analysis, for \eqref{ailp-eq} to hold we must have $\gg |V|^{k-1}$ tuples $(h_1,\ldots,h_{k-1}) \in V^{k-1}$, for which the form
$$ d^k P( h_1,\ldots,h_{k-1},\cdot)$$
vanishes identically.  Applying Proposition \ref{mst}, we can thus find $\gg |V|^k$ tuples $(h,h_1,\ldots,h_{k-1}) \in V^k$ such that
$$ d^k P( h_1 + \omega_1 h,\ldots,h_p + \omega_p h, h_{p+1}, \ldots, h_{k-1},\cdot)$$
vanishes identically for all $(\omega_1,\ldots,\omega_p) \in \{0,1\}^{p}$.  (Note how the hypothesis $k>p$ was needed here in order for this expression to make sense.)  Taking an alternating sum of these expressions, we thus have
$$ d^k P(h,\ldots,h,h_{p+1},\ldots,h_{k-1},\cdot)$$
vanishing identically for $\gg |V|^k$ tuples $(h,h_1,\ldots,h_{k-1}) \in V^k$, and thus for $\gg |V|^{k-p+1}$ tuples $(h,h_p,\ldots,h_{k-1}) \in V^{k-p+1}$.  Once again this implies that
$$ \E_{h,h_{p+1},\ldots,h_k \in V} e(d^k P(h,\ldots,h,h_{p+1},\ldots,h_{k-1},h_k)) \gg 1,$$
and the claim follows from \eqref{arank-eq}, and \eqref{dkp}.
\end{proof}

For future use, we record an immediate corollary of Theorem \ref{ailp}:

\begin{corollary}\label{cor} Suppose that $k > p$ and $\GIP(k-p)$ holds.  Then the map $p: P \mapsto pP$ maps 
$\Poly_{\leq k,\BR}(V \to \ultra \T)$ to $\Poly_{\leq k-p+1,\BR}(V \to \ultra \T)$.
\end{corollary}

We will use Corollary \ref{cor} in the following contrapositive sense: if $k > p$, $\GIP(k-p)$ holds, and $P \in \Poly_{\leq k}(V \to \ultra \T)$ is such that $pP$ has unbounded rank (as a polynomial of degree $\leq k-p+1$), then $P$ has unbounded rank (as a polynomial of degree $\leq k$).  Informally, roots of high rank nonlinear polynomials remain high rank.

We remark that Theorem \ref{ailp} (and Corollary \ref{cor}) fail when $k=p$.  For instance, if $p=2$ and $V = \F_2$, then the function $P(x) := |x|/4 \mod 1$ is a bounded rank polynomial of degree $2$, but $2P(x) = |x|/2 \mod 1$ is an infinite rank polynomial of degree $1$.  Because of the failure of Theorem \ref{ailp} at the endpoint $k=p$, we will need to require certain polynomials to have degree at least two in our arguments; but we will be able to eliminate all linear polynomials from our analysis by exploiting the freedom to pass to finite index subspaces.

\section{Multilinear concatenation}\label{concat-sec}

We now begin the proof of Theorem \ref{GICSM}.  The strategy will be to obtain enough control on the biased form $T \in \CSM^{s+1}(V)$ that one can explicitly write this form as $d^{s+1} P$ on a bounded index subspace $V'$ for some bounded rank polynomial $P \in \Poly_{\leq s+1,\BR}(V')$.

To get some intuition as to what expressions such as $d^{s+1} P$ look like, consider the case when $s+1=4$, $V'=V$, and $P$ takes the form $P = QR$ for some quadratic polynomials $Q,R \in \Poly_{\leq 2,\BR}(V \to \F)$.  Clearly, $P$ is of bounded rank.  A brief computation using the discrete Leibniz rule \eqref{leibniz} then reveals that
\begin{align*}
\ader_a \ader_b \ader_c \ader_d (QR) &= (\ader_a \ader_b Q)(\ader_c \ader_d R) + (\ader_a \ader_c Q) (\ader_b \ader_d R) + (\ader_a \ader_d Q) (\ader_b \ader_c R)\\
&\quad +(\ader_b \ader_c Q)(\ader_a \ader_d R) + (\ader_b \ader_d Q) (\ader_a \ader_c R) + (\ader_c \ader_d Q) (\ader_a \ader_b R)
\end{align*}
for any $a,b,c,d \in V$.  Thus, if we let $B,C \in \CSM^2(V \to \F)$ be the quadratic forms $B := d^2 Q$, $C := d^2 R$, then we have
\begin{align*}
d^4 P(a,b,c,d) &= B(a,b) C(c,d) + B(a,c) C(b,d) + B(a,d) C(b,c) \\ &+ B(b,c) C(a,d) + B(b,d) C(a,c) + B(c,d) C(a,b).
\end{align*}
By Lemma \ref{never}, we thus see that any quadrilinear form $T \in CSM^4(V)$ of the form
\begin{align*} T(a,b,c,d) &= B(a,b) C(c,d) + B(a,c) C(b,d) + B(a,d) C(b,c) \\ &+ B(b,c) C(a,d) + B(b,d) C(a,c) + B(c,d) C(a,b)
\end{align*}
for some $B,C \in \CSM^2(V)$ will be of the desired form $d^4 P$ for Theorem \ref{GICSM}.

To generalise this discussion to higher dimesions, we introduce the following operation.

\begin{definition}[Concatenation]  Let $S \in \CSM^k(V)$ and $T \in \CSM^l(V)$ for some standard integers $k,l \geq 1$.  We define the \emph{concatenation} $S \ast T \in \CSM^{k+l}(V)$ of $S$ and $T$ by the formula
$$ (S \ast T)(h_1,\ldots,h_{k+l}) = \sum_{\{1,\ldots,k+l\} = A \uplus B} S((h_i)_{i \in A}) T((h_j)_{j \in B})$$
where the sum ranges over all partitions of $\{1,\ldots,k+l\}$ into a $k$-element set $A$ and an $l$-element set $B$, and we define $S((h_i)_{i \in A})$ by enumerating $A$ arbitrarily (the precise ordering is not relevant due to the symmetry of $S$), and similarly for $T((h_j)_{j \in B})$.  Thus, for instance, if $S, T \in \CSM^2(V)$, then
$$ \begin{aligned} (S \ast T)(a,b,c,d) &:= S(a,b) T(c,d) + S(a,c) T(b,d) + S(a,d) T(b,c)\\ & + S(b,c) T(a,d) + S(b,d) T(a,c) + S(c,d) T(a,b).
\end{aligned}$$
\end{definition}

It is not hard to see that $S \ast T$ is indeed multilinear and symmetric.  The fact that it is classical is also easily seen after observing that all the binomial coefficients $\binom{p}{j}$ with $1 \leq j < p$ are divisible by $p$ and thus vanish on $\F$.  The operation $\ast$ is also easily seen to be bilinear, commutative, and associative.  As the previous discussion already indicated, this operation is closely related to multiplication on (classical) polynomials.  More precisely, we have:

\begin{lemma}[Product rule]\label{prod}  Let $k,l \geq 1$ be standard integers.  If $P \in \Poly_{\leq k}(V \to \F)$ and $Q \in \Poly_{\leq l}(V \to \F)$, then $d^{k+l}(PQ) = (d^k P) \ast (d^l Q)$.
\end{lemma}

\begin{proof}  We apply the discrete Leibniz rule \eqref{leibniz} repeatedly to expand out the derivative
\begin{equation}\label{hack}
 \ader_{h_1} \ldots \ader_{h_{k+l}}(PQ)
\end{equation}
with $h_1,\ldots,h_{k+l} \in V$.  Note that if $P$ accepts more than $k$ derivatives, or $Q$ accepts more than $l$ derivatives, then the resulting term in the expansion of \eqref{hack} vanishes.  Thus the only terms in \eqref{hack} that survive are those in which $P$ is differentiated exactly $k$ times, and $Q$ differentiated exactly $l$ times.  Collecting all such terms one obtains the claim.
\end{proof}

From this lemma and Lemma \ref{never}, we see that any form $T \in \CSM^{s+1}(V)$ which can be expressed (possibly after passing to a bounded index subspace) as a linear combination of concatenations $T_1 \ast \ldots \ast T_k$ of classical symmetric multilinear forms of degree strictly less than $s+1$, will satisfy the conclusions of Theorem \ref{GICSM}.

In the high characteristic case $p > s+1$, it turns out that these concatenations are the only expressions one needs to consider to establish Theorem \ref{GICSM}.  Unfortunately the situation is more complicated in the low characteristic case $p \leq s+1$.  This can be illustrated by using the symmetric polynomial $S_4 \in \Poly_{\leq 4}(\F_2^n \to \F_2)$ from Example \ref{exam} (this is on a finite-dimensional space rather than a limit finite-dimensional space, but let us ignore this technicality for this discussion).  A routine calculation reveals that the quartilinear form $d^4 S_4 \in \CSM^4(\F_2^n)$ can be expressed as
\begin{equation}\label{df}
d^4 S_4(a,b,c,d) = B(a,b) B(c,d) + B(a,c) B(b,d) + B(a,d) B(b,c)
\end{equation}
for all $a,b,c,d \in \F_2^n$, where $B \in \CSM^2(\F_2^n)$ is the bilinear form $B = d^2 S_2$, thus
$$ B(a,b) := \sum_{1 \leq i,j \leq n: i \neq j} a_i b_j.$$
(The identity \eqref{df}, which was already observed in \cite{lms}, \cite{finrat}, can be established by testing it on generators $a,b,c,d \in \{e_1,\ldots,e_n\}$.)  The right-hand side of \eqref{df} is formally of the form $(B \ast B)/2!$, but the operation of dividing by $2!$ is not well-defined in characteristic two, and so in fact one cannot easily express \eqref{df} in terms of the concatenation operation.  Instead, we have to introduce a new operation to handle expressions of this form:

\begin{definition}[Symmetric power]  Let $T \in \CSM^k(V)$ for some standard integers $k \geq 2$ and $m \geq 1$.  We define the \emph{symmetric power} $\Sym^m(T) \in \CSM^{mk}(V)$ by the formula
$$ \Sym^m(T)(h_1,\ldots,h_{mk}) = \sum_{{\mathcal A}} \prod_{A \in {\mathcal A}} T((h_i)_{i \in A}),$$
where the sum ranges over all partitions ${\mathcal A}$ of $\{1,\ldots,mk\}$ into $m$ subsets $A$ of cardinality $k$ each.  For instance, if $m=k=2$, then
$$ \Sym^2(T)(a,b,c,d) = T(a,b) T(c,d) + T(a,c) T(b,d) + T(a,d) T(b,c).$$
\end{definition}

Again, it is clear that $\Sym^k(T)$ is symmetric and multilinear; the fact that it is classical follows by observing that when $p$ of the arguments of $\Sym^k(T)$ are set to be equal, then the multiplicity of each term is a multiple of $p$ (as it is equal to $p!$ divided by a number of factorials that are strictly less than $p!$; note here we use the hypothesis $k \geq 2$).  Because of our need to avoid the $k=1$ case, we will have to take some care to eliminate all linear forms from the arguments in the next section, by using the trick of passing to a finite index subspace to make these forms vanish.

Observe that
$$ m! \Sym^m(T) = T \ast \ldots \ast T$$
where the right-hand side contains $m$ copies of $T$.  Thus, in high characteristic $p>m$, one can write the symmetric power in terms of the concatenation operation by the formula
\begin{equation}\label{symt}
 \Sym^m(T) = \frac{T \ast \ldots \ast T}{m!}.
\end{equation}
However, in the low characteristic case the symmetric power operation cannot be reduced easily to the concatenation operation, and we need to consider the two operations separately.

There is an analogue of Lemma \ref{prod}:

\begin{lemma}[Symmetric power rule]\label{symprod}  Let $k \geq 2$ and $m \geq 1$ be standard integers, and let $T \in \CSM^k(V)$ be a classical symmetric multilinear form. Then there exists $Q \in \Poly_{\leq mk}(V \to \F)$ such $d^{mk} Q = \Sym^m(T)$.  Furthermore, if $m \geq 2$, then $Q$ has bounded rank.
\end{lemma}

\begin{proof}  By Lemma \ref{never} we may write $T = d^k P$ for some $P \in \Poly_{\leq k}(V \to \F)$.

By a limiting argument it suffices to establish the claim when $V$ is finite dimensional, as long as the bound in ``bounded rank'' depends only on $p,k,m$ and not on the dimension of $V$.

Heuristically, in view of \eqref{symt} and Lemma \ref{prod}, it is natural to try to set $Q$ equal to $P^m/m!$.  This works in the high characteristic case $m<p$, but not in the low characteristic case $p \geq m$ due to the non-invertibility of $m!$.  To get around this, we will use the binomial coefficient $\binom{P}{m}$ instead of $P^m/m!$; but this requires lifting $P$ to a larger group than $\F = \Z/p\Z$.

We turn to the details.  We let $M \geq 0$ be the first integer such that $m<p^{M+1}$, thus $p^M \leq m$.  Using Lemma \ref{polybasic}(v) repeatedly, we may find a polynomial $\tilde P_M \in \Poly_{\leq k + M(p-1)}(V \to \T)$ such that $p^M \tilde P_M = \iota(P)$.  In particular, $\tilde P_M$ takes values in the $(p^{M+1})^\th$ roots of unity.  We may thus pull $\tilde P_M$ back to the cyclic group $\Z/p^{M+1}\Z$ to obtain a polynomial $P_M \in \Poly_{\leq k+M(p-1)}(V \to \Z/p^{M+1}\Z)$ such that $P_M = P \mod p$.

An inspection of the formula $\binom{n}{m} = \frac{n(n-1)\ldots(n-m+1)}{m (m-1)\ldots 1}$ for a binomial coefficient reveals that the map $n \mapsto \binom{n}{m} \mod p$ is periodic with period $p^{M+1}$ whenever $m < p^{M+1}$.  In particular, by abuse of notation we may define the binomial coefficient $\binom{n}{m} \mod p \in \F$ whenever $m < p^{M+1}$ and $n \in \Z/p^{M+1}\Z$.  We then set $Q := \binom{P_M}{m} \mod p$.

We first verify that $Q$ is a polynomial of degree $\leq mk$.  For inductive reasons, we will prove the more general claim that for any $j \geq 0$, any $h_1,\ldots,h_j \in V$, and any $0 \leq m' < p^{M+1}$, the expression $\binom{\ader_{h_1} \ldots \ader_{h_j} P_M}{m'}$ has degree at most $k-j + (m'-1) \max(k-j,1)$.  Clearly, this implies the previous claim by setting $j := 0$ and $m' := m$.

We first address the degenerate case when $k-j + (m'-1) \max(k-j,1)$ is negative, so in particular $m' \leq j-k$.  The polynomial $\ader_{h_1} \ldots \ader_{h_j} P_M$ has degree $\leq k+M(p-1)-j$, so by Lemma \ref{polybasic}(v), it is divisible by $p^{a+1}$ whenever $0 \leq a \leq M$ and $k+a(p-1)-j < 0$.  In particular, $\ader_{h_1} \ldots \ader_{h_j} P_M$ is divisible by $p^{\lfloor \frac{m'-1}{p-1} \rfloor + 1}$.  On the other hand, observe that $\binom{n}{m'}$ is divisible by $p$ whenever $n$ is divisible by $p^a$ and $m' < p^a$.  Since $m' < p^{\lfloor \frac{m'-1}{p-1} \rfloor + 1}$, we obtain the claim.

To handle the non-degenerate cases when $k-j + (m'-1) \max(k-j,1) \geq 0$, we use downward induction on $j$.  The claim is vacuously true for $j$ sufficiently large, so we assume inductively that the claim is proven for all larger values of $j$; for fixed $j$, we also assume inductively that the claim is proven for all smaller values of $m'$.  By Lemma \ref{polybasic}(i), it suffices to show that the expression
\begin{equation}\label{pam}
 \ader_{h_{j+1}} \binom{\ader_{h_1} \ldots \ader_{h_j} P_M}{m'}
\end{equation}
has degree at most $k-j + (m'-1) \max(k-j,1) - 1$ for all $h_{j+1} \in V$.

From the combinatorial identity
$$ \binom{n+n'}{m'} = \sum_{i=0}^{m'} \binom{n}{i} \binom{n'}{m'-i}$$
we see that
\begin{equation}\label{hafm}
 \ader_h \binom{F}{m'} = \sum_{i=1}^{m'} \binom{\ader_h F}{i} \binom{F}{m'-i}
\end{equation}
whenever $h \in V$ and $F: V \to \Z/p^{M+1}\Z$.  We may therefore expand \eqref{pam} as
$$ \sum_{i=1}^{m'} \binom{\ader_{h_1} \ldots \ader_{h_{j+1}} P_M}{i} \binom{\ader_{h_1} \ldots \ader_{h_{j}} P_M}{m'-i}.$$
In each summand, we apply the two induction hypotheses to conclude that the first factor in the summand has degree $\leq k-j-1 + (i-1) \max(k-j-1,1)$, and the second factor has degree $\leq k-j + (m'-i-1) \max(k-j,1)$.  A routine computation shows that
$$ (k-j-1 + (i-1) \max(k-j-1,1)) + (k-j + (m'-i-1) \max(k-j,1)) \leq k-j + (m'-1) \max(k-j,1) - 1$$
whenever $i \geq 1$ (treating the cases $k-j > 1$, $k-j \leq 1$ separately), and the claim then follows from Lemma \ref{polybasic}(iv).  Thus $Q$ has degree $\leq mk$ as desired.

Now we compute the derivative $d^{mk} Q$ of $Q = \binom{P_M}{m}$.  Using \eqref{hafm} we have
$$ \ader_h  \binom{P_M}{m} = \sum_{i=1}^m \binom{\ader_h P_M}{i} \binom{P_M}{m-i}$$
for any $h \in V$.  By the above computations, the polynomial $\binom{\ader_h P_M}{i} \binom{P_M}{m-i}$ has degree 
$$ \leq ((k-1) + (i-1) (k-1)) + (k + (m-i-1)k) = mk - i.$$
In particular, all the terms with $i>1$ have degree strictly less than $mk-1$ and thus will not contribute to $d^{mk}  \binom{P_M}{m}$.  The $i=1$ term can be simplified as $\Delta_h P \binom{P_M}{m-1}$.  We conclude that
$$ d^{mk} \binom{P_M}{m}(h_1,\ldots,h_{mk}) = d^{mk-1} ( (\ader_{h_{mk}} P)
 \binom{P_M}{m-1} ) (h_1,\ldots,h_{mk-1}).$$
Expanding this out using Lemma \ref{prod} we have
\[
\begin{aligned} & d^{mk} \binom{P_M}{m}(h_1,\ldots,h_{mk})
= (d^{k-1}(\ader_{h_{mk}}P)\ast d^{mk-k}  \binom{P_M}{m})(h_1,\ldots,h_{mk-1})\\
& = \sum_{1 \leq i_1 < \ldots < i_{k-1} < mk} d^k P(h_{i_1},\ldots,h_{i_{k-1}},h_{mk}) d^{(m-1)k} \binom{P_M}{m-1}(h_{j_1},\ldots,h_{j_{(m-1)k}})
\end{aligned}
\]
where $1 \leq j_1 < \ldots < j_{(m-1)k} < mk$ are the ordered enumeration of the set $\{1,\ldots,mk-1\} \backslash \{i_1,\ldots,i_{k-1}\}$.  The claim $d^{mk} \binom{P_M}{m} = \Sym^m(d^k P)$ then follows by induction on $m$.
\end{proof}

\begin{example}  We illustrate the above lemma with $p=m=k=2$, with $V = \F^n$ and $P = S_2$ the symmetric polynomial from Example \ref{exam}, thus $S_2: V \to \F$ is the classical quadratic polynomial
$$ S_2(x_1,\ldots,x_n) = \sum_{1 \leq i < j \leq n} x_i x_j.$$
The bilinear form $B := d^2 S_2 : V^2 \to \F$ is then given as
$$ B((x_1,\ldots,x_n),(y_1,\ldots,y_n)) = \sum_{1 \leq i,j \leq n: i \neq j} x_i y_j;$$
this is a classical symmetric bilinear form.

Lemma \ref{symprod} asserts the existence of a classical quartic $Q: V \to \F$ with
\begin{equation}\label{qbb}
d^4 Q(a,b,c,d) = B(a,b) B(c,d) + B(a,c) B(b,d) + B(a,d) B(b,c).
\end{equation}
By \eqref{df} we see that we can take $Q=S_4$.  To connect this with the proof of Lemma \ref{symprod}, we recall that $P = S_2 = L \mod 2$, so if we set $P_1 := L \mod 4$, then $P = P_1 \mod 2$, and
$S_4 = \binom{P_1}{2} \mod 2$.
\end{example}

Theorem \ref{GICSM} now follows from a more explicit claim:

\begin{theorem}[Explicit inverse Gowers for classical symmetric multilinear forms]\label{GICSME} Let $s \geq 2$ be such that $\GIP(s')$ is true for all $0 \leq s' < s$, let $V$ be limit finite-dimensional, and let $T \in \CSM^{s+1}(V)$ be such that $|\E_{h_1,\ldots,h_{s+1} \in V} e(\iota(T(h_1,\ldots,h_{s+1})))| \gg 1$.  Then there exists a bounded index subspace $V'$ of $V$ such that on $(V')^{s+1}$, $T$ is a linear combination (over $\F$) of a bounded number of expressions of the form
\begin{equation}\label{symmr}
\Sym^{m_1}(S_1) \ast \ldots \ast \Sym^{m_r}(S_r)
\end{equation}
for some  $m_1,\ldots,m_r \geq 1$ and $2 \leq k_1,\ldots,k_r < s+1$ and $S_i \in \CSM^{k_i}(V')$ for $i=1,\ldots,r$ with
$$ m_1 k_1+\ldots+ m_r k_r = s+1.$$
\end{theorem}

Indeed, by repeatedly applying Lemma \ref{never} we find that $S_i=d^{k_i}P_i$ for some $P_i \in \Poly_{\le k_i}(V \to \F)$.  By Lemma \ref{symprod} $\Sym^{m_i}(d^{k_i}P_i)=d^{k_im_i}Q_i$ for some $Q_i \in \Poly_{\le m_ik_i}(V \to \F)$, and $Q_i$ is of bounded rank if $m_i \ge 2$. Now by  Lemma \ref{prod}, we see that any expression of the form \eqref{symmr} can be expressed as $d^{s+1} Q$ on $V'$ where 
$Q=\prod_i Q_i \in \Poly_{\leq s+1,\BR}(V' \to \F)$ (note that the product of at least two polynomials of degree $\ge 1$ is necessarily of bounded rank), and Theorem \ref{GICSM} follows by linearity.

It remains to establish Theorem \ref{GICSME}.  To illustrate the type of result one is seeking here, in the case $s+1=6$, one has a classical symmetric sextilinear form $T \in \CSM^6(V)$ which is biased in the sense that
$$ |\E_{a,b,c,d,e,f \in V} e(\iota(T(a,b,c,d,e,f)))| \gg 1,$$
and one wishes to conclude that on a bounded index subspace $V'$ of $V$, $T(a,b,c,d,e,f)$ can be decomposed into a bounded number of pieces such as the expression
$$ B(a,b) B(c,d) B(e,f)$$
plus $\frac{6!}{2! 2! 2! 3!}-1 = 14$ other permutations (for some $B \in \CSM^2(V')$), adding up to $\Sym^3(B)$, or
$$ B(a,b) B(c,d) B'(e,f)$$
plus $\frac{6!}{2! 2! 2! 2!}-1 = 44$ other permutations (for some $B, B' \in \CSM^2(V')$), adding up to $\Sym^2(B) \ast B'$, or
$$ C(a,b,c) C(d,e,f)$$
plus $\frac{6!}{3! 3! 2!}-1 = 9$ other permutations (for some $C \in \CSM^3(V')$), adding up to $\Sym^2(C)$, or
$$ B(a,b) D(c,d,e,f)$$
plus $\frac{6!}{2! 4!} - 1 = 14$ other permutations (for some $B \in \CSM^2(V')$ and $D \in \CSM^4(V')$), adding up to $B \ast D$.

\section{Equidistribution of multilinear maps}\label{equimulti}

In order to establish Theorem \ref{GICSME} (and thus Theorem \ref{GICSM}) we will need an equidistribution theory for classical symmetric multilinear maps, analogous to that in \cite{finrat}, \cite{kauf}.  We introduce some definitions:

\begin{definition}[Bounded rank for multilinear forms]  Let $k \geq 1$ and let $V$ be limit finite-dimensional.  A form $T \in \CSM^k(V)$ is said to be \emph{bounded rank} if there exist a bounded number of forms $S_i \in \CSM^{k_i}(V)$, $i=1,\ldots,m$ for some $1 \leq k_i < k$ such that for $h_1,\ldots,h_k \in V$, the expression $T(h_1,\ldots,h_k)$ is a function of expressions of the form $S_i(h_{j_1}, \ldots, h_{j_{k_i}})$ for some $1 \leq i \leq m$ and $1 \leq j_1 < \ldots < j_{k_i} \leq k$, and \emph{unbounded rank} otherwise.  The space of bounded rank forms will be denoted $\CSM^k_{\LR}(V)$; it is clearly a subspace of $\CSM^k(V)$.
\end{definition}

Thus, for instance, if $k=4$ and $T$ takes the form
$$ T(a,b,c,d) = B(a,b) B(c,d) + B(a,c) B(b,d) + B(a,d) B(b,c)$$
(i.e. $T = \Sym^2(B)$) for some $B \in \CSM^2(V)$, then $T$ would be bounded rank.

Our starting point will be the following result of Kaufman and Lovett \cite{kauf} (which in turn is based on the earlier paper \cite{finrat}), which links bounded rank with bias:

\begin{proposition}[Bias criterion]\label{equi}  Let $k \geq 1$ be a standard integer, and let $V$ be limit finite-dimensional.  
Then a form $T \in \CSM^k(V)$ is bounded rank if and only if
$$ |\E_{h_1,\ldots,h_k \in V} e( \iota(T( h_1,\ldots,h_k) ) )| \gg 1.$$
\end{proposition}

\begin{proof} The ``only if'' part is easy: if $T$ is bounded rank, then by Fourier analysis, the function  $e( \iota(T(h_1,\ldots,h_k)) )$ is a bounded linear combination of functions of the form $e( \iota( \sum_{i=1}^m c_i S_i( h_{j_1},\ldots,h_{j_{k_i}} ) ) )$ for $c_i \in \F$.  Such functions can be factorised as $\prod_{j=1}^d f_j( h_1,\ldots,h_k)$ where the $f_j$ are bounded functions not depending on $h_j$.  On the other hand, we clearly have
$$ \E_{h_1,\ldots,h_k \in V} e( \iota(T( h_1,\ldots,h_k) ) ) \overline{e( \iota(T( h_1,\ldots,h_k) ) )} = 1.$$
Applying the pigeonhole principle we conclude that
$$ |\E_{h_1,\ldots,h_k \in V} e( \iota(T( h_1,\ldots,h_k) ) ) \prod_{j=1}^d \overline{f_j(h_1,\ldots,h_k)}| \gg 1$$
for at least one collection $f_1,\ldots,f_d$ of bounded functions; the claim then follows from Lemma \ref{basic-rank}(iv) (and Lemma \ref{never}).

We now turn to the ``if'' part.  This result follows easily from \cite[Theorem 2]{kauf}.  Indeed, applying that theorem, we see that $T(h_1,\ldots,h_k)$ is a function of a bounded number of polynomials $\ader_v T(h_1,\ldots,h_k)$ for $v \in V^k$, but by the multilinearity of $T$ we see that all such derivatives are functions of expressions of the form $S_i(h_{j_1}, \ldots, h_{j_{k_i}})$ for some $1 \leq j_1 < \ldots < j_{k_i} \leq k$, as required.
\end{proof}

Using the above proposition, we can obtain a criterion for joint equidistribution for certain systems of $\CSM$ forms; the precise definition of joint equidistribution is given in Definition \ref{equidef}.  We need another definition:

\begin{definition}[Systems of CSM forms]  Let $k_0 \geq 1$ be standard, and let $V$ be limit finite-dimensional.  A \emph{$\CSM$-system} $T = (T_{k,i})_{1 \leq k < k_0; 1 \leq i \leq m_k}$ of degree $<k_0$ is a collection of forms $T_{k,i} \in \CSM^k(V)$ for $1 \leq k < k_0$ and $1 \leq i \leq m_k$, where $m_1,\ldots,m_{k-1}$ are standard natural numbers.  A $\CSM$-system is said to be \emph{regular} if, for each $1 \leq k < k_0$, the forms $T_{k,1},\ldots,T_{k,m_k} \in \CSM^k(V)$ are linearly independent modulo $\CSM^k_\LR(V)$, thus one has
$$ a_1 T_{k,1} + \ldots + a_k T_{k,m_k} \not \in \CSM^k_\LR(V)$$
whenever $a_1,\ldots,a_k \in \F$ are not all zero.

Let $S \in \CSM^d(V)$ be a classical symmetric multilinear form of some (standard) degree $d \geq 1$.  We say that $S$ is \emph{measurable} with respect to a $\CSM$-system $T = (T_{k,i})_{1 \leq k < k_0; 1 \leq i \leq m_k}$ if there is a functional relationship of the form
$$ S = F( (T_\alpha)_{\alpha \in A} )$$
where $A$ is the set of tuples 
$$ \alpha = (k_\alpha, i_\alpha, j_{\alpha,1},\ldots,j_{\alpha,k_\alpha})$$
with $1 \leq k_\alpha < k_0$, $1 \leq i_\alpha \leq m_{k_\alpha}$, and $1 \leq j_{\alpha,1} < j_{\alpha,2} < \ldots < j_{\alpha,k_\alpha} \leq d$ (in particular, this forces $k_\alpha \leq d$), and $F: \F^A \to \F$ is a function.
\end{definition}

\begin{example}  A collection $L_1,\ldots,L_{m_1}: V \to \F$ of linear forms, together with a collection $B_1,\ldots,B_{m_2}: V^2 \to \F$ of classical symmetric bilinear forms, will form a $\CSM$-system of degree $<3$.  In order for this system to be regular, the linear forms $L_1,\ldots,L_{m_1}$ must be linearly independent, and no non-trivial linear combination of the $B_1,\ldots,B_{m_2}$ can be of bounded rank (i.e. expressible in terms of boundedly many linear forms).  A quartilinear form such as $B_1 \ast B_2$, $B_1 \ast L_1 \ast L_2$, $B_1 \ast \Sym^2(L_1)$, $\Sym^2(B_1)$, or linear combinations thereof, will be measurable with respect to $T$ (provided that $m_1,m_2$ are large enough so that these expressions make sense, of course).
\end{example}

Proposition \ref{equi} can now be recast as follows:

\begin{corollary}[Bias criterion, again]\label{equi-2}  Let $k \geq 1$ be a standard integer, and let $V$ be limit finite-dimensional.  
Let $T \in \CSM^k(V)$.  Then one has
$$ |\E_{h_1,\ldots,h_k \in V} e( \iota(T( h_1,\ldots,h_k) ) )| \gg 1$$
if and only if $T$ is measurable with respect to a $\CSM$-factor $S = (S_{k',i})_{1 \leq k' < k; 1 \leq i \leq m_{k'}}$ of degree $<k$.
\end{corollary}

The $\CSM$-factor $S$ given by the above corollary is not necessarily regular, but we may always regularise it as follows:

\begin{lemma}[Regularity lemma]\label{rego}  Let $k_0 \geq 1$ be a standard integer, and let $T = (T_{k,i})_{1 \leq k < k_0; 1 \leq i \leq m_k}$ be a $\CSM$-factor of degree $<k_0$.  Then there exists a \emph{regular} $\CSM$-factor $S = (S_{k,i})_{1 \leq k < k_0; 1 \leq i \leq m'_k}$ of degree $<k_0$, such that every multilinear form $T_{k,i}$ in $T$ is measurable with respect to $S$.
\end{lemma}

\begin{proof}  We assume inductively that the claim has already been proven for all smaller values of $k_0$ (this hypothesis is vacuous for $k_0=1$).

We consider the top-order forms $T_{k_0-1,1},\ldots,T_{k_0-1,m_{k_0-1}} \in \CSM^{k_0-1}(V)$, projected to the quotient space $\CSM^{k_0-1}(V) / \CSM^{k_0-1}_\LR(V)$.  As every finitely generated vector space has a finite basis, we may thus find forms $S_{k_0-1,1},\ldots,S_{k_0-1,m'_{k_0-1}} \in \CSM^{k_0-1}(V)$ that are linearly independent modulo $\CSM^{k_0-1}_\LR(V)$, thus that each $T_{k_0-1,i}(h_1,\ldots,h_k)$ is a linear combination (over $\F$) of the $S_{k_0-1,1}(h_1,\ldots,h_k),\ldots,S_{k_0-1,m'_{k_0-1}}(h_1,\ldots,h_k)$, plus a bounded rank form of $h_1,\ldots,h_k$, which by definition of bounded rank can be expressed as a function of a bounded family of forms $U( h_{i_1},\ldots,h_{i_{k'}})$ with $1 \leq k' < k$, $U \in \CSM^{k'}(V)$, and $1 \leq i_1 < \ldots < i_{k'} \leq k'$.  We may add all such forms $U$ to the list of lower order forms $T_{k',1},\ldots,T_{k',m_{k'}}$.  Applying the induction hypothesis to those lower order forms we then obtain the claim.
\end{proof}

Applying this regularity lemma to Corollary \ref{equi-2} we obtain the following improvement:

\begin{corollary}[Regularised bias criterion]\label{equi-3}  Let $k \geq 1$ be a standard integer, and let $V$ be limit finite-dimensional.  
Let $T \in \CSM^k(V)$.  Then one has
$$ |\E_{h_1,\ldots,h_k \in V} e( \iota(T( h_1,\ldots,h_k) ) )| \gg 1$$
if and only if $T$ is measurable with respect to a \emph{regular} $\CSM$-factor $S = (S_{k',i})_{1 \leq k' < k; 1 \leq i \leq m_{k'}}$ of degree $<k$.
\end{corollary}

To use this criterion, we need the following \emph{counting lemma} which complements the regularity lemma.

\begin{lemma}[Counting lemma]\label{joint}  Let $k_0 \geq 1$ be a standard integer, and let $T = (T_{k,i})_{1 \leq k < k_0; 1 \leq i \leq m_k}$ be a regular $\CSM$-system.  Then for any $d \geq 1$, the expressions 
$$ T_{k,i,j_1,\ldots,j_k}: (h_1,\ldots,h_d) \mapsto T_{k,i}( h_{j_1}, \ldots, h_{j_k} ),$$ 
where $1 \leq k < k_0$, $1 \leq i \leq m_k$, and $1 \leq j_1 < \ldots < j_k \leq d$, as functions from $V^d$ to $\F$, are jointly equidistributed (as defined in Definition \ref{equidef}).
\end{lemma}

\begin{proof}  
By the Weyl equidistribution criterion (Lemma \ref{weyl-equi}), it suffices to show that
\begin{equation}\label{head}
 \E_{h_1,\ldots,h_d \in V} e( \iota( \sum_{1 \leq k < k_0} \sum_{i=1}^{m_k} \sum_{1 \leq j_1 < \ldots < j_k \leq d} c_{k,i,j_1,\ldots,j_k} T_{k,i}(h_{j_1},\ldots,h_{j_k})) = o(1)
\end{equation}
whenever $c_{k,i,j_1,\ldots,j_k} \in \F$ are not all zero.

Let $1 \leq k_1 < k_0$ be the largest $k$ for which there is a non-zero coefficient $c_{k,i,j_1,\ldots,j_k}$.  By relabeling we may assume that it is $c_{k_1,1,1,\ldots,k_1}$ which is non-zero.  We may then factorise the left-hand side of \eqref{head} as
$$ \E_{h_1,\ldots,h_d \in V} e( \iota(T_*(h_1,\ldots,h_{k_1}))) \prod_{i=1}^{k_1} f_i(h_1,\ldots,h_d)$$
where $f_i: V \to \ultra \C$ are limit functions bounded in magnitude by $1$ which are independent of the $h_i$ variable, and $T_*: V^{k_1} \to \F$ is the multilinear form
$$ T_*(h_1,\ldots,h_{k_1}) := \sum_{i=1}^{m_1} c_{k_1,i,1,\ldots,k_1} T_{k_1,i}(h_1,\ldots,h_{k_1}).$$
By hypothesis, $T_*$ has unbounded rank.  By Proposition \ref{equi}, we conclude that
$$ \E_{h_1,\ldots,h_d \in V} e( \iota(T_*(h_1,\ldots,h_{k_1})) ) = o(1).$$
Using Lemma \ref{basic-rank} (and Lemma \ref{never}) we conclude that
$$ \E_{h_1,\ldots,h_d \in V} e( \iota(T_*(h_1,\ldots,h_{k_1}) ) \prod_{i=1}^{k_1} f_i(h_1,\ldots,h_d)) = o(1)$$
and the claim follows.
\end{proof}

\begin{remark} The counting lemma is essentially asserting that a regular $\CSM$-factor is equidistributed on cubes in the sense of Definition \ref{strong-equi}, but to formalise this rigorously, one needs to develop a theory of polynomial maps in several variables, that generalises the material in Section \ref{poly-alg} to groups filtered by $\N^d$ rather than $\N$.  This can be done (see \cite{gtz}), but we will not introduce this additional notation here.
\end{remark}

In the next section we will use the counting and regularity lemmas to finish off the proof of Theorem \ref{GICSME} and thus Theorem \ref{GICSM}.

\section{Conclusion of the multilinear inverse conjecture}\label{multiconc}

We are now ready to complete the proof of Theorem \ref{GICSME}.  Let $s, T, V$ be as in that theorem.  Applying Corollary \ref{equi-3}, we conclude that we can write 
$$T = F((S_\alpha)_{\alpha \in A}),$$
where $(S_{k,i})_{1 \leq k \leq s; 1 \leq i \leq m_k}$ is a regular $\CSM$-system, $A$ is the set of tuples
$$ \alpha = (k_\alpha, i_\alpha, j_{\alpha,1},\ldots,j_{\alpha,k_\alpha})$$
with $1 \leq k_\alpha \leq s$, $1 \leq i_\alpha \leq m_{k_\alpha}$, and $1 \leq j_{\alpha,1} < j_{\alpha,2} < \ldots < j_{\alpha,k_\alpha} \leq s+1$, $F: \F^A \to \F$ is a function, and
for each $\alpha \in A$, $S_\alpha \in \CSM^{k_\alpha}(V)$ is the form
$$ S_\alpha(h_1,\ldots,h_{k_\alpha}) := S_{k_\alpha,i_\alpha}( h_{j_{\alpha,1}}, \ldots, h_{j_{\alpha,k_\alpha}} ).$$

The linear forms $S_{1,1},\ldots,S_{1,m_{1}}$ in the regular $\CSM$-system can be eliminated by observing that they are simultaneously constant on some bounded index subspace of $V$, and so by passing to that subspace (using Lemma \ref{subspace}) we may assume that $m_1=0$, i.e. all forms $S_{k,i}$ in the $\CSM$-system are bilinear or higher in order.  Of course, the $\CSM$-system remains regular after doing so.

For each $S_\alpha$, we refer to the set $J_\alpha := \{j_{\alpha,1},\ldots,j_{\alpha,k_\alpha}\} \subset \{1,\ldots,s\}$ as the \emph{support} of $\alpha$.

The permutation group $\operatorname{Sym}(\{1,\ldots,s+1\})$ acts on the variables $h_1,\ldots,h_{s+1}$, and thus permutes the index set $A$.  Because $T$ is symmetric (and $(S_\alpha)_{\alpha \in A}: V \to \F^A$ is surjective, by Lemma \ref{joint}), we see that $F$ is symmetric with respect to this action.

Next, we show that $F$ is also multilinear:

\begin{proposition}\label{fmulti}  $F: \F^A \to \F$ is a linear combination (over $\F$) of monomials
\begin{equation}\label{mono}
 (x_\alpha)_{\alpha \in A} \mapsto x_{\alpha_1} \ldots x_{\alpha_r}
\end{equation}
where $\alpha_1,\ldots,\alpha_r$ are elements of $A$ whose supports $J_{\alpha_1}, \ldots, J_{\alpha_r}$ partition $\{1,\ldots,s+1\}$.
\end{proposition}

For instance, if $s+1=4$, this proposition asserts that $T(a,b,c,d)$ is a linear combination of expressions such as
$$ B(a,b) B'(c,d)$$
for bilinear forms $B, B' \in \CSM^2(V)$ in the $\CSM$-system, but not expressions such as
$$ B(a,b) B'(a,c)$$
or 
$$ B(a,b)^2 B(c,d).$$

\begin{proof}  Split $A = A_{\ni s+1} \cup A_{\not \ni s+1}$, where $A_{\ni s+1}$ consist of those $\alpha \in A$ whose support $J_\alpha$ contains $s+1$, and $A_{\not \ni s+1} := A \backslash A_{\ni s+1}$.  We split $\F^A = \F^{A_{\ni s+1}} \times \F^{A_{\not \ni s+1}}$ in the obvious manner.  We claim the linearity statement
\begin{equation}\label{lineari}
F( x_{\ni s+1} + y_{\ni s+1}, x_{\not \ni s+1}) = F( x_{\ni s+1}, x_{\not \ni s+1}) + F( y_{\ni s+1}, x_{\not \ni s+1}) 
\end{equation}
whenever $x_{\ni s+1}, y_{\ni s+1} \in \F^{A_{\ni s+1}}$ and $x_{\not \ni s+1} \in \F^{A_{\not \ni s+1}}$.

To prove \eqref{lineari}, it suffices by the linearity of 
$$T = F( (S_\alpha)_{\alpha \in A_{\ni s+1}}, (S_\alpha)_{\alpha \in A_{\not \ni s+1}})$$ 
in the $h_{s+1}$ variable to locate $h_1,\ldots,h_{s},h_{s+1},h'_{s+1} \in V$ such that
\begin{align*}
S_{\ni s+1}( h_1,\ldots, h_{s}, h_{s+1} ) &= x_{\ni s+1} \\
S_{\ni s+1}( h_1,\ldots, h_{s}, h'_{s+1} ) &= y_{\ni s+1} \\
S_{\not \ni s+1}( h_1,\ldots, h_{s}, h_{s+1} ) &= x_{\not \ni s+1} 
\end{align*}
since this implies that
$$
S_{\ni s+1}( h_1,\ldots, h_{s}, h_{s+1} +h'_{s+1}) = x_{\ni s+1} +y_{\ni s+1}$$
and
$$ S_{\not \ni s+1}( h_1,\ldots, h_{s}, h'_{s+1} ) = x_{\not \ni s+1}.$$
But the existence of  $h_1,\ldots,h_{s},h_{s+1},h'_{s+1} \in V$ with these properties follows immediately from Lemma \ref{joint}.

By symmetry, \eqref{lineari} generalises to
\begin{equation}\label{snot}
F( x_{\ni j} + y_{\ni j}, x_{\not \ni j}) = F( x_{\ni j}, x_{\not \ni j}) + F( y_{\ni j}, x_{\not \ni j}) 
\end{equation}
whenever $1 \leq j \leq s+1$, $x_{\ni j}, y_{\ni j} \in \F^{A_{\ni j}}$ and $x_{\not \ni j} \in \F^{A_{\not \ni j}}$, where $A_{\ni j}, A_{\not\ni j}$ are defined analogously to $A_{\ni s+1}, A_{\not \ni s+1}$ by replacing $s+1$ with $j$, and $\F^A$ is identified with $\F^{A_{\ni j}} \times \F^{A_{\not \ni j}}$ in the obvious manner.  We claim that the identities \eqref{snot} imply that $F$ is a linear combination of the monomials \eqref{mono} (note that the converse claim is clear).

To establish this implication, we induct on $s$.  The case $s=0$ is easily verified, so suppose $s \geq 1$ and the claim has already been proven for smaller $s$.  For every $\alpha \in A_{\ni s+1}$, we consider the derivative $\ader_{e_\alpha} F: \F^A \to \F$ of $F$ in the basis direction $e_\alpha$.  From \eqref{snot} (applied to each $j$ in the support of $\alpha$) we see that $\ader_{e_\alpha} F(x)$ is independent of any coefficient $x_\beta$ whose support $J_\beta$ intersects $J_\alpha$, and thus descends to a function $F_\alpha$ on $\F^{A_{J_\alpha^\perp}}$, where $A_{J_\alpha^\perp}$ is the set of $\beta \in A$ whose support lies in $\{1,\ldots,s+1\} \backslash J_\alpha$.  The linearity properties \eqref{snot} for $j \not \in J_\alpha$ descend from $F$ to $F_\alpha$, so by the induction hypothesis (and relabeling) each $F_\alpha$ is a linear combination of monomials \eqref{mono} with $J_{\alpha_1},\ldots,J_{\alpha_r}$ partitioning $\{1,\ldots,k\} \backslash J_\alpha$.  In particular, $x_\alpha F_\alpha$ is a linear combination of monomials of the desired form.  The function $F - \sum_{\alpha \in A_{\ni s+1}} x_\alpha F_\alpha$ is then invariant in the $e_\alpha$ direction for all $\alpha \in A_{\ni s+1}$ and thus vanishes by \eqref{snot} (which implies in particular that $F(0,x_{\not \ni s+1}) = 0$ for all $x_{\not \ni s+1} \in \F^{A_{\not \ni s+1}}$).  The claim follows.
\end{proof}

\begin{remark} Proposition \ref{fmulti} can be viewed as a variant of Corollary \ref{equi-factor2}, but to make this connection precise one would have to generalise the machinery in Appendix \ref{poly-alg} to $\N^s$-filtered groups, as in \cite{gtz}.
\end{remark}

From Proposition \ref{fmulti} we have
$$ F( (x_\alpha)_{\alpha \in A} ) = \sum_{\{\alpha_1,\ldots,\alpha_r\}} c_{\{\alpha_1,\ldots,\alpha_r\}} x_{\alpha_1} \ldots x_{\alpha_r} $$
for some coefficients $c_{\{\alpha_1,\ldots,\alpha_r\}}$, where $\{\alpha_1,\ldots,\alpha_r\}$ ranges over all unordered collections of elements $\alpha_1,\ldots,\alpha_r$ of $A$ whose supports partition $\{1,\ldots,s+1\}$.  The multilinear forms $x_{\alpha_1} \ldots x_{\alpha_r}$ are clearly linearly independent, and so the coefficients $c_{\{\alpha_1,\ldots,\alpha_r\}}$ are uniquely determined by $F$.  In particular, since $F$ is symmetric with respect to the permutation action of $\operatorname{Sym}(\{1,\ldots,s\})$, the coefficients $c_{\{\alpha_1,\ldots,\alpha_r\}}$ must be symmetric also.  Substituting $x_\alpha := S_\alpha$, we now write
$$ T = \sum_{\{\alpha_1,\ldots,\alpha_r\}} c_{\{\alpha_1,\ldots,\alpha_r\}} S_{\alpha_1} \ldots S_{\alpha_r}.$$
From the symmetry of the coefficients $c_{\{\alpha_1,\ldots,\alpha_r\}}$, we may split this sum into orbits of the action of the permutation group, and conclude that $T$ is in fact a linear combination of the basic symmetric monomials
$$ \Sym^{m_1}(S_{k_1,i_1}) \ast \ldots \ast \Sym^{m_l}(S_{k_l,i_l})$$
with $m_1 k_1 + \ldots + m_l k_l = s+1$.  Note that all the $k_1,\ldots,k_l$ are at least $2$, because we have deleted all the linear forms from the $\CSM$-system.  This gives Theorem \ref{GICSME} as required.

\section{A regularity lemma and equidistribution for non-classical polynomials}\label{regsec}

It remains to establish Theorem \ref{er}.  To do this, it is convenient for inductive reasons to establish a technical strengthening of Theorem \ref{er}. We first need an analogue of the notion of a regular $\CSM$-system, but now for polynomials instead of multilinear forms:

\begin{definition}[Regular factor]  Let $V$ be a limit finite-dimensional vector space.  A \emph{factor} is a bounded family $P = (P_{i,j})_{1 \leq i \leq m; 0 \leq j \leq J_i}$ of polynomials $P_{i,j} \in \Poly_{\leq D_i+j(p-1)}(V \to \ultra \T)$, where $m \geq 0$ and $J_1,\ldots,J_m \geq 0$ and $D_1,\ldots,D_m \geq 2$ are standard natural numbers, obeying the relations
\begin{equation}\label{peo}
p P_{i,j} = P_{i,j-1}
\end{equation}
for all $1 \leq i \leq m$ and $0 \leq j \leq J_i$ (with the convention $P_{i,-1} = 0$).  In particular, each $P_{i,j}$ takes values in the $(p^{j+1})^\th$ roots of unity $\frac{1}{p^{j+1}} \Z /\Z$.  We refer to $m$ as the \emph{dimension} of the factor, $D_1,\ldots,D_m$ as the \emph{initial degrees}, and $J_1,\ldots,J_m$ as the \emph{depths}.  The \emph{degree} of the factor is the quantity $\sup_{1 \leq i \leq m} D_i + J_i(p-1)$.

A \emph{depth extension} of $P$ is a factor of the form $P' = (P_{i,j})_{1 \leq i \leq m; 0 \leq j \leq J'_i}$, where $J'_i \geq J_i$ for each $1 \leq i \leq m$, and the polynomials $P_{i,j}$ in $P'$ agree with their counterparts in $P$ for $j \leq J_i$.
If $P'$ is a depth extension of $P$, we call $P$ a \emph{depth retraction} of $P'$.  If $d$ is a standard integer, we define the \emph{degree $\leq d$ depth retraction} $P_{\leq d}$ of $P$ to be the retraction formed by deleting all $P_{i,j}$ with $D_i + J_i(p-1) > d$.

If, for every standard integer $k \geq 2$, the polynomials $P_{i,j}$ with $D_i + j(p-1) = k$ are linearly independent in $\Poly_{\leq k}(V \to \ultra\T)$ modulo $\Poly_{\leq k,\BR}(V \to \ultra \T)$, we say that the factor $P$ is \emph{regular}.

A function $Q: V \to \ultra \T$ is said to be \emph{measurable} with respect to the factor $P$ if one has $Q = F( P_{1,J_1},\ldots,P_{m,J_m} )$ for some function $F: \ultra \T^m \to \ultra \T$. 
\end{definition}

When working exclusively with classical polynomials, one can set all the depths $J_i$ to zero, and the notion of a factor and a regular factor then become essentially the same as those considered in \cite{finrat} (see also \cite{kauf} for a variant of the notion of regularity in the low characteristic case).  However, when trying to regularise non-classical polynomials, one unfortunately needs to consider factors of positive depth, which are more technical to study.

Note that in our definition of a regular factor, the degrees $D_1,\ldots,D_m$ are at least two.  This is because we can eliminate any linear polynomials that arise in the analysis by passing to a finite index subspace.  We will need to eliminate the linear case in order to avoid the $k=p$ case of Corollary \ref{cor}, which is false.

\begin{example}\label{ppp}  Suppose one has three polynomials $P_{1,0} \in \Poly_{\leq D_1}(V \to \frac{1}{p}\Z/\Z)$, $P_{1,1} \in \Poly_{\leq D_1+p-1}(V \to \frac{1}{p^2}\Z/\Z)$, and $P_{2,0} \in \Poly_{\leq D_2+p-1}(V \to \frac{1}{p}\Z/\Z)$ for some natural numbers $D_1,D_2 \geq 2$, with $p P_{1,1} = P_{1,0}$, then $P = (P_{1,0}, P_{1,1}, P_{2,0})$ would be a factor of dimension $2$, initial degrees $D_1,D_2$, depths $1,0$, and degree $\max(D_1+p-1,D_2)$.  One can view this factor as a map $x \mapsto (P_{1,1}(x),P_{2,0}(x))$ from $V$ from $(\frac{1}{p^2}\Z/\Z) \times (\frac{1}{p}\Z/\Z)$ (the polynomial $P_{1,0}$ can be omitted from this map as it is determined by $P_{1,1}$).  

From Lemma \ref{polybasic}(iv), one can find roots $P_{1,2} \in \Poly_{\leq D_1+2(p-1)}(V \to \frac{1}{p^3}\Z/\Z)$ and $P_{2,1} \in \Poly_{\leq D_2+(p-1)}(V \to \frac{1}{p^2}\Z/\Z)$ of $P_{1,1}$ and $P_{2,0}$ respectively (thus $pP_{1,2} = P_{1,1}$ and $pP_{2,1} = P_{2,0}$), then $P' = (P_{1,0}, P_{1,1}, P_{1,2}, P_{2,0}, P_{2,1})$ is a depth extension of the factor $P = (P_{1,0}, P_{1,1}, P_{2,0})$ (or equivalently, $P$ is a depth retraction of $P'$), in which the depths have been increased from $J_1=1,J_2=0$ to $J'_1=2,J'_2=1$.  Of course, one can iterate this procedure and perform depth extensions of $P$ to arbitrary depths.  If $P$ is interpreted as a map from $V$ to $(\frac{1}{p^2}\Z/\Z) \times (\frac{1}{p}\Z/\Z)$, one can view $P'$ as a lift of that map to $(\frac{1}{p^3}\Z/\Z) \times (\frac{1}{p^2}\Z/\Z)$, with the original map factoring through the map $(x,y) \mapsto (px,py)$ from $(\frac{1}{p^3}\Z/\Z) \times (\frac{1}{p^2}\Z/\Z)$ to $(\frac{1}{p^2}\Z/\Z) \times (\frac{1}{p}\Z/\Z)$.

For sake of concreteness, let us now suppose that $D_2=D_1+p-1$.  Then $P$ is regular precisely when the degree $D_1$ polynomial $P_{1,0}$ has unbounded rank, and all non-trivial linear combinations of the degree $D_2=D_1+p-1$ polynomials $P_{1,1}$, $P_{2,0}$ have unbounded rank.  In order for $P'$ to be regular, one must also add the additional property that no non-trivial linear combination of the degree $D_2+p-1=D_1+2(p-1)$ polynomials $P_{1,2}, P_{2,1}$ have unbounded rank.  In this case, $P$ is the degree $\leq D_2$ depth retraction of $P'$.
\end{example}

A convenient property of depth extensions is that they preserve regularity:

\begin{lemma}[Depth extensions preserve regularity]\label{depth-reg}
Let $P = (P_{i,j})_{1 \leq i \leq m; 0 \leq j \leq J_i}$ be a regular factor, and let $P' = (P'_{i,j})_{1 \leq i \leq m'; 0 \leq j \leq J'_i}$ be a depth extension of $P$.  Then $P'$ is also regular.
\end{lemma}

\begin{proof} By induction, it suffices to verify the claim in the case when $J'_i \leq J_i+1$.  

Suppose for contradiction that $P'$ is not regular, then there exists $k \geq 0$ such that the polynomials $P'_{i,j}$ with $D_i + j(p-1) = k$ have a non-trivial linear dependence in $\Poly_{\leq k}(V \to \ultra\T)$ modulo $\Poly_{\leq k,\BR}(V \to \ultra \T)$, thus
$$ \sum_{1 \leq i \leq m; 0 \leq j \leq J_i; D_i + j(p-1) = k} c_{i,j} P'_{i,j} \in \Poly_{\leq k,\BR}(V \to \ultra \T)$$
for some coefficients $c_{i,j} \in \Z$, not all zero.

Suppose that $c_{i,j}$ vanished whenever $j=J_i+1$, then this linear dependence already occured in $P$, implying that $P$ was also not regular, a contradiction.  Thus we have $c_{i,J_i+1} \neq 0$ for at least one $i$; since $D_i \geq 2$, this also forces $k > p$.  (It is here that we crucially need to prevent $D_i$ from being equal to $1$.)  We now multiply the above linear dependence by $p$ using Corollary \ref{cor} and \eqref{peo} to conclude that
$$ \sum_{1 \leq i \leq m; 0 \leq j \leq J_i; D_i + j(p-1) = k} c_{i,j} P'_{i,j-1} \in \Poly_{\leq k-p+1,\BR}(V \to \ultra \T).$$
But this is again a non-trivial linear dependence in $P$, again yielding contradiction.
\end{proof}

We now localise the exact roots property from Theorem \ref{er} to depth extensions of a regular factor.  More precisely, we introduce the following property:

\begin{definition}[Exact roots property]  Let $s$ be a standard integer.  We say that the \emph{exact roots property} $\ER(s)$ holds if, whenevever $P = (P_{i,j})_{1 \leq i \leq m; 0 \leq j \leq J_i}$ is a regular factor on a limit finite-dimensional vector space $V$ of some initial degrees $D_1,\ldots,D_m$, and $Q \in \Poly_{\leq s+1}(V \to \ultra \T)$ is a function of $P$, and $P' = (P_{i,j})_{1 \leq i \leq m; 0 \leq j \leq J'_i}$ is a depth extension of $P$ with $D_i+J'_i(p-1) > s$ for all $1 \leq i \leq m$, then there exists $R \in \Poly_{\leq s+p}(V' \to \ultra \T)$ that is a function of $P'$ such that $pR = Q$ on $V$.  Furthermore, $R$ can be taken to be a linear combination of those $P_{i,j}$ in $P'$ with $D_i + j(p-1) = s+p$, plus a function of those $P_{i,j}$ with $D_i + j(p-1) < s+p$.
\end{definition}

\begin{example} We continue Example \ref{ppp}, again supposing that $D_2=D_1+p-1$ for concreteness.  Assume that $P$ is regular.  Let $s < D_1 + 2(p-1)$ be such that $\ER(s)$ holds, and suppose that we have a polynomial $Q \in \Poly_{\leq s+1}(V \to \ultra \T)$ that is a function of the factor $P$, thus $Q = F( P_{1,1}, P_{2,0})$ for some $F: (\frac{1}{p^2}\Z/\Z) \times (\frac{1}{p}\Z /\Z) \to \ultra \T$.  Then we can find a polynomial $R \in \Poly_{\leq s+p}(V \to \ultra \T)$ that is a root of $Q$ (thus $pR=Q$), which is a function of $P'$ (thus $R = F'(P_{1,2},P_{2,1})$ for some $F': (\frac{1}{p^3}\Z/\Z) \times (\frac{1}{p^2}\Z \times \Z) \to \ultra \T$).  If $s$ was equal to $D_1+p-2$, then $R$ would be a linear combination of $P_{1,2}, P_{2,1}$ plus a functin $F''(P_{1,1},P_{2,0})$ of the lower degree polynomials $P_{1,2}, P_{2,1}$.
\end{example}

One consequence of the exact roots property is that it allows for a regularity lemma:

\begin{lemma}[Regularity lemma]\label{rl}  Let $s_0 \geq 1$ be a standard integer such that $\ER(s)$ and $\GIP(s)$ hold for all $0 \leq s < s_0$.  Suppose that $P = (P_1,\ldots,P_m)$ are a bounded tuple of polynomials $P_1,\ldots,P_m \in \Poly_{\leq s_0}(V \to \ultra \T)$ on a limit finite-dimensional vector space $V$.  Then there exists a bounded index subspace $V'$ of $V$ and a regular factor $Q = (Q_{i,j})_{1 \leq i \leq m; 0 \leq j \leq J_i}$ on $V'$ of degree at most $s_0$ such that all the $P_i$ are measurable with respect to $Q$ on $V'$.  Furthermore, on $V'$, if $P_i$ has degree $d_i$, then $P_i$ is a linear combination (over $\Z$) of those polynomials $Q_{i,j}$ with $D_i + j(p-1)=d_i$, plus a function of the depth retraction $Q_{\leq d_i-1}$ of $Q$.
\end{lemma}

\begin{proof}  We induct on $s_0$.  The claim is trivial for $s_0 \leq 1$ (note that any polynomial of degree $1$ can be made constant by passing to a bounded index subspace), so suppose that $s_0 \geq 2$ and that the claim has already been proven for smaller $s_0$.  Observe from Lemma \ref{polybasic} that $p \Poly_{\leq s_0}(V \to \ultra \T) \subset \Poly_{\leq s_0,\BR}(V \to \ultra \T)$, thus the abelian group $\Poly_{\leq s_0}(V \to \ultra \T)/\Poly_{\leq s_0,\BR}(V \to \ultra \T)$ is in fact a vector space.  This vector space contains $\Poly_{\leq s_0}(V \to \iota(\F))/\Poly_{\leq s_0,\BR}(V \to \iota(\F))$ as a subspace.

As every finitely generated vector space has a finite basis, we may thus represent $P_1,\ldots,P_m$ as a linear combination (over $\Z$) of a bounded number of polynomials $P'_1,\ldots,P'_{m'} \in \Poly_{\leq s_0}(V \to \ultra \T)$ that are linearly independent modulo $\Poly_{\leq s_0,\BR}(V \to \ultra \T) + \Poly_{\leq s_0}(V \to \iota(\F))$, a bounded number of classical polynomials $R_1,\ldots,R_{m''} \in \Poly_{\leq s_0}(V \to \iota(\F))$ that are linearly independent modulo $\Poly_{\leq s_0,\BR}(V \to \iota(\F))$, and a bounded number of bounded rank polynomials $S_1,\ldots,S_{m'''} \in \Poly_{\leq s_0, \BR}(V \to \ultra \T)$.   By definition of $\Poly_{\leq s_0, \BR}(V \to \ultra \T)$, the $S_1,\ldots,S_{m'''}$ are in turn functions of a bounded number of polynomials $S'_1,\ldots,S'_{m''''} \in \Poly_{\leq s_0-1}(V \to \ultra \T)$.  

By Lemma \ref{polybasic}, the polynomials $pP'_1,\ldots,pP'_{m'}$ have degree at most $\max(s_0-p+1,0)$, and in particular have degree $\leq s_0-1$.
Applying the induction hypothesis, and passing to a bounded index subspace $V'$ of $V$, the polynomials $S'_1,\ldots,S'_{m''''}, pP'_1,\ldots,pP'_{m'}$, when restricted to $V'$, are then all functions of a single regular factor $Q = (Q_{i,j})_{1 \leq i \leq m'''''; 0 \leq j \leq J_i}$ of degree $\leq s_0-1$ with initial degrees $D_1,\ldots,D_{m'''''} \geq 2$.  

Henceforth all polynomials will be understood to be restricted to $V'$.  Using Lemma \ref{subspace}, we see that the various linear independence properties on $P'_1,\ldots,P'_{m'}$ and $R_1,\ldots,R_{m''}$ descend from $V$ to $V'$.

By using Lemma \ref{polybasic}(v) to perform depth extensions on $Q$ as necessary, we may assume that 
$$ s_0-p < D_i + J_i(p-1) \leq s_0-1$$
for all $i=1,\ldots,m$; of course, this keeps the degree of $Q$ to be $\leq s_0-1$.  Note that $Q$ also remains regular, thanks to Lemma \ref{depth-reg}.

We now perform one further depth extension to $Q$ to obtain $Q' := (Q_{i,j})_{1 \leq i \leq m'''''; 0 \leq j \leq J_i+1}$ by choosing $Q_{i,J_i+1}$ to be a polynomial in $\Poly_{\leq D_i + (J_i+1)(p-1)}(V' \to \ultra \T)$ obeying $p Q_{i,J_i+1} = Q_{i,J_i}$; such a polynomial is available thanks to Lemma \ref{polybasic}(v).  This is clearly a factor.  By Lemma \ref{depth-reg}, $Q'$ is also regular.

Note that when constructing $Q'$, we have the freedom to modify each $Q_{i,J_i+1}$ additively by a classical polynomial from $\Poly_{\leq D_i + (J_i+1)(p-1)}(V' \to \iota(\F))$; this freedom will be important later on.

Now consider the polynomials $pP'_1,\ldots,pP'_{m'}$.  By Lemma \ref{polybasic}, they have degree $\leq \max(s_0-p+1,0)$, and are also functions of $Q$.  Applying the hypothesis $\ER(s_0-p)$ (and refining $V'$ if necessary), we conclude that we can find polynomials $U_1,\ldots,U_{m'} \in \Poly_{\leq s_0}(V' \to \ultra \T)$ that are functions of $Q'$ and such that
$$ p P'_l = p U_l$$
for all $1 \leq l \leq m'$.  In other words, $P'_l$ and $U_l$ differ by a classical polynomial in $\Poly_{\leq s_0}(V' \to \iota(\F))$.  (Note that these claims are trivial when $s_0 \leq p-1$; the hypothesis $\ER(s_0-p)$ is only needed when $s_0 > p-1$). Furthermore,  each $U_l$ is a linear combination of those $Q_{i,J_i+1}$ with $D_i + (J_i+1)(p-1) = s_0$, plus a function of $Q'_{\leq s_0-1}$, thus
\begin{equation}\label{pil}
 U_l = \sum_{i \in A} c_{l,i} Q_{i,J_i+1} + E_l
\end{equation}
for some coefficients $c_{l,i} \in \Z$, where $E_l$ is a function of $Q_{\leq s_0-1}$, and $A$ is the set of those $1 \leq i \leq m'''''$ with $D_i + (J_i+1)(p-1) = s_0$.  We claim that the vectors $\vec c_l := (c_{l,i} \mod p)_{i \in A} \in \F^A$ for $l=1,\ldots,m'$ are linearly independent.  Indeed, suppose for contradiction that we had a non-trivial linear dependence 
$$ a_1 \vec c_1 + \ldots + a_{m'} \vec c_{m'} = 0$$ 
in $\F^A$ for some $a_1,\ldots,a_{m'} \in \{0,\ldots,p-1\}$, not all zero.  Then by \eqref{pil}, the polynomial
$$ a_1 U_1 + \ldots + a_{m'} U_{m'}$$
is a function of $Q_{\leq s_0-1}$, and of degree $\leq s_0$.  As the $U_l$ differ from $P'_l$ by an element of
$\Poly_{\leq s_0}(V' \to \iota(\F))$, we conclude that the $P'_1,\ldots,P'_{m'}$ are linearly dependent modulo $\Poly_{\leq s_0, \BR}(V' \to \ultra \T) + \Poly_{\leq s_0}(V' \to \iota(\F))$, contradicting the construction of $P'_1,\ldots,P'_{m'}$.  Thus the $\vec c_1,\ldots,\vec c_{m'}$ are linearly independent.

We would like to modify $Q'$ so that the $U_l$ will agree with $P_l'$ exactly. To do this, 
recall that we had the freedom to modify each of the $Q_{i,J_i+1}$ by an arbitrary classical polynomial in $\Poly_{\leq s_0}(V' \to \iota(\F))$; this modifies the $U_l$ by a corresponding classical polynomial in $\Poly_{\leq s_0}(V' \to \iota(\F))$.  Because the $\vec c_1,\ldots,\vec c_{m'}$ are linearly independent, the $U_l$ can be so modified independently.  Since $P'_l$ and $U_l$ already only differed by such a classical polynomial, we can thus modify each of the $Q_{i,J_i+1}$ so that the $U_l$ are equal to $P'_l$ simultaneously for all $1 \leq i \leq m'$.  Having done so, we now see that the $P'_l$ are functions of $Q'_{\leq s_0}$.

Now we extend $Q'_{\leq s_0}$ to a further factor $Q''$ by adjoining the classical polynomials $R_1,\ldots,R_{m''}$ as new dimensions of degree $s_0$ and depth $0$.  This is still a factor; we claim that it remains regular.  To see this, we need to show that the $R_1,\ldots,R_{m''}$, together with the $Q_{i,j}$ with $D_i + j(p-1)=s_0$, are linearly independent modulo $\Poly_{\leq s_0,\BR}(V' \to \ultra \T)$.  Note from construction that all such $Q_{i,j}$ must be of the form $Q_{i,J_i+1}$.  Suppose for contradiction that there was a non-trivial linear dependence.  As the $R_1,\ldots,R_{m''}$ were already linearly independent modulo $\Poly_{\leq s_0,\BR}(V \to \ultra \T)$, this dependence must involve at least one of the $Q_{i,J_i+1}$.  Multiplying by $p$ and using Corollary \ref{cor}, we see that there is a dependence among those $Q_{i,J_i}$ with $D_i + (J_i+1)(p-1)=s_0$ modulo $\Poly_{\leq s_0-p+1, \BR}(V' \to \ultra \T)$, but this contradicts the regularity of $Q$.

By construction, all the polynomials $P'_1,\ldots,P'_{m'},R_1,\ldots,R_{m''},S_1,\ldots,S_{m'''}$ are functions of the regular factor $Q''$, which has degree $\leq s_0$, and so $P_1,\ldots,P_m$ are functions of $Q''$ also.  A careful inspection of the above argument also shows that each $P_i$ was in fact an integer linear combination of those $Q''_{i,j}$ in $Q''$ with $D_i + j(p-1)=d_i$, plus a function of $Q''_{\leq d_i-1}$ (the cases $d_i=s_0$ and $d_i<s_0$ have to be treated separately).  The claim follows.
\end{proof}

In view of the regularity lemma, Theorem \ref{er} now follows from

\begin{theorem}[Exact roots, technical version]\label{ertech}  Let $s_0 \geq 1$  be such that $\ER(s_0-1)$ holds, and $\GIP(s)$ holds for all $s \leq s_0$.  Then $\ER(s_0)$ holds.
\end{theorem}

Indeed, assuming Theorem \ref{ertech}, then in the situation in Theorem \ref{er}, we have $\ER(s')$ for all $0 \leq s' \leq s$ by strong induction (the case $\ER(0)$ being trivial).  If $P \in \Poly_{\leq s'+1, \BR}(V \to \ultra \T)$ for some $0 \leq s' \leq s$, then by 
Lemma \ref{rl}, $P$ is a function of a regular factor $R$ of degree $\leq s'$.  We then create a depth extension $R'$ of $R$ by extending all the depths $J'_i$ so that $s' < D_i + J'_i(p-1) \leq s'+p-1$; the existence of such a depth extension is guaranteed by Lemma \ref{polybasic}(v).  Applying $\ER(s)$, we conclude that we can find $Q \in \Poly_{\leq s'+p}(V \to \ultra \T)$ that is a function of $R'$ such that $pQ = P$.  Since $R'$ has degree $\leq s'+p-1$, $Q$ has bounded rank, and Theorem \ref{er} follows.

It remains to establish Theorem \ref{ertech}.   To do this, we first require an equidistribution lemma, analogous to (but more complicated than) Lemma \ref{joint}.

Let $P = (P_{i,j})_{1 \leq i \leq m; 0 \leq j \leq J_i}$ be a factor with initial degrees $D_1,\ldots,D_m$.  By \eqref{peo}, the polynomials in $P$ are in fact functions of $\tilde P := (P_{i,J_i})_{1 \leq i \leq m}$, which we interpret as a map from $V$ to $\ultra \T^m$.  

It is now convenient to use the machinery of polynomial algebra, which is reviewed in Section \ref{poly-alg}.  We define a filtration $(\ultra \T^m)_\N$ on $\ultra \T^m$ by defining $(\ultra \T^m)_k$ for each standard natural number $k$ to be the (finite) group generated by the elements $\frac{1}{p^{J_i-j+1}} e_i \mod \ultra \Z^m$ for which $1 \leq i \leq m$, $0 \leq j \leq J_i$, and $k \leq D_i + j(p-1)$, where $e_1,\ldots,e_m$ is the standard basis of $\R^m$.  This is easily seen to be a filtration (see Example \ref{filt-ab}).  We call this the \emph{filtration with depths $J_1,\ldots,J_m$ and initial degrees $D_1,\ldots,D_m$}.

\begin{lemma}[Equidistribution of regular factors]\label{equil}  Suppose that $\GIP(s)$ is true for all $0 \leq s \leq s_0$, let $P = (P_{i,j})_{1 \leq i \leq m; 0 \leq j \leq J_i}$ be a factor of degree $\leq s_0$, and let $k \geq 0$ be a standard integer.    Set $\tilde P := (P_{i,J_i})_{1 \leq i \leq m}$, and let $\ultra \T^m$ be given the filtration with depths $J_1,\ldots,J_m$ and initial degrees $D_1,\ldots,D_m$.  We give $V$ the maximal degree $\leq 1$ filtration (see Example \ref{filt-ab}).
\begin{itemize}
\item[(i)] $\tilde P: V \to \ultra \T^m$ is a polynomial map (see Definition \ref{diffeo}).
\item[(ii)] If furthermore $P$ is regular, then $\tilde P: V \to \ultra \T^m$ is equidistributed on cubes (see Definition \ref{strong-equi}).
\end{itemize}
\end{lemma}

\begin{proof}  To prove (i), it suffices by Definition \ref{diffeo} to verify that
\begin{equation}\label{aorta}
 \ader_{h_1} \ldots \ader_{h_k} \tilde P(x) \in (\ultra \T^m)_k
\end{equation}
whenever $k \in \N$ and $h_1,\ldots,h_k, x \in V$.  But if $1 \leq i \leq m$, then we clearly have
$$ \ader_{h_1} \ldots \ader_{h_k} P_{i,J_i}(x) \in \frac{1}{p^{J_i+1}} \Z / \Z$$
since $P_{i,J_i}$ takes values in $\frac{1}{p^{J_i+1}} \Z / \Z$; and if
$ D_i + j(p-1) < k \le  D_i + (j+1)(p-1)$ for some $0 \leq j \leq J_{i}$ then
$$ \ader_{h_1} \ldots \ader_{h_k} P_{i,j}(x) = 0$$
and since $p^{J_i-j} P_{i,J_i} =P_{i,j}$ we get 
$$ \ader_{h_1} \ldots \ader_{h_k} P_{i,J_i}(x) \in \frac{1}{p^{J_i-j}} \Z / \Z.$$
Comparing this with the definition of $(\ultra \T^m)_k$ we obtain \eqref{aorta}.

Now we verify (ii), which is trickier.  We need to show that the map $\HK^k(\tilde P): \HK^k(V) \to \HK^k(\ultra \T^m)$ is equidistributed.  By the Weyl equidistribution criterion (Lemma \ref{weyl-equi}), it suffices to show that
\begin{equation}\label{eep}
 \E_{x,h_1,\ldots,h_k \in V} e(\eta(\HK^k(\tilde P)(x,h_1,\ldots,h_k))) = o(1)
\end{equation}
whenever $\eta: \HK^k(\ultra \T^m) \to \T$ is a non-zero homomorphism.  

Observe that $\HK^k(\ultra \T^m)$ is a (finite) subgroup of $(\T^m)^{\{0,1\}^k}$. By Pontryagin duality, $\eta$ must therefore be the restriction of a homomorphism from  $(\T^m)^{\{0,1\}^k}$ to $\T$.  In other words, we can find (non-unique) integers $c_{i,\omega}$ for $1 \leq i \leq m$ and $\omega \in \{0,1\}^m$ such that
\begin{equation}\label{cio}
 \eta( (t_{i,\omega})_{1 \leq i\leq m; \omega \in \{0,1\}^k} ) = \sum_{i=1}^m \sum_{\omega \in \{0,1\}^k} c_{i,\omega} t_{i,\omega}
\end{equation}
for all $(t_{i,\omega})_{1 \leq i\leq m; \omega \in \{0,1\}^k}$ in $\HK^k(\ultra \T^m)$.  In particular, the left-hand side of \eqref{eep} becomes
\begin{equation}\label{eep-2}
 \E_{x,h_1,\ldots,h_k \in V} e(\sum_{i=1}^m \sum_{\omega \in \{0,1\}^k} c_{i,\omega} P_{i,J_i}( x + \omega_1 h_1 +\ldots + \omega_k h_k )).
\end{equation}

The coordinates $t_{i,\omega} \in \ultra \T$ of points $(t_{i,\omega})_{1 \leq i\leq m; \omega \in \{0,1\}^k}$ in $\HK^k(\ultra \T^m)$ obey a number of constraints.  Firstly, for each $\omega$, $(t_{i,\omega})_{1 \leq i \leq m}$ must lie in $(\ultra \T^m)_0$, or in other words we have
\begin{equation}\label{pj1}
 p^{J_i+1} t_{i,\omega} = 0
\end{equation}
whenever $1 \leq i \leq m$ and $\omega \in \{0,1\}^k$.  Secondly, from Proposition \ref{hk-dest} we see that
\begin{equation}\label{piji}
 p^{J_i-j} \sum_{\omega \in F} (-1)^{|\omega|} t_{i,\omega} = 0
\end{equation}
whenever $1 \leq i \leq m$, $\omega \in \{0,1\}^k$, $0 \leq j \leq J_i$, and $F$ is a face in $\{0,1\}^k$ of dimension greater than $D_i + j(p-1)$.  (In fact, Proposition \ref{hk-dest} asserts that these are the \emph{only} constraints on the $t_{i,\omega}$.)

We can use these constraints to place the coefficients $c_{i,\omega}$ in a ``reduced form'', as follows.  First observe from \eqref{piji} that if there exist $1 \leq i \leq m$, $\omega \in \{0,1\}^k$, and $0 \leq j \leq J_i$ with $|\omega| > D_i + j(p-1)$ and $c_{i,\omega} \geq p^{J_i-j}$ or $c_{i,\omega} < 0$ then by adding a suitable multiple of \eqref{piji} to \eqref{cio}, one can place $c_{i,\omega}$ in the interval $\{0,\ldots,p^{J_i-j}-1\}$, at the expense of changing the values of $c_{i,\omega'}$ for various $\omega'$ with $|\omega'| < |\omega|$.  Iterating this procedure (starting with those $\omega$ with large values of $|\omega|$ and then working downward) we may assume without loss of generality that 
\begin{equation}\label{down}
0 \leq c_{i,\omega} < p^{J_i-j}
\end{equation}
whenever $1 \leq i \leq m$, $\omega \in \{0,1\}^k$, and $0 \leq j \leq J_i$ is such that $|\omega| > D_i + j(p-1)$.  

In a similar spirit, by using \eqref{pj1}, we may assume that
$$ 0 \leq c_{i,\omega} < p^{J_i+1}$$
for all $1 \leq i \leq m$ and $\omega \in \{0,1\}^k$.

Since $\eta$ is non-zero, at least one of the $c_{i,\omega}$ is non-zero.  Let $\omega_* \in \{0,1\}^k$ be such that $c_{i,\omega_*}$ is non-zero for at least one $1 \leq i \leq m$, and such that $|\omega_*|$ is maximal with respect to this property.  By permutation symmetry we may assume that $\omega_* = 1^K 0^{k-K} = (1,\ldots,1,0,\ldots,0)$ for some $0 \leq K \leq k$.  The expression \eqref{eep-2} can then be factored as
$$
 \E_{x,h_1,\ldots,h_k \in V} e(Q(x+h_1+\ldots+h_K)) \prod_{l=1}^K f_l(x,h_1,\ldots,h_k)$$
where $Q: V \to \ultra \T$ is the expression
$$ Q := \sum_{i=1}^m c_{i,\omega_*} P_{i,J_i}$$
and each $f_l: V^k \to \ultra \C$ is a limit function bounded in magnitude by $1$ and independent of $h_l$.  Using the second Cauchy-Schwarz-Gowers inequality (Lemma \ref{gow}(v)), it thus suffices to show that
$$ \|e(Q)\|_{U^K(V)} = o(1).$$

Let $I := \{ 1 \leq i \leq m: c_{i,\omega_*} \neq 0 \}$, then $I$ is non-empty.
For each $i \in I$, let $0 \leq j_i \leq J_i$ be the least integer $j$ such that $p^{J_i-j} | c_{i,\omega_*}$.  Since $p^{J_i-j_i} P_{i,J_i} = P_{i,j_i}$, we thus have
\begin{equation}\label{qform}
 Q = \sum_{i \in I} a_i P_{i,j_i}
\end{equation}
for some integers $a_i$ that are not divisible by $p$.

Let
$$ D := \sup_{i \in I} D_i + j_i(p-1).$$
As $P$ has degree $\leq s_0$, we have $D \leq s_0$.  Also, since each $P_{i,j_i}$ has degree $D_i + j_i(p-1)$, we see that $Q$ has degree $\leq D$.

If one had
$$ D_i+j_i(p-1) < |\omega_*| = K$$
for some $i \in I$, then from \eqref{down} we would have $c_{i,\omega_*} < p^{J_i-j}$, a contradiction; so we must have $D_i+j_i(p-1) \geq K$ for all $i \in I$.  In particular, $D \geq K$.  By the monotonicity of the Gowers norms (Lemma \ref{gow}(ii)), it thus suffices to show that
$$ \|e(Q)\|_{U^D(V)} = o(1).$$
Applying the induction hypothesis $\GIP(D-1)$, it thus suffices to show that $Q \not \in \Poly_{\leq D,\BR}(V \to \ultra \T)$.  But, as $P$ is regular, the polynomials $P_{i,j_i}$ with $D_i+j_i(p-1)=D$ are linearly independent over $\F$ modulo $\Poly_{\leq D,\BR}(V \to \ultra \T)$, and the polynomials $P_{i,j_i}$ with $D_i+j_i(p-1)<D$ already lie in $\Poly_{\leq D,\BR}(V \to \ultra \T)$.  Since there is at least $i \in I$ with $D_i+j_i(p-1)=D$, and all coefficients $a_i$ in \eqref{qform} are nonzero modulo $p$, we obtain the desired claim.
\end{proof}

From the above corollary and Corollary \ref{equi-factor2} we conclude

\begin{corollary}[Polynomials on a regular factor]\label{klack}  Suppose that $\GIP(s)$ is true for all $0 \leq s \leq s_0$, let $P = (P_{i,j})_{1 \leq i \leq m; 0 \leq j \leq J_i}$ be a regular factor of degree $\leq s_0$, let $\tilde P: V \to \ultra \T^m$ be the associated map $\tilde P = (P_{i,J_i})_{1 \leq i \leq m}$, let $d \geq 0$ be an integer, and let $f: \ultra \T^m \to \ultra \T$ be a function.  Then the following are equivalent:
\begin{itemize}
\item $f(\tilde P): V \to \ultra \T$ is a polynomial of degree $\leq d$.
\item $f: \ultra \T^m \to \ultra \T$ is a polynomial map (where we give $\ultra \T^m$ the filtration with depths $J_1,\ldots,J_m$ and initial degrees $D_1,\ldots,D_m$, and $\ultra \T$ the maximal degree $\leq d$ filtration, see Example \ref{filt-ab}).
\end{itemize}
\end{corollary}

We now lift the conclusion of Corollary \ref{klack} from $\ultra \T^m$ to $\Z^m$.  Given natural numbers $D_1,\ldots,D_m$, we define the \emph{filtration $(\Z^m)_\N$ of initial degrees $D_1,\ldots,D_m$} on $\Z^m$ by setting $(\Z^m)_k$, for each natural number $k$, to be the subgroup of $\Z^m$ generated by those elements $p^j e_i$ for which $1 \leq i \leq m$, $j \in \N$, and $k \leq D_i + j(p-1)$; this is easily seen to be a filtration.

\begin{corollary}[Polynomials on a regular factor, again]\label{klack2}  Suppose that $\GIP(s)$ is true for all $0 \leq s \leq s_0$, let $P=(P_{i,j})_{1 \leq i \leq m; 0 \leq j \leq J_i}$ be a regular factor of degree $\leq s_0$ and initial degrees $D_1,\ldots,D_m$, let $d \geq 0$ be an integer.  Let $Q: V \to \ultra \T$ be a limit function.  Then the following are equivalent:
\begin{itemize}
\item $Q$ is measurable with respect to $P$, and is a polynomial of degree $\leq d$.  
\item There exists a polynomial map $\tilde f: \Z^m \to \ultra \T$ from $\Z^m$ (with the filtration of initial degrees $D_1,\ldots,D_m$) to $\ultra \T$ (with the maximal degree $\leq d$ filtration), such that $\tilde f$ is periodic with period $p^{J_i+1} e_i$ for each $1 \leq i \leq m$, and such that one has
\begin{equation}\label{qlift}
 Q(x) = \tilde f( a_1,\ldots,a_m ),
\end{equation}
whenever $x \in V$ and $a_1,\ldots,a_m \in \Z$ are such that
\begin{equation}\label{qlift-2}
P_{i,J_i}(x) = \frac{a_i}{p^{J_i+1}}\mod 1,
\end{equation}
for all $1 \leq i \leq m$.
\end{itemize}
\end{corollary}

\begin{proof}  If $Q$ is measurable with respect to $P$ and is a polynomial of degree $\leq d$, then by Corollary \ref{klack}, we can write $Q = f(\tilde P)$, where $f: \ultra \T^m \to \ultra \T$ is a polynomial map from $\ultra \T^m$ (with the filtration of initial degrees $D_1,\ldots,D_m$ and depths $J_1,\ldots,J_m$) to $\ultra \T$ (with the maximal degree $\leq d$ filtration).  

Let $\phi: \Z^m \to \ultra \T^m$ be the map
$$ \phi( a_1,\ldots,a_m) := ( \frac{a_1}{p^{J_1+1}} \mod 1, \ldots, \frac{a_m}{p^{J_m+1}} \mod 1 ).$$
One easily verifies that this is a polynomial map (indeed, it is a filtered homomorphism) from $\Z^m$ (with the filtration of initial degrees $D_1,\ldots,D_m$) to $\ultra \T^m$ (with the filtration of initial degrees $D_1,\ldots,D_m$ and depths $J_1,\ldots,J_m$).  Thus the function $\tilde f: \Z^m \to \ultra \T$ defined by $\tilde f := f \circ \phi$ is also a polynomial map from $\Z^m$ (with the filtration of initial degrees $D_1,\ldots,D_m$) to $\ultra \T$ (with the maximal degree $\leq d$ filtration).  It is also periodic with period $p^{J_i+1} e_i$ for each $1 \leq i \leq m$, because $\phi$ is also periodic with these periods.  By construction one also has \eqref{qlift} whenever \eqref{qlift-2}.  This proves one implication of the corollary.  The other implication follows by reversing the above argument (noting that $\phi$ is weakly equidistributed on cubes in the sense of Definition \ref{weak-equi}, so that one can apply Lemma \ref{equi-factor}).
\end{proof}

In view of the above corollary, Theorem \ref{ertech} can now be deduced from an analogous result on the integer lattice $\Z^m$, which we formulate precisely as follows:

\begin{proposition}[Exact roots in $\Z^m$]\label{erz} Let $m \geq 0$, $D_1,\ldots,D_m \geq 1$, and $d \geq 0$ be standard natural numbers.
Let $\tilde f: \Z^m \to \ultra \T$ be a polynomial map from $\Z^m$ (with the filtration of initial degrees $D_1,\ldots,D_m$) to $\ultra \T$ (with the maximal degree $\leq d$ filtration).  Then we can find a polynomial map $\tilde g: \Z^m \to \ultra \T$ from $\Z^m$ (with the filtration of initial degrees $D_1,\ldots,D_m$) to $\T$ (with the maximal degree $\leq d+p-1$ filtration) such that $p\tilde g = \tilde f$.  Furthermore, $\tilde g$ is a linear combination (over $\Z$) of the functions $(a_1,\ldots,a_m) \mapsto \frac{a_i}{p^{j+1}} \mod 1$ with $1 \leq i \leq m$ and $D_i + j(p-1) = d+p-1$, plus a function of the expressions $a_i \mod p^{j+1}$ with $1 \leq i \leq m$ and $D_i + j(p-1) < d+p-1$.  In particular, $\tilde g$ is periodic with period $p^{j+1} e_i$ whenever $1 \leq i \leq m$ and $D_i + j(p-1) > d$.
\end{proposition}

We remark that in contrast with the previous arguments, the above proposition holds even when the $D_i$ are equal to $1$ (indeed, Lemma \ref{polybasic}(iii) can be viewed as a special case in which $D_1=\ldots=D_m=1$ and the functions are periodic with period $p\Z^m$).

We will prove this Proposition in the next section.  For now, we show how Proposition \ref{erz} implies Theorem \ref{ertech} and thus Theorem \ref{er}.

\begin{proof}[Proof of Theorem \ref{ertech} assuming Proposition \ref{erz}]  Let $s_0 \geq 1$  be such that $\ER(s_0-1)$ holds, and $\GIP(s)$ holds for all $s \leq s_0$.  Let $P$ be a regular factor of degree at most $s_0$ and some initial degrees $D_1,\ldots,D_m$, let $Q \in \Poly_{\leq s_0+1}(V \to \ultra \T)$ be a function of $P$, and let $P' = (P_{i,j})_{1 \leq i \leq m; 0 \leq j \leq J'_i}$ be a depth extension of $P$ with $D_i+J'_i(p-1) > s_0$ for all $1 \leq i \leq m$.   Our objective is to find a polynomial $R \in \Poly_{\leq s+p}(V \to \ultra \T)$ that is a function of $P'$ such that $pR = Q$.

Note from Lemma \ref{depth-reg} that $P'$ is automatically regular.
  
By Corollary \ref{klack2}, we can find a function $\tilde f: \Z^m \to \ultra \T$ of weighted degree $\leq s_0+1$ which is periodic with period $p^{J_i+1} e_i$ for each $i=1,\ldots,m$, such that one has
$$ Q(x) = \tilde f( a_1,\ldots,a_m )$$
whenever $x \in V$ and $a_1,\ldots,a_m \in \Z$ are such that
$$ P_{i,J_i}(x) = \frac{a_i}{p^{J_i+1}} \mod 1$$
for all $1 \leq i \leq m$.

Applying Proposition \ref{erz}, one can find a function $\tilde g: \Z^m \to \ultra \T$ of weighted degree $\leq s_0+p$ which is periodic with period $p^{j+1} e_i$ whenever $1 \leq i \leq m$ and $D_i + j(p-1) > s_0+1$, such that $p \tilde g = \tilde f$.  In particular, $\tilde g$ is periodic with period $p^{J'_i+1} e_i$ for each $1 \leq i \leq m$.  We may therefore define the function $R: V \to \ultra \T$ by setting
$$ R(x) := \tilde g(a_1,\ldots,a_m)$$
whenever $x \in V$ and $a_1,\ldots,a_m \in \Z$ are such that
$$ P_{i,J'_i}(x) = \frac{a_i}{p^{J'_i+1}} \mod 1$$
for all $1 \leq i \leq m$; the periodicity properties of $\tilde g$ ensure that $R$ is well-defined.  By Corollary \ref{klack2}, $R$ is a function of $P'$ which is a polynomial of degree $\leq s_0+p$. By construction, one has $pR=Q$, and the claim follows.
\end{proof}

The only remaining task is to establish Proposition \ref{erz}.  This will be the subject of the next section.

\section{Exact roots for polynomials on $\Z^m$}\label{erpsec}

We now prove Proposition \ref{erz}.  Throughout this section the dimension $m\geq 0$ and initial degrees $D_1,\ldots,D_m \geq 1$ are fixed.

It is convenient to rephrase the polynomiality condition in terms of derivatives.  Define a \emph{basic generator} $v$ to be an element of $\Z^m$ of the form $p^j e_i$, where $1 \leq i \leq m$ and $j \geq 0$.  Define a \emph{multigenerator} to be a tuple $\vec v = (v_1,\ldots,v_r)$ of basic generators, where $r \geq 0$ is a standard natural number.  We associate to each basic generator $p^j e_i$ a \emph{weighted degree} $\deg(p^j e_i) := D_i + j(p-1)$, and associate to each multigenerator $\vec v = (v_1,\ldots,v_r)$ a weighted degree $\deg(\vec v) := \deg(v_1) + \ldots + \deg(v_r)$.  We also associate to $\vec v$ the differential operator
$$ \ader_{\vec v} := \ader_{v_1} \ldots \ader_{v_r}.$$
We say that a function $\tilde f: \Z^m \to \ultra \T$ has \emph{weighted degree} $\leq d$ if
one has $\ader_{\vec v} \tilde f = 0$ whenever $\deg(\vec v) > d$.  In other words, we have
$$ (\prod_{i=1}^m \prod_{j=0}^\infty \ader_{p^j e_i}^{a_{i,j}}) \tilde f = 0$$
whenever $a_{i,j}$ are natural numbers (at most finitely many of which are non-zero) with $\sum_{i=1}^m \sum_{j=0}^\infty a_{i,j}(D_i+j(p-1)) > d$.

From Proposition \ref{polygen} we have

\begin{proposition}[Differential characterisation of polynomiality]\label{diff}  Let $\tilde f: \Z^m \to \ultra \T$ be a function, and let $d \geq 0$ be a standard natural number.  Then the following are equivalent:
\begin{itemize}
\item[(i)] $\tilde f$ is a polynomial map from $\Z^m$ (with the filtration of initial degrees $D_1,\ldots,D_m$) to $\ultra \T$ (with the maximal degree $\leq d$ filtration).
\item[(ii)] $\tilde f$ has weighted degree $\leq d$.
\end{itemize}
\end{proposition}

One nice feature of a weighted  bounded degree polynomials is that they have some periodicity properties:

\begin{lemma}[Periodicity properties]\label{wt} Let $\tilde f: \Z^m \to \ultra \T$ be of weighted degree $\leq d$.  
\begin{itemize}
\item[(i)] For any $1 \leq i \leq m$, $\tilde f$ is periodic with period $p^j e_i$ whenever $1 \leq i \leq m$ and $j \geq 0$ are such that $D_i + j(p-1) > d$.
\item[(ii)]  $\tilde f$ is a linear combination (over $\Z$) of the functions $(a_1,\ldots,a_m) \mapsto \frac{a_i}{p^{j+1}} \mod 1$ with $1 \leq i \leq m$ and $D_i + j(p-1) = d$, plus a function of the expressions $a_i \mod p^{j+1}$ with $1 \leq i \leq m$ and $D_i + j(p-1) < d$. 
\end{itemize}
\end{lemma}

\begin{proof} If $D_i+j(p-1) > d$ then $\ader_{p^j e_i} \tilde f = 0$ by Proposition \ref{diff}.  This proves (i).

Now we turn to (ii).  Let $I$ denote the set of all $1 \leq i \leq m$ for which there is a natural number $j_i$ for which $D_i+j_i(p-1)=d$.  If $i \in I$, then $\ader_{p^{j_i} e_i} \tilde f$ has weighted degree $\leq 0$ and is thus constant.  On the other hand, by (i), $\tilde f$ is periodic with period $p^{j_i+1} e_i$.  We conclude that $\ader_{p^{j_i} e_i} \tilde f = c_i/p$ for some $c_i \in \{0,\ldots,p-1\}$.  Write $\tilde g := \sum_{i \in I} c_i \frac{a_i}{p^{j_i+1}} \mod 1$, then we see that $\tilde g$ is periodic with period $p^{j_i} e_i$ for each $i \in I$, and thus (by (i)), is also periodic with periodic $p^j e_i$ whenever $1 \leq i \leq m$ and $D_i + j(p-1) < d$.
\end{proof}

From Proposition \ref{diff} and Lemma \ref{wt} we may thus rephrase Proposition \ref{erz} as follows:

\begin{proposition}[Exact roots in $\Z^m$, again]\label{erz-again} Let $d \geq 0$ be standard.
Let $\tilde f: \Z^m \to \ultra \T$ be a map of weighted degree $\leq d$.  Then we can find a map $\tilde g: \Z^m \to \ultra \T$ of weighted degree $\leq d+p-1$ such that $p\tilde g = \tilde f$. 
\end{proposition}

In order to prove Proposition \ref{erz-again}, we use the following explicit description of those functions $\tilde f: \Z^m \to \ultra \T$ of a given weighted degree, which generalises Lemma \ref{polybasic}(iii):

\begin{proposition}[Classification of polynomials]\label{erz-class}  Let $d \geq 0$ be standard, and let $\tilde f: \Z^m \to \ultra \T$ be a map.  Then the following are equivalent:
\begin{itemize}
\item[(i)] $\tilde f$ has weighted degree $\leq d$.
\item[(ii)]  $\tilde f$ can be expressed as
$$ \tilde f(x_1,\ldots,x_m) = \alpha + \sum_{\text{\tiny $\begin{array}{ll} & i_1,\ldots,i_m \geq 0; r \geq 0: \\ &(\sum_{j=1}^m D_j i_j)+r(p-1) \leq d\end{array}$}} \frac{c_{i_1,\ldots,i_m,r}}{p^{r+1}} \binom{x_1}{i_1} \ldots \binom{x_m}{i_m} \mod 1$$
for some $\alpha \in \ultra \T$ and integers $c_{i_1,\ldots,i_m,r}$.
\end{itemize}
\end{proposition}

Proposition \ref{erz-class} immediately implies Proposition \ref{erz-again}, since any element $\alpha \in \ultra \T$ has a $p^{\th}$ root, and any ``monomial'' $\frac{c_{i_1,\ldots,i_m,r}}{p^{r+1}} \binom{x_1}{i_1} \ldots \binom{x_m}{i_m}$ of degree $\leq d$ has a $p^{\th}$ root $\frac{c_{i_1,\ldots,i_m,r}}{p^{r+2}} \binom{x_1}{i_1} \ldots \binom{x_m}{i_m}$ of degree $\leq d+p-1$.

\begin{proof}  We first show that (ii) implies (i).  As each constant function $\alpha$ clearly is of degree $\leq d$, it suffices by linearity to show that the multinomial
$$ M: (x_1,\ldots,x_m) \mapsto \frac{1}{p^{r+1}} \binom{x_1}{i_1} \ldots \binom{x_m}{i_m} \mod 1$$
has weighted degree $\leq (\sum_{j=1}^m D_j i_j) + r(p-1)$ for any $i_1,\ldots,i_m \geq 0$ and $r \in \Z$.  (For inductive reasons we include the case when $r$ is negative, but the claim is trivial in those cases as the multinomial $M$ then vanishes modulo $1$.)

We prove this by induction on the weighted degree $d_M := (\sum_{j=1}^m D_j i_j) + r(p-1)$.  When $d_M \leq 0$ or less the claim is trivial, so suppose that $d_M$ is positive, and that the claim has already been proven for smaller values of $d_M$.

To show that $M$ has weighted degree $\leq d_M$, it then suffices to show that $\ader_{p^k e_j} M$ has weighted degree $\leq d_M - D_j - k(p-1)$ for each $1 \leq j \leq m$ and $k \geq 0$.  For sake of argument we shall just verify this when $j=m$, though the other cases are of course similar.  From the binomial identity \eqref{hafm} one has
$$ \ader_{p^k e_m} M(x_1,\ldots,x_m) = \sum_{l=1}^{i_m} \frac{1}{p^{r+1}} \binom{x_1}{i_1} \ldots \binom{x_m}{i_m-l} \binom{p^k}{l}\ \mod 1,$$
so it will suffice to show that each term
\begin{equation}\label{pear}
 \frac{1}{p^{r+1}} \binom{x_1}{i_1} \ldots \binom{x_m}{i_m-l} \binom{p^k}{l}\ \mod 1
\end{equation}
with $1 \leq l \leq i_m$ has weighted degree $\leq d_M - D_m - k(p-1)$.

Fix $l$.  We may assume that $l \leq p^k$, since the binomial coefficient $\binom{p^k}{l}$ vanishes otherwise. Let $t$ be the largest natural number such that $p^t$ divides $l$, then $t \leq k$.  Inspecting the binomial coefficient
$$ \binom{p^k}{l} = \frac{p^k}{l} \frac{p^k-1}{1} \frac{p^k-2}{2} \ldots \frac{p^k-l+1}{l-1}$$
we see that $p^{k-t}$ divides $\binom{p^k}{l}$.  Absorbing this factor into the $\frac{1}{p^{r+1}}$ term in \eqref{pear} and using the induction hypothesis, we conclude that \eqref{pear} has weighted degree
$$ \leq (\sum_{j=1}^{m-1} D_j i_j) + D_m (i_m - l) + (r-k+t)(p-1).$$
But note that
$$ D_m (l-1) \geq (l-1) \geq p^t-1 \geq t(p-1),$$
and the claim follows.

Next, we show that (i) implies (ii).  This claim is trivial for $d \leq 0$, so suppose inductively that $d > 0$ and that the claim has been proven for smaller values of $d$.  We then fix $d$ and assume as a second induction hypothesis that the claim has already been proven for smaller dimensions than $m$.  We may assume that $m > 0$, since the $m=0$ case is trivial.

Let $\tilde f$ be of weighted degree $\leq d$, and consider the derivative $\ader_{e_m} \tilde f$.  By (i), this function has weighted degree $\leq d-D_m$, and thus by the induction hypothesis has a representation of the form
$$ \ader_{e_m} \tilde f(x_1,\ldots,x_m) = \alpha + \sum_{\text{\tiny $\begin{array}{ll} & i_1,\ldots,i_m \geq 0; r \geq 0: \\ &(\sum_{j=1}^m D_j i_j)+r(p-1) \leq d-D_m\end{array}$}} \frac{c_{i_1,\ldots,i_m,r}}{p^{r+1}} \binom{x_1}{i_1} \ldots \binom{x_m}{i_m} \mod 1.$$
We now introduce the function
$$ \tilde g(x_1,\ldots,x_m) := \sum_{\text{\tiny $\begin{array}{ll} & i_1,\ldots,i_m \geq 0; r \geq 0: \\ &(\sum_{j=1}^m D_j i_j)+r(p-1) \leq d-D_m\end{array}$}} \frac{c_{i_1,\ldots,i_m,r}}{p^{r+1}} \binom{x_1}{i_1} \ldots \binom{x_{m-1}}{i_{m-1}} \binom{x_m}{i_m+1} \mod 1.$$
As (ii) implies (i), we know that $\tilde g$ has weighted degree $\leq d$.  From Pascal's identity we have
$$ \ader_{e_m} \tilde g(x_1,\ldots,x_m) = \sum_{\text{\tiny $\begin{array}{ll} & i_1,\ldots,i_m \geq 0; r \geq 0: \\ &(\sum_{j=1}^m D_j i_j)+r(p-1) \leq d-D_m\end{array}$}} \frac{c_{i_1,\ldots,i_m,r}}{p^{r+1}} \binom{x_1}{i_1} \ldots \binom{x_m}{i_m} \mod 1$$
and thus
$$ \ader_{e_m} \tilde f = \alpha + \ader_{e_m} \tilde g.$$
We thus have
$$ \tilde f(x_1,\ldots,x_m) = \alpha x_m + \tilde f(x_1,\ldots,x_{m-1},0) + \tilde g(x_1,\ldots,x_m).$$
As $\tilde f$ has weighted degree $\leq d$, the $m-1$-dimensional function $(x_1,\ldots,x_{m-1}) \mapsto \tilde f(x_1,\ldots,x_{m-1},0)$ does also.  By the second induction hypothesis, $\tilde f(x_1,\ldots,x_{m-1},0)$ is already of the required form for (ii), while $\tilde g$ is also of the required form by construction.  It remains to show that the term $\alpha x_m$ has the required form.

By linearity, $\alpha x_m$ is of weighted degree $\leq d$.  By Lemma \ref{wt}, we thus see that $p^j \alpha = 0$ whenever $D_m + j(p-1) > d$.  If we thus let $j$ be the first natural number for which $D_m + j(p-1) > d$, then $\alpha x_m$ is a multiple of $\frac{1}{p^j} \binom{x_m}{1}$; as (ii) implies (i), this has degree $\leq D_m + (j-1)(p-1) \leq d$, and the claim follows.
\end{proof}

The proof of Proposition \ref{erz-again} (and thus Proposition \ref{erz}) is now complete.

\section{Deducing the inverse conjecture from the inverse conjecture for polynomials}\label{deduce}

In this section we deduce Theorem \ref{main} from Theorem \ref{main-p}.  This deduction can be done in either a finitary or an infinitary setting.  In the finitary setting, one uses structural decomposition theorems as in \cite{tao-stoc}, \cite{gowers-reg}, \cite{finrat}, \cite{gowers-wolf-2}, \cite{gowers-wolf-3}; the arguments in \cite{gowers-wolf-3} are particularly close to those here.  In the infinitary setting one can proceed by analogous decomposition theorems based on conditional expectation.  We shall follow the latter approach here, in order to illustrate the parallel nature of the two arguments.  (This latter approach is also adopted in \cite{szeg}.)

We first give a general abstract structural decomposition.

\begin{lemma}[Decomposition]\label{deco}  Let $V$ be a limit finite set, and let ${\mathcal S}$ be a family of limit functions $P: V \to X$ on $V$, each of which takes only a finite number of values.  (We do not assume that ${\mathcal S}$ is itself a limit set.)  Let $f: V \to \ultra \C$ be a limit function bounded in magnitude by some standard real $A$.  Then one can decompose
$$ f = f_\str + f_\psd$$
where $f_\str, f_\psd: V \to \ultra \C$ are limit functions bounded in magnitude by $A$ and $2A$ respectively, with the following properties:
\begin{itemize}
\item ($f_\str$ almost structured)  For every standard $\eps > 0$, one can find a function $f_\eps: V \to \ultra \C$ that is a function of boundedly many functions from ${\mathcal S}$, such that $f_\eps$ is bounded in magnitude by $A$ and $\|f_\str-f_\eps\|_{L^2(V)} := (\E_{x \in V} |f_\str(x) - f_\eps(x)|^2)^{1/2} \leq \eps$.
\item ($f_\psd$ pseudorandom)  For every function $g: V \to \ultra \C$ that depends on only boundedly many functions from ${\mathcal S}$, one has $\langle f_\psd,g\rangle_{L^2(V)} := \E_{x \in V} f_\psd(x) \overline{g(x)} = o(1)$.
\end{itemize}
\end{lemma}

\begin{proof}  Given any finite subset ${\mathcal S}_0$ of ${\mathcal S}$, let ${\mathcal B}({\mathcal S}_0)$ be the $\sigma$-algebra of $V$ generated by the level sets of the functions of ${\mathcal S}_0$; this is a finite $\sigma$-algebra, with every atom being a limit subset of $V$.  Given such a $\sigma$-algebra, we can define the conditional expectation $\E(f|{\mathcal B}({\mathcal S}_0)): V \to \ultra \C$ of $f$ by the formula
$$ \E(f|{\mathcal B}({\mathcal S}_0))(x) := \E_{y \in {\mathcal B}({\mathcal S_0})(x)} f(y)$$
for all $x \in V$, where ${\mathcal B}({\mathcal S_0})(x)$ is the atom of ${\mathcal B}({\mathcal S_0})$ that contains $x$.  Clearly $\E(f|{\mathcal B}({\mathcal S}_0))$ is bounded in magnitude by $A$, and the \emph{energy}
$$ {\mathcal E}({\mathcal S}_0) := \| \E(f|{\mathcal B}({\mathcal S}_0)) \|_{L^2(V)}^2$$
is a non-negative real number between $0$ and $A^2$.  Let $E_{\max}$ denote the supremum of ${\mathcal E}({\mathcal S}_0)$; as the energy is monotone in ${\mathcal S}_0$ we can thus (using the axiom of choice) find an increasing sequence ${\mathcal S}_n$ for $n \in \N$ such that ${\mathcal E}({\mathcal S}_n) \to E_{\max}$.

From Pythagoras' theorem we have
$$ \st \|\E(f|{\mathcal B}({\mathcal S}_{n'}) - \E(f|{\mathcal B}({\mathcal S}_{n})\|_{L^2(V)}^2 = {\mathcal E}({\mathcal S}_{n'}) - {\mathcal E}({\mathcal S}_n)$$
for any $n' > n$, where $\st x$ denotes the standard part of the limit real $x$.  Thus the $\E(f|{\mathcal B}({\mathcal S}_{n})$ are an $L^2$ Cauchy sequence in the sense that
$$ \lim_{n,n' \to \infty} \st \|\E(f|{\mathcal B}({\mathcal S}_{n'})) - \E(f|{\mathcal B}({\mathcal S}_{n}))\|_{L^2(V)}^2 = 0$$
We claim that this implies the existence of a limit function $f_\str: V \to \C$, bounded in magnitude by $A$, such that
$$ \lim_{n \to \infty} \st \|f_\str - \E(f|{\mathcal B}({\mathcal S}_{n}))\|_{L^2(V)}^2 = 0.$$
Indeed, if we write $V = \prod_{\alpha \to \alpha_\infty} V_\alpha$, $f = \lim_{\alpha \to \alpha_\infty} f_\alpha$, and ${\mathcal S}_n = \prod_{\alpha \to \alpha_\infty} {\mathcal S}_{n,\alpha}$, one can set
$$ f_\str := \lim_{\alpha \to \alpha_\infty} \E(f_\alpha|{\mathcal B}({\mathcal S}_{n_\alpha,\alpha}))$$
and the claim will follow if $n_\alpha$ increases to infinity at a sufficiently slow rate; we omit the routine details.

Now let $g$ depend on a bounded number ${\mathcal S}'$ of functions from ${\mathcal S}$, such that $g$ is bounded in magnitude by $B$.  Then for any standard natural $n$, one can rewrite
$$ \st \langle f - \E(f|{\mathcal B}({\mathcal S}_n)), g \rangle_{L^2(V)}$$
as
$$ \st \langle \E(f|{\mathcal B}({\mathcal S}_n \cup {\mathcal S}')) - \E(f|{\mathcal B}({\mathcal S}_n)), g \rangle_{L^2(V)}$$
which by the Cauchy-Schwarz inequality is bounded in magnitude by
$$ B \|\E(f|{\mathcal B}({\mathcal S}_n \cup {\mathcal S}')) - \E(f|{\mathcal B}({\mathcal S}_n))\|_{L^2(V)}$$
which by Pythagoras' theorem and definition of $E_{\max}$ is bounded by
$$ B (E_{\max} - {\mathcal E}({\mathcal S}_n))^{1/2}$$
and thus
$$ \lim_{n \to \infty} \st \langle f - \E(f|{\mathcal B}({\mathcal S}_n)), g \rangle_{L^2(V)} = 0.$$
Taking limits using the Cauchy-Schwarz and triangle inequalities, we conclude that
$$ \lim_{n \to \infty} \st \langle f - f_\str, g \rangle_{L^2(V)} = 0.$$
Setting $f_\psd := f - f_\str$, we now obtain the claim.
\end{proof}

\begin{remark} Although we will not need this fact here, it is often useful to observe that if $f$ is non-negative, then $f_\str$ and $f_\eps$ can be taken to be non-negative also.  One can also establish this lemma using the machinery of \emph{Loeb measure} \cite{loeb}: if $\mu$ is Loeb measure on $V$, then $f_\str$ is essentially the conditional expectation (in $L^2(\mu)$) of $f$ with respect to the $\sigma$-algebra generated by ${\mathcal S}$.  See \cite{szeg} for an implementation of this approach (and \cite{towsner} for some further discussion of the role of Loeb measure in the nonstandard version of the Gowers norms).
\end{remark}

Next, we recall one of the main theorems from \cite{berg}, phrased in the ultralimit setting:

\begin{theorem}[Weak inverse Gowers conjecture]\label{gis2} Let $s \geq 0$ be standard.  Then there exists a standard integer $d = d(s,p) \geq 0$ such that for every limit finite-dimensional vector space $V$ and every bounded limit function $f: V \to \ultra \C$ with $\|f\|_{U^{s+1}(V)} \gg 1$, there exists $P \in \Poly_{\leq d}(V \to \ultra \T)$ such that $|\E_{x \in V} f(x) e(-P(x))| \gg 1$. 
\end{theorem}

\begin{proof} See \cite[Corollary 1.23]{berg}.  The translation to the ultralimit setting proceeds exactly as in Section \ref{ultra}.
\end{proof}

We are now ready to deduce Theorem \ref{main} from Theorem \ref{main-p}.  Fix $s, p$, and let $d$ be the minimal integer for which Theorem \ref{gis2} holds for this value of $s, p$.  If $d \leq s$, then we are done by Theorem \ref{gis}, so suppose for sake of contradiction that $d>s$.

By construction of $d$, we can find a limit finite-dimensional vector space $V$ and a bounded limit function $f: V \to \ultra \C$ such that $\|f\|_{U^{s+1}(V)} \gg 1$, but such that
$$ \langle f, e(P) \rangle_{L^2(V)} = o(1)$$
whenever $P$ is a polynomial of degree $\leq d-1$.  By Fourier analysis (and Lemma \ref{polybasic}(vi)), this implies that
$$ \langle f, g \rangle_{L^2(V)} = o(1)$$
whenever $g$ is a function of a bounded number of polynomials of degree $\leq d-1$.  In particular we have
$$ \langle f, e(P) \rangle_{L^2(V)} = o(1)$$
whenever $P$ is a polynomial of degree $\leq d$ of bounded rank.

Using Lemma \ref{deco} (and Lemma \ref{polybasic}(vi)), we can decompose $f = f_\str + f_\psd$, where $f_\str, f_\psd: V \to \ultra \C$ are bounded limit functions, we have
$$ \langle f_\psd, g \rangle_{L^2(V)} = o(1)$$
whenever $g$ is a function of a bounded number of polynomials of degree $\leq d$, and for every $\eps > 0$ we can approximate $f_\str$ by a bounded limit function $f_\eps$ that is a function of a bounded number of polynomials of degree $\leq d$.  In particular, by Theorem \ref{gis2}, one has
$$ \|f_\psd\|_{U^{s+1}(V)} = o(1)$$
and hence by the triangle inequality for $U^{s+1}(V)$ (Lemma \ref{gow}(i)) one has
$$ \|f_\str\|_{U^{s+1}(V)} \gg 1.$$
Also from the triangle inequality we see that
\begin{equation}\label{fa}
\langle f_\str, e(P) \rangle_{L^2(V)} = o(1)
\end{equation}
whenever $P$ is a polynomial of degree $\leq d$ of bounded rank.

As the $f_\str - f_\eps$ are uniformly bounded in $\eps$, and have (the standard part of the) $L^2$ norm going to zero as $\eps \to 0$, we see (using Lemma \ref{gow}(iii)) that
$$ \st \|f_\str - f_\eps \|_{U^{s+1}(V)} \to 0$$
as $\eps \to 0$. For all sufficiently small $\eps > 0$, we thus have
\begin{equation}\label{feps-large}
 \|f_\eps\|_{U^{s+1}(V)} \gg 1
\end{equation}
uniformly in $\eps$.  

By Fourier analysis, we can express $f_\eps$ as a bounded linear combination of phases $e(P)$, where the $P$ are polynomials of degree $\leq d$.  We separate $f_\eps = f'_\eps + f''_\eps$, where $f'_\eps$ is a linear combination of phases $e(P)$ of unbounded rank, and $f''_\eps$ is a linear combination of phases of $e(P)$ bounded rank.  

From Theorem \ref{main-p} (and Theorem \ref{main-q}) we see that $\|e(P)\|_{U^d(V)} = o(1)$ whenever $P$ has unbounded rank, and in particular (by Lemma \ref{gow}(ii)) $\|e(P)\|_{U^{s+1}(V)} = o(1)$ and $\E_{x \in V} e(P(x)) = o(1)$.  Since the difference of a degree $d$ polynomial of unbounded rank and a degree $d$ polynomial of bounded rank remains of unbounded rank, we also have $\E_{x \in V} e(P(x)-Q(x)) = o(1)$ whenever $Q$ is of bounded rank.  We conclude that $f'_\eps, f''_\eps$ are essentially orthogonal in the sense that
$$ \langle f'_\eps, f''_\eps \rangle_{L^2(V)} = o(1)$$
and hence
$$ \|f''_\eps\|_{L^2(V)}^2 = \langle f_\eps, f''_\eps \rangle_{L^2(V)} + o(1).$$
On the other hand, from \eqref{fa} one has
$$ \langle f_\str, f''_\eps \rangle_{L^2(V)} = o(1)$$
while from the Cauchy-Schwarz inequality we see that
$$ \st \langle f_\str - f_\eps, f''_\eps \rangle_{L^2(V)} \to 0$$
as $\eps \to 0$.  We conclude that
$$ \st \|f''_\eps\|_{L^2(V)} \to 0$$
as $\eps \to 0$, which in particular implies (by Lemma \ref{gow}(iii)) that
$$ \st \|f''_\eps\|_{U^{s+1}(V)} \to 0$$
as $\eps \to 0$.  Also, as $f'_\eps$ is a bounded linear combination of $e(P)$ for $P$ of unbounded rank, and thus of infinitesimal $U^{s+1}(V)$ norm, we see from the triangle inequality (Lemma \ref{gow}(i)) that
$$ \|f'_\eps\|_{U^{s+1}(V)} = o(1)$$
and hence
$$ \st \|f_\eps\|_{U^{s+1}(V)} \to 0$$
as $\eps \to 0$, contradicting \eqref{feps-large}.  This concludes the deduction of Theorem \ref{main} from Theorem \ref{main-p}.

\appendix

\section{Basic theory of ultralimits}\label{nsa-app}

In this appendix we review the machinery of ultralimits.  

We will assume the existence of a \emph{standard universe} ${\mathfrak U}$ which contains all the objects and spaces of interest for Theorem \ref{main} or Theorem \ref{main-p}, such as the natural numbers, standard finite-dimensional vector spaces $\F^n$ and their elements, the unit circle $\T$ and its elements, functions from the former spaces to the latter (such as polynomials $P \in \Poly_{\leq d}(\F^n \to \T)$), and so forth.  The precise construction of this universe is not important, so long as it forms a set.  We refer to objects and spaces inside the standard universe as \emph{standard objects} and \emph{standard spaces}, with the latter being sets whose elements are in the former category.  Thus for instance, elements of $\N$ are standard natural numbers, and for every standard natural number $n$, $\F^n$ is a standard finite-dimensional vector space.  Strictly speaking, the universe ${\mathfrak U}$ cannot contain \emph{all} finite-dimensional vector spaces, as the class of such spaces is not a set, but for the purposes of proving Theorem \ref{main} or Theorem \ref{main-p} we only need to pick one representative of each isomorphism class of such spaces, such as $\F^n$ for $n=0,1,2,\ldots$, and these certainly form a set.

The one technical ingredient we need is the following:

\begin{lemma}[Ultrafilter lemma]  There exists a collection $\alpha$ of subsets of the natural numbers $\N$ with the following properties:
\begin{enumerate}
\item (Monotonicity) If $A \in \alpha$ and $B \supset A$, then $B \in \alpha$.
\item (Closure under intersection) If $A,B \in \alpha$, then $A \cap B \in \alpha$.
\item (Maximality) If $A \subset \N$, then either $A \in \alpha$ or $\N \backslash A \in \alpha$, but not both.
\item (Non-principality) If $A \in \alpha$, and $A'$ is formed from $A$ by adding or deleting finitely many elements to or from $A$, then $A' \in \alpha$.
\end{enumerate}
\end{lemma}

\begin{proof}  The collection of subsets of $\N$ which are cofinite (i.e. whose complement is finite) already obeys the monotonicity, closure under intersection, and non-principality properties.  Using Zorn's lemma\footnote{By using this lemma, our results thus rely on the axiom of choice, which we will of course assume throughout this paper.  On the other hand, it is possible to rephrase Theorem \ref{main} and Theorem \ref{main-p} in the language of Peano arithmetic.  Applying a famous theorem of G\"odel\cite{godel}, we then conclude that Theorem \ref{main-p} is provable in ZFC if and only if it is provable in ZF.  In fact, it is possible (with some effort) to directly translate these ultrafilter arguments to a (lengthier) argument in which ultrafilters or the axiom of choice is not used.  We will not do so here, though, as the translation is quite tedious.  In particular, the regularity lemma and equidistribution arguments in this paper will become messier, resembling those that appear in \cite{finrat} or \cite{kauf}.}, one can enlarge this collection to a maximal collection, which then obeys all the required properties.
\end{proof}

Throughout the paper, we fix a non-principal ultrafilter $\alpha$.  A property $P(\n)$ depending on a natural number $\n$ is said to hold \emph{for $\n$ sufficiently close to $\alpha$} if the set of $\n$ for which $P(\n)$ holds lies in $\alpha$.

Once we have fixed this ultrafilter, we can now define limit objects and spaces:

\begin{definition}[Limit objects]
Given a sequence $(x_\n)_{\n \in \N}$ of standard objects in ${\mathfrak U}$, we define their \emph{ultralimit} $\lim_{\n \to \alpha} x_\n$ to be the equivalence class of all sequences $(y_\n)_{\n \in \N}$ of standard objects in ${\mathfrak U}$ such that $x_\n = y_\n$ for $\n$ sufficiently close to $\alpha$.  Note that the ultralimit $\lim_{\n \to \alpha} x_\n$ can also be defined even if $x_\n$ is only defined for $\n$ sufficiently close to $\alpha$.

An ultralimit of standard natural numbers is known as a \emph{limit natural number}, an ultralimit of standard real numbers is known as a \emph{limit real number}, etc.

For any standard object $x$, we identify $x$ with its own ultralimit $\lim_{\n \to \alpha} x$.  Thus, every standard natural number is a limit natural number, etc.

Any operation or relation on standard objects can be extended to limit objects in the obvious manner.  For instance, the sum of two limit real numbers $\lim_{\n \to \alpha} x_\n$, $\lim_{\n \to \alpha} y_\n$ is the limit real number
$$\lim_{\n \to \alpha} x_\n + \lim_{\n \to \alpha} y_\n = \lim_{\n \to \alpha} x_\n + y_\n,$$
and the statement $\lim_{\n \to \alpha} x_\n < \lim_{\n \to \alpha} y_\n$ means that $x_\n < y_\n$ for all $\n$ sufficiently close to $\alpha$.
\end{definition}

\begin{remark} A famous theorem of {\L}os asserts that any statement in first-order logic which is true about standard objects, is automatically true for limit objects as well.  For instance, the standard real numbers form an ordered field, and so the limit real numbers do also, because the axioms of an ordered field can be phrased in first-order logic.  We will use this theorem in the sequel without further comment.
\end{remark}

\begin{definition}[Limit spaces and functions]  Let $(X_\n)_{\n \in \N}$ be a sequence of standard spaces $X_\n$ in ${\mathfrak U}$ indexed by the natural numbers.  The \emph{ultraproduct} $\prod_{\n \to \alpha} X_\n$ of the $X_\n$ is defined to be the space of all ultralimits $\lim_{\n \to \alpha} x_\n$, where $x_\n \in X_\n$ for all $\n$.  Note $X_\n$ only needs to be well-defined for $\n$ sufficiently close to $\alpha$ in order for the ultraproduct to be well-defined.  If $X$ is a set, the set $\prod_{\n \to \alpha} X$ is known as the \emph{ultrapower} of $X$ and is denoted $\ultra X$.  Thus for instance $\ultra \N$ is the space of all limit natural numbers, $\ultra \R$ is the space of all limit reals, etc.

We define a \emph{limit set} to be an ultraproduct of sets, a \emph{limit group} to be an ultraproduct of groups, a \emph{limit finite set} to be an ultraproduct of finite sets, and so forth.  A \emph{limit subset} of a limit set $X = \prod_{\n \to \alpha} X_\n$ is a limit set of the form $Y = \prod_{\n \to \alpha} Y_\n$, where $Y_\n$ is a standard subset of $X_\n$ for all $\n$ sufficiently close to $\alpha$.

Given a sequence of standard functions $f_\n: X_\n \to Y_\n$ between standard sets $X_\n, Y_\n$, we can form the \emph{ultralimit} $f = \lim_{\n \to \alpha} f_\n$ to be the function $f: \prod_{\n \to \alpha} X_\n \to \prod_{\n \to \alpha} Y_\n$ defined by the formula
$$ f( \lim_{\n \to \alpha} x_\n ) := \lim_{\n \to \alpha} f_\n(x_\n).$$
We refer to $f$ as a \emph{limit function} or \emph{limit map}.
\end{definition}

\begin{remark} In the nonstandard analysis literature, limit natural numbers are known as \emph{nonstandard natural numbers}, limit sets are known as \emph{internal sets}, and limit functions are known as \emph{internal functions}. We have chosen the limit terminology instead as we believe that it is less confusing and emphasises the role of ultralimits in the subject.

It is important to note that not every subset of a limit set is again a limit set, for instance $\N$ is not a limit subset of $\ultra \N$ (this fact is known as the \emph{overspill principle}).  Indeed, one can think of the limit subsets of a limit set as being analogous to the measurable subsets of a measure space.  In a similar vein, not every function between two limit sets is a limit function; in this regard, limit functions are analogous to measurable functions.  This analogy can be deepened by using the theory of \emph{Loeb measures}, but we will avoid using this machinery here.
\end{remark}

\subsection{Asymptotic notation}

By taking ultralimits, one can formalise asymptotic notation, such as the $O()$ notation, in a manner that requires no additional quantifiers:

\begin{definition}[Asymptotic notation]    A limit complex number $X$ is said to be \emph{bounded} if one has $|X| \leq C$ for some standard real number $C$, in which case we also write $X=O(1)$ or $|X| \ll 1$.  More generally, given a limit complex number $X$ and limit non-negative number $Y$, we write $|X| \ll Y$, $Y \gg |X|$, or $X = O(Y)$ if one has $|X| \leq CY$ for some standard real number $C$.  We write $X = o(Y)$ if one has $|X| \leq \eps Y$ for every standard $\eps > 0$.  Observe that for any $X, Y$ with $Y$ positive, one has either $|X| \gg Y$ or $X = o(Y)$.  We say that $X$ is \emph{infinitesimal} if $X=o(1)$, and \emph{unbounded} if $1/X = o(1)$.  Thus for instance any limit complex number $X$ will either be bounded or unbounded.
\end{definition}

\begin{example}  The limit real $\lim_{\n \to \alpha} 1/\n$ defines an infinitesimal, but non-zero, limit real number $x$; its reciprocal $\lim_{\n \to \alpha} \n$ is an unbounded limit real.  
\end{example}

From the Bolzano-Weierstrass theorem, every bounded limit complex number $x$ can be expressed uniquely as the sum of a standard real number $\st(x)$ and an infinitesimal $x-\st(x)$; we refer to $\st(x)$ as the \emph{standard part} of $x$.

\section{Properties of the Gowers norms}\label{gowapp}

In this appendix we record some basic properties of the Gowers norms.  We use the normalised $L^p$ norms
$$ \|f\|_{L^p(G)} := (\E_{x \in G} |f(x)|^p)^{1/p}$$
for any finite non-empty set $G$ and any $f: G \to \C$.

\begin{lemma}\label{gow}  Let $G = (G,+)$ be a finite abelian group, and let $d \geq 1$ be an integer.
\begin{itemize}
\item[(i)] The Gowers norm $\| \|_{U^d(G)}$ is a norm on functions $f: G \to \C$ for $d \geq 2$, and a semi-norm for $d=1$.  In particular, we have the \emph{Gowers triangle inequality}
$$ \|f+g\|_{U^d(G)} \leq \|f\|_{U^d(G)} + \|g\|_{U^d(G)}$$
for $f, g: G \to \C$.
\item[(ii)] One has the monotonicity property
$$ \| f \|_{U^d(G)} \leq \|f\|_{U^{d+1}(G)}$$
for all $f: G \to \C$.  In particular
$$ |\E_{x \in V} f(x)| = \|f\|_{U^1(G)} \leq \|f\|_{U^d(G)}.$$
\item[(iii)] One has the bound
$$ \|f\|_{U^d(G)} \leq \|f\|_{L^{2^d/(d+1)}(G)}$$
for all $f: G \to \C$.  
\item[(iv)] One has the \emph{first Cauchy-Schwarz-Gowers inequality}
$$ |\E_{x,h_1,\ldots,h_d \in G} \prod_{\omega \in \{0,1\}^d} f_\omega(x+\omega_1 h_1 + \ldots + \omega_d h_d)| \leq \prod_{\omega \in \{0,1\}^d} \|f_\omega\|_{U^d(G)}$$
for all $\{0,1\}^d$-tuples $(f_\omega)_{\omega \in \{0,1\}^d}$ of functions $f_\omega: G \to \C$, where $\omega := (\omega_1,\ldots,\omega_d)$.
\item[(v)] One has the \emph{second Cauchy-Schwarz-Gowers inequality}
$$ |\E_{x_1,\ldots,x_d \in G} f(x_1+\ldots+x_d) \prod_{j=1}^d F_j(x_1,\ldots,x_d)| \leq \|f\|_{U^d(G)}$$
for all $f: G \to \C$ and $F_j: G^d \to \C$, if each $F_j$ is bounded in magnitude by $1$ and is independent of the $x_j$ variable.
\item[(vi)] If $P \in \Poly_{\leq d-1}(G \to \T)$ and $f: G \to \C$ then 
\begin{equation}\label{modulate}
\|f e(P)\|_{U^d(G)} = \|f\|_{U^d(G)}.
\end{equation}
\end{itemize}
\end{lemma}

\begin{proof}  Claim (i) is proven in \cite[Lemma 3.9]{gowers}, \cite[Section 5.1]{gt-primes} or \cite[Section 11.1]{tao-vu}.  Claim (ii) is proven in \cite[Section 11.1]{tao-vu}.  Claim (iii) (which is also \cite[Exercise 11.1.13]{tao-vu}) follows easily from the recursive formula
$$ \| f\|_{U^{d+1}(V)} = (\E_{h \in V} \| \mder_h f \|_{U^d(V)}^{2^d})^{1/2^{d+1}} $$
and induction on $d$, together with the special case
$$ 
(\E_{h \in V} \| f(\cdot+h) g(\cdot) \|_{L^{2^d/(d+1)}(V)}^{2^d})^{1/2^{d}}
\leq \|f\|_{L^{2^{d+1}/(d+2)}(V)} \|g\|_{L^{2^{d+1}/(d+2)}(V)}$$
of Young's convolution inequality, which follows from the more traditional instance
$$
(\E_{h \in V} |\E_{x \in V} F(x+h) G(h)|^{d+1})^{1/(d+1)}
\leq \|F\|_{L^{2(d+1)/(d+2)}(V)} 
$$
of that inequality by setting $F := |f|^{2^d/(d+1)}$ and $G := |g|^{2^d/(d+1)}$.

Claim (iv) is proven in \cite[Lemma 3.8]{gowers}, \cite[Section 5.1]{gt-primes} or \cite[Section 11.1]{tao-vu}; Claim (v) is proven in \cite[Appendix B]{green-tao-linearprimes}.  Claim (vi) follows immediately from the identity $\mder_{h_1} \ldots \mder_{h_d} e(P) = 1$ for all $h_1,\ldots,h_d \in V$.
\end{proof}

\section{Polynomial algebra}\label{poly-alg}

In this appendix we review the general theory of polynomial maps (and related objects, such as cubes) on (filtered) groups that are not necessarily abelian (in particular, they may be nilpotent).  This theory was initiated by Lazard \cite{lazard} and Leibman \cite{leibman-group-1}, \cite{leibman-group-2} (inspired in part by the classical Hall-Petresco formula \cite{hall}, \cite{petresco}), and further developed by Host and Kra \cite{host-kra}, \cite{host-kra-nil} and by Green and the authors \cite{green-tao-linearprimes}, \cite{gtz}.  Our discussion here is largely drawn from the paper \cite{gtz}.

Polynomial algebra works on both multiplicative groups $G = (G,\cdot)$ and on additive groups $G = (G,+)$.  For sake of concreteness we shall set out the theory here using multiplicative group notation, but one can of course adapt all the definitions here to additive groups in an obvious manner, and in fact most of the applications of this theory in this paper will be in the additive setting.  Our conventions will be that additive groups are always understood to be abelian, whereas multiplicative groups are not necessarily abelian.  

The concepts here can be defined both in the standard and nonstandard setting, but again for concreteness we shall work purely in the standard universe in this appendix.  But all of the results here can easily be phrased in the language of first-order logic (they involve only finitely many quantifiers) and so extend without difficulty to the nonstandard universe also.

\subsection{The category of filtered groups}

The theory of polynomial maps is most naturally expressed in terms of a certain category of filtered groups, which we will now define.

\begin{definition}[Filtered group]  A \emph{filtered group} $G = (G,G_\N)$ is a multiplicative group $G = (G,\cdot)$, together with a nested sequence 
$$ G \geq G_0 \geq G_1 \geq \ldots$$
of subgroups $G_\N = (G_n)_{n \in \N}$, obeying the commutator relation $[G_i,G_j] \subset G_{i+j}$ for all $i,j \in \N$, where $[G_i,G_j]$ is the group generated by the commutators $g_i^{-1} g_j^{-1} g_i g_j$ with $g_i \in G_i, g_j \in G_j$.  We refer to $G_\N$ as a \emph{filtration} of $G$.

A filtered group is said to have \emph{degree $\leq s$} for some natural number $s$ if $G_i$ is trivial for all $i > s$.
\end{definition}

\begin{example}[Abelian case]\label{filt-ab}  When the group $G = (G,+)$ is additive (and thus abelian), a filtered group is simply a nested sequence $G \geq G_0 \geq G_1 \geq \ldots$ of subspaces (since the commutator relation is automatic in this case).  In particular, for any natural number $k \geq 0$, one can give any additive group $G$ the \emph{maximal degree $k$ filtration} $G_\N = G_\N^{(k)}$, defined by setting $G_i$ equal to $G$ when $i \leq k$ and $G_i = \{0\}$ for $i>k$. 
\end{example}

\begin{example}[Lower central series]  Any group $G$ can become a filtered group by taking $G_i$ to be the lower central series of $G$, thus $G_0=G_1 := G$ and $G_{i+1} := [G,G_i]$ for $i \geq 1$.  
\end{example}

\begin{remark}  In this paper we will only filter groups $G$ by the natural numbers $\N$.  However it is sometimes convenient to filter groups by other sets, such as $\N^k$, in order to develop a theory of ``multidegree'' for polynomials of several variables; see \cite{gtz}.  While one could use this notation to describe the multilinear maps that arise in this paper, we have chosen not to do so here in order not to add even more terminology to what is already quite a notation-intensive argument.
\end{remark}

An obvious way to make the class of all filtered groups a category is to use the \emph{filtered homomorphisms} $\phi: H \to G$ between two filtered groups $H = (H,H_\N)$, $G = (G,G_\N)$, defined as a group homomorphism from $H$ to $G$ that maps each $H_i$ to $G_i$.  However, this turns out to be too small a class of morphisms for our purposes, and we will need to use instead the larger class of \emph{polynomial maps} between two filtered groups.  This concept can be defined in a number of different ways.  The quickest way is via differentiation:

\begin{definition}[Polynomial maps via differentiation]\label{diffeo}  Let $H = (H,H_\N)$, $G = (G,G_\N)$ be filtered groups, and let $\phi: H \to G$ be a map.  For any $h \in H$, we define the \emph{derivative} $\partial_h \phi: H \to G$ of $\phi$ in the direction $h$ by the formula
$$ \partial_h \phi(x) := \phi(hx) \phi(x)^{-1}.$$
We say that the map $\phi: H \to G$ is a \emph{polynomial map} if one has
$$ \partial_{h_1} \ldots \partial_{h_m} \phi(x) \in G_{i_1 + \ldots + i_m}$$
whenever $m \geq 0$ and $i_1,\ldots,i_m \in \N$, and $h_j \in H_{i_j}$ for all $1 \leq j \leq m$.  The space of all polynomial maps from $H$ to $G$ will be denoted $\Poly(H \to G)$.
\end{definition}

\begin{example}[Non-classical polynomials as polynomial maps] If $V, G$ are additive groups, with $V$ given the maximal degree $\leq 1$ filtration, and $G$ the maximal degree $\leq k$ filtration for some $k \in \N$, then $\Poly(V \to G)$ corresponds precisely to the space $\Poly_{\leq k}(V \to G)$ defined in Definition \ref{polydef} (this definition was for finite-dimensional vector spaces $V$, but the definition clearly also makes sense for other additive groups).  In particular, a non-classical polynomial $P: V \to \T$ of degree $\leq k$ is also a polynomial map from $V$ (with the maximal degree $\leq 1$ filtration) to $\T$ (with the maximal degree $\leq k$ filtration).
\end{example}

\begin{example} Every filtered homomorphism is a polynomial map.  For any $g \in G$, the left translation maps $x \mapsto gx$ and right translation maps $x \mapsto gx$ are polynomial maps from $G$ to itself.
\end{example}

\begin{remark}
A basic theorem of Lazard and Leibman \cite{lazard}, \cite{leibman-group-1}, \cite{leibman-group-2} asserts that $\Poly(H \to G)$ is a group; see e.g. \cite[Corollary B.11]{gtz}.  This generalises the (obvious) fact that $\Poly_{\leq k}(V \to G)$ is a group in the additive case.
\end{remark}

A convenient fact about polynomiality is that it suffices to check it on generators:

\begin{proposition}[Checking polynomiality on generators]\label{polygen}  Let $H, G$ be filtered groups, and for each $i \in \N$, let $E_i$ be a set of generators for $H_i$.  Then a map $\phi: H \to G$ is polynomial if and only if
$$ \partial_{h_1} \ldots \partial_{h_m} \phi(x) \in G_{i_1 + \ldots + i_m}$$
whenever $m \geq 0$ and $i_1,\ldots,i_m \in \N$, and $h_j \in E_{i_j}$ for all $1 \leq j \leq m$.  
\end{proposition}

\begin{proof} See \cite[Proposition B.17]{gtz}.
\end{proof}

It is not immediately obvious from Definition \ref{diffeo} that the polynomial maps turn the class of filtered groups into a category, because one has to show that the composition of two polynomial maps is still polynomial.  However, this can be achieved via the machinery of \emph{Host-Kra cube groups}, which we now pause to define.

\begin{definition}[Host-Kra cube group]  Let $G = (G,G_\N)$ be a filtered group.  For any $k \in \N$, we define the $k^{\th}$ \emph{Host-Kra cube group} $\HK^k(G) = \HK^k(G,G_\N)$ of this filtration to be the subgroup of $G^{\{0,1\}^k}$ generated by those elements $(g_\omega)_{\omega \in \{0,1\}^k}$ which take the form $g_\omega = g$ for $\omega \in F$ and $g_\omega = \id$ otherwise, where $F$ is a face of $\{0,1\}^k$ of some codimension $i$, and $g$ is an element of $G_i$.  Elements $(g_\omega)_{\omega \in \{0,1\}^k}$ of $\HK^k(G)$ will be referred to as \emph{$k$-dimensional cubes\footnote{These cubes are also referred to as \emph{parallelopipeds} in some literature, e.g. \cite{host-kra-nil}.} in $G$}.
\end{definition}

We have an alternate description of these groups via a ``Taylor expansion'':

\begin{proposition}\label{prop}  Let $G = (G,G_\N)$ be a filtered group, let $k \in \N$, and let $g \in G^{\{0,1\}^k}$.  Then $g \in \HK^k(G)$ if and only if there exist ``Taylor coefficients'' $g_J \in G_{|J|}$ for each subset $J \subset \{1,\ldots,k\}$ such that
\begin{equation}\label{taylo-exp}
 g = (\prod_{J \subset \{1,\ldots,k\}} g_J^{\prod_{j \in J} \omega_j})_{\omega \in \{0,1\}^k} = \prod_{J \subset \{1,\ldots,k\}} (g_J^{\prod_{j \in J} \omega_j})_{\omega \in \{0,1\}^k},
\end{equation}
where the subsets of $\{1,\ldots,k\}$ are ordered lexicographically (i.e. $J < J'$ whenever $\sum_{j \in J} 2^{-j} < \sum_{j \in J'} 2^{-j}$).
Furthermore, the $g_J$ are determined uniquely by $g$.
\end{proposition}

\begin{proof} See \cite[Lemma 6.4]{green-tao-nilratner}.
\end{proof}

Thus, for instance, $\HK^2(G)$ consists of all tuples of the form $(g_{00}, g_{00} g_{01}, g_{00} g_{10}, g_{00} g_{01} g_{10} g_{11})$, where $g_{00} \in G_0$, $g_{01}, g_{10} \in G_1$, and $g_{11} \in G_2$. 

\begin{theorem}[Polynomial maps via cubes]\label{cube}  Let $G = (G,G_\N)$ and $H = (H,H_\N)$ be filtered groups, and let $\phi: H \to G$ be a map.  Then $\phi$ is a polynomial map if and only if $\phi$ preserves cubes, in the sense that for any $k \in \N$ and $(h_\omega)_{\omega \in \{0,1\}^k} \in \HK^k(H)$, the tuple $(\phi(h_\omega))_{\omega \in \{0,1\}^k}$ lies in $\HK^k(G)$.  (In other words, the map $\phi^{\oplus \{0,1\}^k}: H^{\{0,1\}^k} \to G^{\{0,1\}^k}$ maps $\HK^k(H)$ to $\HK^k(G)$.)
\end{theorem}

As an immediate corollary of this theorem, we see that the composition of two polynomial maps is again polynomial, and so the class of filtered maps is now a category.  If $\phi: H \to G$ is a polynomial map and $k \in \N$, we use $\HK^k(\phi): \HK^k(H) \to \HK^k(G)$ to denote the restriction of $\phi^{\oplus\{0,1\}^k}: H^{\{0,1\}^k} \to G^{\{0,1\}^k}$ to $\HK^k(H)$ and $\HK^k(G)$.

\begin{proof} When $H$ is additive, this theorem was proven in \cite[Proposition 6.5]{green-tao-nilratner}.  We will give an alternate proof based on \cite[Theorem B.10]{gtz}.  To use this theorem, we need a generalisation of the Host-Kra groups.  For any natural numbers $i_1,\ldots,i_m \in \N$, define the \emph{Host-Kra group} $\HK^{i_1,\ldots,i_k}(G)$ of a filtered group $G$ 
to be the subgroup of $G^{\{0,1\}^k}$ generated by the elements of the form
$$ ( g_\omega )_{\omega \in \{0,1\}^k},$$
where $\omega_0 \in \{0,1\}^k$, $g_{\omega_0} \in G_{\sum_{(\omega_0)_j=1} i_j}$, and $g_\omega$ equals $g_{\omega_0}$ when $\omega_j \geq (\omega_0)_j$ for all $1 \leq j \leq k$, and is the identity otherwise.  Thus for instance, when $i_1=\ldots=i_k=1$, then $\HK^{1,\ldots,1}(G) = \HK^k(G)$.  It is easy to adapt the proof of Proposition \ref{prop} to see that elements of $\HK^{i_1,\ldots,i_k}(G)$ are precisely those tuples of the form \eqref{taylo-exp}, where each Taylor coefficient $g_J$ now lies in $G_{\sum_{j \in J} i_j}$ rather than $G_{|J|}$.

The result \cite[Theorem B.10]{gtz} asserts that $\phi$ is a polynomial map if and only if $\phi$ (or more precisely, $\phi^{\oplus \{0,1\}^k}$) maps $\HK^{i_1,\ldots,i_k}(H)$ to $\HK^{i_1,\ldots,i_k}(G)$ for every $k, i_1,\ldots,i_k \in \N$.  In view of this result, to prove Theorem \ref{cube}, it suffices to show that if $\phi$ maps $\HK^k(H)$ to $\HK^k(G)$, then it maps $\HK^{i_1,\ldots,i_k}(H)$ to $\HK^{i_1,\ldots,i_k}(G)$ for every $k,i_1,\ldots,i_k \in \N$.

For any $k,i_1,\ldots,i_k \in \N$.  Let $P(i_1,\ldots,i_k)$ denote the assertion that $\phi^{\oplus \{0,1\}^k}$ maps $\HK^{i_1,\ldots,i_k}(H)$ to $\HK^{i_1,\ldots,i_k}(G)$.  By hypothesis, $P(1,\ldots,1)$ is true for any number of $1$'s; our task is to then show that $P(i_1,\ldots,i_k)$ is true in general.  The case $k=0$ is trivial, so we may assume that $k \geq 1$.

Suppose first that $i_k = 0$.  An inspection of the definition then shows that $\HK^{i_1,\ldots,i_{k-1},0}(G) = \HK^{i_1,\ldots,i_{k-1}}(G) \times \HK^{i_1,\ldots,i_{k-1}}(G)$, and similarly for $H$.  As a consequence, we see that $P(i_1,\ldots,i_{k-1},0)$ is implied by $P(i_1,\ldots,i_{k-1})$.  From this observation we may assume without loss of generality that $i_k \geq 1$.  By symmetry, we may in fact assume that $i_j \geq 1$ for $j=1,\ldots,k$.

Consider the map $\eta: G^{\{0,1\}^k} \to G^{\{0,1\}^{k+i_k-1}}$ defined by
\begin{equation}\label{tgdef}
\eta( (g_\omega)_{\omega \in \{0,1\}^k} ) := (g_{\omega_1,\ldots,\omega_{k-1},\min(\omega_k,\ldots,\omega_{k+i_k-1})})_{\omega \in \{0,1\}^{k+i_k-1}}.
\end{equation}
This is clearly an injective group homomorphism when $i_k \geq 1$.  We claim that 
\begin{equation}\label{eta-g}
 \eta(\HK^{i_1,\ldots,i_k}(G)) =  \eta(G^{\{0,1\}^k}) \cap \HK^{\tilde i_1,\ldots,\tilde i_{k+i_k-1}}(G),
\end{equation}
where $\tilde i_j$ is equal to $i_j$ for $j<k$ and equal to $1$ for $k \le j \leq k+i_k-1$.  It is easy to see that the left-hand group in \eqref{eta-g} is included in the right-hand side, simply by checking what $\eta$ does to each generator of $\HK^{i_1,\ldots,i_k}(G)$.  The reverse inclusion is a little trickier.  Suppose that $g \in G^{\{0,1\}^k}$ is such that $\eta(g)$ lies in $\HK^{\tilde i_1,\ldots,\tilde i_{k+i_k-1}}(G)$.
From \eqref{taylo-exp} and induction, we see that the Taylor coefficients $\eta(g)_J$ of $\eta(g)$ vanish unless $J$ either contains $\{k,\ldots,k+i_k-1\}$ or is disjoint from $\{k,\ldots,k+i_k-1\}$.  As a consequence, each of the factors $(\eta(g)_J^{\prod_{j \in J} \omega_j})_{\omega \in \{0,1\}^{k+i_k-1}}$ of the Taylor expasion of $\eta(g)$  are equal to $\eta(g_F)$ for some generator $g_F$ of $\HK^{i_1,\ldots,i_k}(G)$, and the claim follows.

From \eqref{eta-g}, we see that $P(i_1,\ldots,i_k)$ is implies by $P(\tilde i_1,\ldots,\tilde i_{k+i_k}) = P(i_1,\ldots,i_{k-1},1,\ldots,1)$, where $i_k$ copies of $1$ appear in the latter expression.  By symmetry and iteration, we conclude that $P(i_1,\ldots,i_k)$ is implied by $P(1,\ldots,1)$ where $i_1+\ldots+i_k$ copies of $1$ appear in the latter expression, and Theorem \ref{cube} follows.
\end{proof}

\begin{remark} The group $\HK^k(G)$ itself comes with a natural filtration, with $\HK^k(G)_i$ defined to be the Host-Kra group $\HK^k(G,(G_{j+i})_{j \in \N})$ of $G$ with the shifted filtration $(G_{j+i})_{j \in \N}$; see \cite[Proposition B.15]{gtz}.  Theorem \ref{cube} can then be used to show that $\HK^k$ can be viewed as a functor from the category of filtered groups to itself.  These functors are related to each other by the pleasant identity $\HK^j \circ \HK^k = \HK^{j+k}$ for all $j,k \in\N$; in particular, one can define $\HK^k$ recursively as an iteration of the functor $\HK^1$.  We will however not adopt this perspective here.
\end{remark}

\subsection{The additive case} 

Let $V = (V,\N)$ be an additive filtered group, and let $k \in \N$.  By Proposition \ref{prop}, $\HK^k(V)$ consists precisely of those tuples of the form
$$
\left(\sum_{J \subset \{1,\ldots,k\}} (\prod_{j \in J} \omega_j) v_J\right)_{\omega \in \{0,1\}^k}$$
with $v_J \in V_{|J|}$ for all $J \subset \{1,\ldots,k\}$.  Thus, for instance, $\HK^2(V)$ is the space of all quadruples
$$ (v_{00}, v_{00} + v_{01}, v_{00} + v_{10}, v_{00} + v_{01} + v_{10} + v_{11} )$$
where $v_{00} \in V_0$, $v_{01}, v_{10} \in V_1$, and $v_{11} \in V_2$.

There is also an equivalent ``dual'' description of this space (which we will need to prove Proposition \ref{equil}):

\begin{proposition}[Description of $\HK^k(V)$]\label{hk-dest}  Let $V = (V,V_\N)$ be an additive filtered group.  Then for any $k \in \N$, $\HK^k(V)$ consists precisely of those tuples $(v_\omega)_{\omega \in \{0,1\}^k} \in V^{\{0,1\}^k}$ such that
\begin{equation}\label{vo}
 \sum_{\omega \in F} (-1)^{|\omega|} v_\omega \in V_i
\end{equation}
whenever $0 \leq i \leq k$ and $F$ is a face of $\{0,1\}$ of dimension $i$.  
\end{proposition}

Thus, for instance, $\HK^2(V)$ is the space of all quadruples $(v_{00},v_{01}, v_{10}, v_{11})$ such that
\begin{align*}
v_{00}, v_{01}, v_{10}, v_{11} &\in V_0 \\
v_{00} - v_{01}, v_{00} - v_{10}, v_{01}-v_{11}, v_{10}-v_{11} &\in V_1 \\
v_{00} - v_{01} - v_{10} + v_{11} &\in V_2.
\end{align*}
Of course, this is equivalent to the previous description of $\HK^2(V)$ after a change of variables.

\begin{proof}  Let $V^{[k]}$ denote the space of all tuples $(v_\omega)_{\omega \in \{0,1\}^k} \in V^{\{0,1\}^k}$ obeying the constraints \eqref{vo}.  This is clearly a subgroup of $V^{\{0,1\}^k}$.  By checking the generators of $\HK^k(V)$, we see that $\HK^k(V) \subset V^{[k]}$.  Now we prove the reverse inclusion $V^{[k]} \subset \HK^k(V)$.  This claim is obvious for $k=0$, so we may assume inductively that $k \geq 1$ and that the claim has already been proven for $k-1$.

Let $v := (v_\omega)_{\omega \in \{0,1\}^k} \in V^{[k]}$.  We split $v = v' + v''$, where
$$ v' := (v_{\omega_1,\ldots,\omega_{k-1},0})_{\omega \in \{0,1\}^k} $$
and $v'' := v-v'$.  It is easy to see from \eqref{vo}, that $v'$ and $v''$ lie in $V^{[k]}$.  The tuple
$$ (v_{\omega_1,\ldots,\omega_{k-1},0})_{\omega \in \{0,1\}^{k-1}} $$
lies in $V^{[k-1]}$, and hence in $\HK^{k-1}(V)$ by induction hypothesis.  Extending each generator of $\HK^{k-1}(V)$ to $\HK^k(V)$ by the homomorphism
$$(w_{\omega_1,\ldots,\omega_{k-1}})_{\omega \in \{0,1\}^{k-1}} \mapsto (w_{\omega_1,\ldots,\omega_{k-1}})_{\omega \in \{0,1\}^{k}},$$ 
we then see that $v'$ lies in $\HK^k(V)$.  In a similar spirit, we see from \eqref{vo} that the tuple
$$ (v_{\omega_1,\ldots,\omega_{k-1},0} - v_{\omega_1,\ldots,\omega_{k-1},0})_{\omega \in \{0,1\}^{k-1}}$$
lies in the analogue of $V^{[k-1]}$ in which the filtration $(V_n)_{n \in \N}$ is replaced by the shifted filtration $(V_{n+1})_{n \in \N}$.  By induction hypothesis, this tuple thus lies in the analogue of $\HK^{k-1}(V)$; extending each generator of this group to $\HK^k(V)$ by the homomorphism
$$(w_{\omega_1,\ldots,\omega_{k-1}})_{\omega \in \{0,1\}^{k-1}} \mapsto (\omega_k w_{\omega_1,\ldots,\omega_{k-1}})_{\omega \in \{0,1\}^{k}},$$ 
we see that $v''$ lies in $\HK^k(V)$.  Thus $v$ lies in $\HK^k(V)$, and the claim follows.
\end{proof}

\subsection{Equidistribution}

We isolate a special class of polynomial maps:

\begin{definition}[Weak equidistribution on cubes]\label{weak-equi}  A polynomial map $\phi: H \to G$ from one filtered group $H = (H,H_\N)$ to another $G = (G,G_\N)$ is said to be \emph{weakly equidistributed on cubes} if the maps $\HK^k(\phi): \HK^k(H) \to \HK^k(G)$ are surjective for every $k \in 		\N$.
\end{definition}

Informally, a polynomial map that is weakly equidistributed on cubes can attain every possible set of values on a cube in $H$, subject of course to the polynomiality requirement that this set of values must form a cube in $G$.  

The significance of weak equidistribution for us lies in the fact that they can be used to factorise polynomial maps:

\begin{lemma}[Factorisation via weak equidistribution]\label{equi-factor}  Let $G,H,K$ be filtered groups, let $\phi: H \to G$ be a polynomial map that is weakly equidistributed on cubes, and let $\psi: G \to K$ be a map.  Then $\psi$ is a polynomial map if and only if $\psi \circ \phi$ is polynomial.
\end{lemma}

\begin{proof} This is immediate from Theorem \ref{cube}.
\end{proof}

In practice, we will derive weak equidistribution from a stronger equidistribution property, which we formulate in the language of non-standard analysis:

\begin{definition}[Equidistribution]\label{equidef}  Let $A$ be a non-empty limit finite set, and let $B$ be a finite set.  A limit map $f: A \to B$ is said to be \emph{equidistributed} if one has
$$ \frac{1}{|A|} |\{ a \in A: f(a) = b \}| = \frac{1}{|B|} + o(1)$$
for all $b \in B$.  

A finite collection of functions $f_i:A \to B_i$, $i=1,\ldots,k$ into finite sets $B_i$ is said to be \emph{jointly equidistributed} if the combined function $(f_1,\ldots,f_k): A \to B_1 \times \ldots \times B_k$ is jointly equidistributed.
\end{definition}

\begin{definition}[Equidistribution on cubes]\label{strong-equi}  A polynomial map $\phi: H \to G$ from a limit-finite filtered group $H = (H,H_\N)$ to a finite filtered group $G = (G,G_\N)$ is said to be \emph{(strongly) equidistributed on cubes} if the maps $\HK^k(\phi): \HK^k(H) \to \HK^k(G)$ are equidistributed for every $k \in \N$.
\end{definition}

Observe that if $A$ has unbounded cardinality, then every equidistributed limit map from $A$ to a finite set $B$ is automatically surjective; in praticular, strong equidistribution implies weak equdistribution.  As a consequence of this and Lemma \ref{equi-factor}, we obtain

\begin{corollary}[Factorisation via strong equidistribution]\label{equi-factor2}  Let $G,H,K$ be filtered groups with $G$ and $G_0$ limit-finite with unbounded cardinality, and $H$ finite.  Let $\phi: H \to G$ be a polynomial map that is strongly equidistributed on cubes, and let $\psi: G \to K$ be a map.  Then $\psi$ is a polynomial map if and only if $\psi \circ \phi$ is polynomial.
\end{corollary}

For future reference, we observe a convenient criterion for equidistribution.

\begin{lemma}[Weyl equidistribution criterion]\label{weyl-equi}  Let $A$ be a non-empty limit finite set, and let $B$ be a finite abelian group.  Then a limit map $f: A \to B$ is strongly equidistributed if and only if one has
\begin{equation}\label{ea}
 \E_{a \in A} e( \xi(f(a)) ) = o(1)
\end{equation}
for all non-zero characters (i.e. homomorphisms) $\xi: B \to \T$.
\end{lemma}

\begin{proof} By the Fourier inversion formula, the condition \eqref{ea} is equivalent to the bound
$$ \E_{a \in A} F(f(a)) = \E_{b \in B} F(b) + o(1)$$
holding for all standard functions $F: B \to \C$.  But this is clearly equivalent in turn to the equidistribution of $f$.
\end{proof}

\section{Properties of non-classical polynomials}\label{nonclass}

In this appendix we prove Lemma \ref{polybasic}.  The arguments here are analogous to those established in the context of dynamical systems in \cite{berg}.

It is convenient to introduce a ring of \emph{formal differential operators}.

\begin{definition}[Differential operators]  Let $V$ be a finite-dimensional or limit finite-dimensional vector space.
A \emph{differential operator} on $V$ is a formal combination (using addition and multiplication) of integers and additive derivatives $\ader_h$ (or equivalently, the shifts $T_h$), thus for instance $3 - 5 \ader_h + 7 \ader_h \ader_k + \ader_h^3$ is a differential operator.  More generally, a \emph{formal differential operator} is a finite or infinite series $\sum_i a_i \ader_{v_{i,1}} \ldots \ader_{v_{i,d_i}}$ where $a_i \in \Z$, $v_{i,1},\ldots,v_{i,d_i} \in V$, and for each $d$ there are at most finitely many $i$ with $d_i \leq d$.  Thus for instance $\sum_{i=0}^\infty \ader_h^i$ is a formal differential operator.  Note that both differential operators and formal differential operators act linearly (over $\Z$) on $\Poly_{\leq d}(V \to G)$ for every $d$ and $G$.  We say that two formal differential operators on $V$ are \emph{equivalent} if they act the same on every space $\Poly_{\leq d}(V \to G)$.  We let $\Diff(V)$ denote the space of formal differential operators on $V$ modulo this equivalence relation; this is clearly a commutative ring (note that $\ader_h \ader_k = \ader_k \ader_h$ for all $h,k$).
\end{definition}

\begin{example}  If $h, k \in V$, then we have the \emph{cocycle equation}
\begin{equation}\label{cocycle-eq}
\ader_{h+k} = \ader_h + T_h \ader_k 
\end{equation}
in $\Diff(V)$, since we have
$$ \ader_{h+k} f(x) = \ader_h f(x) + \ader_k f(x+h)$$
for all $f: V \to G$ and $x,h,k \in V$.  This cocycle equation can also be deduced from the group law
$$
T_{h+k} = T_h T_k
$$
and the identity
$$ \ader_h = T_h - 1.$$
\end{example}

The reason for working with formal differential operators rather than genuine differential operators is that any formal differential operator of the form $1+D$, where $D$ consists of \emph{higher order terms} in the sense that it contains no constant term in its expansion, is invertible in $\Diff(V)$ by formal Neumann series:
$$ (1+D)^{-1} = 1 - D + D^2 - \ldots.$$
To illustrate this, take $h \in V$.  Since $ph = 0$, we clearly have
$$ T_{h}^p = T_{ph} = 1.$$
Expanding $T_h = 1+ \ader_h$ and using the binomial formula, we conclude after some rearrangement that
$$ \ader_h^p = - p \ader_h (1 + \frac{p-1}{2} \ader_h + \ldots + \ader_h^{p-2}).$$
The expression in parentheses can be inverted by formal Neumann series.  We conclude the following fundamental fact:

\begin{lemma}[Multiplication by $p$]\label{pmult}  For any $h \in V$, we have $p \ader_h = \ader_h^p \times I_h$ for some invertible $I_h \in \Diff(V)$.  Furthermore, $I_h$ is equal to $-1$ plus higher order terms.
\end{lemma}

A heuristic way to interpret this lemma is that the operation of multiplication by $p$ resembles a differential operator of order $p-1$; dually, $\frac{1}{p}$ resembles a polynomial of degree $p-1$.  (This may help explain the condition $i_1+\ldots+i_n+j(p-1) \leq d$ in \eqref{plan}.)

We now begin the proof of Lemma \ref{polybasic}.

The first part of claim (i) is clear by induction on $d$.  To prove the second part, we observe from the cocycle identity \eqref{cocycle-eq} that if $\ader_h P, \ader_k P$ both lie in $\Poly_{\leq d-1}(V \to G)$, then $\ader_{h+k} P$ does also, and so the second part of (i) follows from the first.  (One can also deduce (i) from Proposition \ref{polygen} and induction.)

Next, we establish (ii).  We begin with the one-dimensional case $n=1$.  For $0 \leq d \leq p-1$, the vector space $\Poly_{\leq d}(V \to \F)$ clearly contains the vector space spanned by the monomials $1,x,\ldots,x^d$, which are linearly independent as can be seen from computing a Vandermonde determinant (or using the Newton interpolation formula).  On the other hand, the differential operator $\ader_1: \Poly_{\leq d}(V \to \F) \to \Poly_{\leq d-1}(V \to \F)$ has kernel equal to the constant functions $\F$, and so the dimension of $\Poly_{\leq d}(V \to \F)$ can only exceed that of $\Poly_{\leq d-1}$ by $1$ at most.  By induction we thus see that $\Poly_{\leq d}(V \to \F)$ is equal to the $d+1$-dimensional space spanned by $1,x,\ldots,x^d$ for $0 \leq d \leq p-1$.  In particular, $\Poly_{\leq p-1}(V \to \F)$ must be equal to the $p$-dimensional space of all functions from $V$ to $\F$, and the claim follows.

Now we assume inductively that $n > 1$ and that (ii) has already been proven for smaller dimensions.  We parameterise an element $x \in \F^n$ as $x = (x',x_n)$ where $x' \in \F^{n-1}$ and $x_n \in \F$.  If $P \in \Poly_{\leq d}(V \to \F)$, then clearly the one-dimensional maps $x_n \mapsto P(x',x_n)$ are polyomials of degree $\leq d$ for each fixed $x'$.  Applying the one-dimensional case of (ii), we conclude that
$$ P(x',x_n) = \sum_{0 \leq i_n \leq \min(p-1,d)} P_{i_n}(x') x_n^{i_n}$$
for some functions $P_{i_n}: \F^{n-1} \to \F$ that are uniquely determined by $P$.  Differentiating this identity $i_n$ times in the direction of the $n^{th}$ generator $e_n$ of $\F^n$, and $d-i_n$ times in directions in $\F^{n-1}$, we conclude that each $P_{i_n}$ is a polynomial of degree $\leq d-i_n$.  The claim then follows from the induction hypothesis.

We skip (iii) for the moment and move on to (iv).  If $h \in V$ and $P, Q: V \to R$, we have
$$ T_h(PQ) = (T_h P) (T_h Q);$$
expanding $T_h = 1 + \ader_h$, we conclude the \emph{discrete Leibniz rule}
\begin{equation}\label{leibniz}
\ader_h (PQ) = (\ader_h P) Q + P (\ader_h Q) + (\ader_h P) (\ader_h Q).
\end{equation}
The claim (iv) can now be easily established by an induction on $d+d'$ (noting that the claim is trivial if $d$ or $d'$ is negative).   (Alternatively, one can deduce (iv) from \cite[Example B.18]{gtz}.)

We remark that one should view the final term $(\ader_h P) (\ader_h Q)$ in \eqref{leibniz} to be a lower order error term, so that \eqref{leibniz} becomes a perturbation of the classical Leibniz rule $D(PQ) = (DP) Q + P(DQ)$ for derivations $D$.

Now we prove part of (v).  If $P \in \Poly_{\leq d}(V \to G)$ for some $d \geq p-1$, then 
$$\ader_{h_1}^p \ader_{h_2} \ldots \ader_{h_{d-p+2}} P = 0$$
for any $h_1,\ldots,h_{d-p+2} \in V$, and thus by Lemma \ref{pmult}
$$\ader_{h_1} \ader_{h_2} \ldots \ader_{h_{d-p+2}} pP = 0.$$
We conclude that $pP \in \Poly_{\leq d-p+1}(V \to G)$.  We conclude that for any integer $d$, the map $p: P \mapsto pP$ maps $\Poly_{\leq d}(V \to G)$ to $\Poly_{\leq \max(d-p+1,0)}(V \to G)$.  This proves everything in (v) except for the assertion that this map $P \mapsto pP$ is surjective.

Now we return to (iii). The fact that every  expression $P$ of the form \eqref{plan} is a polynomial of degree $\leq d$, and vice versa follows from the special case of Proposition \ref{erz-class} when all the initial degrees $D_i$ are equal to $1$.  (This argument is non-circular, because Lemma \ref{polybasic}(iii) is not used in the proof of Proposition \ref{erz-class}.  Another proof of this part of Lemma \ref{polybasic}(iii) can be found in \cite[\S 1.12]{tao-poincare}. 

Now we establish the uniqueness claim in (iii).  The claim is trivial for $d=0$, so suppose inductively that $d \geq 1$ and that uniqueness has already been established for smaller values of $d$.  Since $\alpha = P(0)$ from \eqref{plan} we see that the $\alpha$ are unique and can thus be subtracted away.
Applying the uniqueness claim to the lower-degree polynomial $pP$, we see that all the coefficients in \eqref{plan} with $j \geq 1$ are unique.  Subtracting off these terms also, we are left with a classical polynomial expansion \eqref{iot}, and the claim follows from (ii).  This concludes the proof of (iii).

Now that we have (iii), the surjectivity claim of (v) is immediate, since one simply replaces all the $p^{j+1}$ denominators in \eqref{plan} by $p^{j+2}$, and replaces $\alpha$ with a $p^{\th}$ root $\alpha'$ as in the proof of (iii).  This completes the proof of (v).

The claim (vi) follows immediately from (v) and an induction on $d$.

%Finally, we establish (vii).  We expand
%$$ \sum_{i \in \F_p} P(x+iv) = \sum_{i=0}^{p-1} (1+\ader_v)^i P(x),$$
%which by the geometric series formula is (formally) equal to
%$$ \frac{(1+\ader_v)^p - 1}{\ader_v} P(x)$$
%which we can expand as
%$$ (p(1 + \frac{p-1}{2} \ader_v + \ldots + \ader_v^{p-2}) + \ader_v^{p-1})P(x).$$
%If $P$ has degree $\leq p-2$, then $\ader_v^{p-1}$ vanishes, and so
%$$ \sum_{i \in \F_p} P(x+iv) = (1 + \frac{p-1}{2} \ader_v + \ldots + \ader_v^{p-2}) pP(x).$$
%The operator $1 + \frac{p-1}{2} \ader_v + \ldots + \ader_v^{p-2}$ can be inverted by Neumann series, and the claim follows.

\section{On exact roots in dynamical systems}\label{root-counter}

In \cite{berg}, \cite{tao-ziegler}, the inverse conjecture for the Gowers norms were attacked via an ergodic theory approach, based on a structual analysis of \emph{$\Fw$-systems}.  These systems consisted of a probability space $(X,{\mathcal B},\mu)$, together with a measure-preserving action $(T_g)_{g \in \Fw}$ of the infinite vector space $\Fw := \bigcup_{n=1}^\infty \F^n$ (where we nest $\F^n$ inside $\F^{n+1}$ in the obvious manner).  

A \emph{polynomial of degree $\leq d$} on such a system is a measurable function $P: X \to \T$ such that $\ader_{h_1} \ldots \ader_{h_{d+1}} P = 0$ a.e. for all $h_1,\ldots,h_{d+1} \in \Fw$, where $\ader_h f := f \circ T_h - 1$.  The analogue of the exact roots property from Remark \ref{exact} is then

\begin{claim}\label{root}  Let $X = (X, {\mathcal B}, \mu, (T_g)_{g \in \Fw})$ be a $\Fw$-system, let $d \geq 0$ be an integer, and let $P: X \to \T$ be a polynomial of degree $\leq d$.  Then there exists a polynomial $Q: X \to \T$ of degree $\leq d+p-1$ with $pQ = P$.
\end{claim}

Using Lemma \ref{polybasic}(v), it is not difficult to verify this claim when $X$ is finite; the claim is also easy when $d=1$, as one can then (up to a constant) express $P = \iota(\tilde P)$ for some linear polynomial $\tilde P: X \to \F$, and one can verify that the polynomial $Q := e( |\tilde P|/p^2 )$ will have the desired properties (note from Lemma \ref{polybasic}(iii) that the map $x \mapsto |x|/p^2 \mod 1$ has degree $\leq p$ on $\F$).  Unfortunately, the claim fails in general.  For instance, we have

\begin{proposition}\label{root-counter-prop}  Claim \ref{root} is false when $p=2$ and $d=2$.
\end{proposition}

The purpose of this appendix is to prove this proposition, which explains why we were unable to use the ergodic theory method from \cite{tao-ziegler}, \cite{berg} to establish the main results in this paper.

The reason that the exact roots property holds in the finitary setting but not in the ergodic setting can be explained as follows.  In the finite setting $\F_2^n$, all functions are considered to be measurable; but in the ergodic setting, only a limited number of functions are measurable.  For instance, one may be working in a system generated by a single function $f$ and its shifts $T_h f$, so that every measurable function in the system can be approximated to arbitrary accuracy by a finite combination of shifts $T_h f$ of these functions.  

Now consider the finitary quadratic function $P := \frac{L}{4} \mod 1$ on $\F_2^n$, which is the analogue of the function $\iota_4(t)$ considered above.  This function has a degree $\leq 3$ root, namely the function $Q := \frac{L}{8} \mod 1$.  However, this polynomial $Q$ is not ``measurable'' in the system generated by $P$, in the sense that one cannot express (even approxmiately) $Q$ as a function of a bounded number of shifts $\ader_h P$ of $P$; indeed, one can formulate this precisely and then deduce this from the arguments in the proof of Proposition \ref{root-counter-prop} given below, combined with the Furstenberg correspondence principle, but we will not do this here.

On the other hand, if we allow ourselves the freedom to \emph{extend} the system $X$ to a larger one, then it appears that one can recover the exact roots property.  For instance, in the system $X = \F^\N \times_\rho \Z/4\Z$ defined below, we may extend this system to the system $Y := \F^\N \times_{\rho_8} \Z/8\Z$, where the cocycle $\rho_8$ is defined exactly as with $\rho = \rho_4$ but using the modulus $8$ rather than $4$.  Letting $t_8$ be the vertical coordinate function of $Y$, the degree $\leq 2$ polynomial $\iota_4(t)$ in $X$ then lifts to $2\iota_8(t_8)$ in $Y$, where $\iota_8: \Z/8\Z \to \{0,1\ldots,7\}$ is the obvious map.  This has an obvious root that is of degree $\leq 3$, namely $\iota_8(t_8)$; this function is the ergodic analogue of the finitary function $Q = \frac{L}{8} \mod 1$ considered earlier.  It is likely that this phenomenon generalises, in that Claim \ref{root} becomes true again if we allow $Q$ to take values in an extension of $P$, but we will not pursue this statement here.  (Note though that one can use extensions to simplify the proofs of various multiple recurrence and convergence results in ergodic theory; see for instance \cite{austin}.)

In principle, this weakened form of Claim \ref{root}, in which the root takes values in an extension, may possibly be used to extend the ergodic theory arguments in \cite{berg} to the low characteristic setting, and in particular to recover the ergodic version of the inverse conjecture for the Gowers norms for $\Fw$ (see \cite{berg}, \cite{tao-ziegler}) in that setting.  However, we were not able to achieve this, as the inductive arguments in \cite{berg} rely on reducing the system and are thus not compatible with taking extensions.

We now begin the formal proof of Proposition \ref{root-counter-prop}. Let $\F := \F_2$. Consider the Cantor space $\F^\N = \prod_{n \in \N} \F^\N$ (with the product $\sigma$-algebra and the uniform (Bernoulli) probability measure), which has an obvious action of $\Fw$.  If we let $x_i: \F^\N \to \F$, $i \in \N$ be the coordinate functions, we thus have
$$ \ader_{e_i} x_j = \delta_{ij}$$
for $i,j \in \N$, where $e_1,e_2,\ldots$ are the generators of $\Fw$ and $\delta_{ij}$ is the Kronecker delta.  In particular, the functions $\iota(x_j): X \to \T$ are polynomials of degree $\leq 1$.

We define a cocycle $\rho = \rho_4: \Fw \times \F^\N \to \Z/4\Z$ taking values in the cyclic group $\Z/4\Z$ by the formula
\begin{equation}\label{rhon}
 \rho(\sum_{i \in A} e_i, x) := \sum_{i \in A}(1 - 2 |x_i|) \mod 4
\end{equation}
for any finite set $A$.  One easily verifies the \emph{cocycle equation}
$$ \rho(h+k,x) = \rho(h,x) + \rho(k,T_h x)$$
for any $h,k \in \Fw$.  We can therefore build the cocycle extension $X := \F^\N \times_\rho \Z/4\Z$ of $\F^\N$, defined as the space of pairs $(x,t)$ with $x \in \F^\N$ and $t \in \Z/4\Z$ (with the product probability measure) endowed with the shift
$$ T_h (x,t) := (T_h x, t + \rho(h,x)).$$
This can be easily verified to be a $\Fw$-system.  If we let $\iota_4: \Z/4\Z \to \T$ be the map $\iota_4(i) := i/4$, and $t: (x,t) \mapsto t$ be the coordinate function, then the function $\iota_4(t)$ has derivatives
$$ \ader_h \iota_4(t) = \iota_4( \rho(h,\cdot) ),$$
and thus by \eqref{rhon}, $\iota_4(t)$ is a polynomial of degree $\leq 2$.  Similarly, $2\iota_4(t)$ is a polynomial of degree $\leq 1$.

If Claim \ref{root} was true, then there would exist a polynomial $Q: X \to \T$ of degree $\leq 3$ such that $2Q = \iota_4(t)$.  To show that this is not possible, we need to classify all the polynomials of degree $\leq 3$.

We begin with the polynomials of degree $\leq 0$, which (up to measure zero errors) are simply the $\Fw$-invariant functions $P: X \to \T$.  We claim that $X$ is ergodic, so that the only invariant functions are the constants (up to measure zero errors).  It suffices to show that every invariant set $E$ in $X$ has zero measure or full measure.  Given any $\eps$, we can approximate $E$ to an error of measure $\eps$ by a set $E_\eps$ which depends on only a finite number $x_1,\ldots,x_n$ of the base coordinate functions, together with the vertical coordinate $t$.  The set $E_\eps$ is then invariant up to errors of measure $O(\eps)$.  Inspecting the action of the shift $T_{e_{n+1}}$ on $E_\eps$, one then easily concludes that $E_\eps$ must differ by an error of $O(\eps)$ from a set which does not depend on the vertical coordinate.  Taking $\eps \to 0$, we conclude that $E$ is (up to measure zero errors) independent of the vertical coordinate, and thus descends to an invariant subset of $\F^\N$.  But it is standard from the theory of Bernoulli systems that such sets have either zero measure or full measure.

Now we classify the polynomials of degree $\leq 1$.  If $P$ has degree $\leq 1$, then for each $e_i$, $\ader_{e_i} P$ is degree $\leq 0$ and hence constant.  Since $\ader_{2e_i} P$ is necessarily trivial, we conclude from the cocycle equation $\ader_{h+k} = \ader_h + T_h \ader_k$ with $h = k = e_i$ that $2 \ader_{e_i} P = 0$.  Thus we can find coefficients $c_i \in \F$ such that $\ader_{e_i} P = c_i$ for all $i$.  Also, since $P$ is measurable, it differs by an error of $\eps$ (in measure) from a function which depends on only finitely many of the $x_1,\ldots,x_n$ and $t$.  Note that if $n_1 > n_2>n$ and $(x,t)$ are such that $x_{n_1} \neq x_{n_2}$, then the shift $T_{e_{n_1}-e_{n_2}}$ does not affect the $x_1,\ldots,x_n,t$ coefficients.  On the other hand, $T_{e_{n_1}-e_{n_2}} P = P + c_{n_1} - c_{n_2}$.
and thus (if $\eps$ is small enough) we have $c_{n_1} = c$ independent of $n_1$ for $n_1 > n$.
We now see that $P$ differs from the degree $\leq 1$ polynomial 
$$\sum_{i=1}^n (c_i-c) \iota(x_i) + 2\iota_4(t)$$ 
by a degree $\leq 0$ polynomial, which is thus constant.  Thus all degree $\leq 1$ polynomials take the form
$$ P =  \sum_i c_i \iota(x_i) + 2 d \iota_4(t) + \alpha$$
where $\alpha \in \T$, $d \in \Z$, and at most finitely many of the $c_i \in \Z$ are non-zero.

To classify polynomials of higher degree we employ the method of \emph{vertical differentiation}, which is used extensively in the ergodic theory literature (see for instance \cite{host-kra}, \cite{zieg-jams}, \cite{berg}); for this simple example we use a very concrete instance of this method here.  We define the vertical derivatives $\Delta_s P$ of a function $P: X \to \T$ for any $s \in \Z/4\Z$ by the formula
$$ \Delta_s P(x,t) := P(x,t+s) - P(x,t).$$
Observe that these operators commute with themselves and with the $\Fw$ action, and in particular commute with the ordinary derivatives $\ader_h$.

The key observation is that $\Delta_2$ behaves like a differential operator of order two:

\begin{lemma}\label{verto} Let $s=1,2$.  If $P: X \to \T$ has degree $\leq d$ for some integer $d$, then $\Delta_s P$ has degree $\leq d-s$.
\end{lemma}

\begin{proof} By repeated differentiation it suffices to verify this when $d=s-1$.  But this follows from the classification of polynomials of degree $\leq 0$ and degree $\leq 1$ that has already been established.
\end{proof}

We can now classify polynomials of degree $\leq 2$ (cf. Lemma \ref{polybasic}(iii)):

\begin{lemma}[Classification of quadratics]\label{quadrat} Let $P: X \to \T$ be of degree $\leq 2$.  Then we can write
\begin{equation}\label{pdecomp}
 P = \sum_{i<j} c_{ij} \frac{|x_i| |x_j|}{2} + \sum_i c_i \frac{|x_i|}{4} + \sum_i d_i \frac{|x_i| |S_1|}{2} + d \frac{|S_1|}{4} + e \frac{|t|}{4} + \alpha \mod 1
\end{equation}
where $c_{ij}, c_i, d_i, d, e$ are integers, of which only finitely many are non-zero, $\alpha \in \T$, and $S_1: X \to \F$ is the function $S_1 := t \mod 2$, and $x \mapsto |x|$ for $x \in \Z/4\Z$ is the obvious map to the fundamental domain $\{0,1,2,3\}$.
\end{lemma}

\begin{proof}  By a computation, one verifies that all the expressions on the right-hand side of \eqref{pdecomp} are indeed polynomials of degree $\leq 2$.

From Lemma \ref{verto}, $\Delta_2 P$ is constant.  Using the cocycle identity
\begin{equation}\label{coc2}
\Delta_2 P(x,t) + \Delta_2 P(x,t+2) = 0
\end{equation}
we see that this constant is either $0$ or $1/2$.  In the latter case, we can subtract off $|t|/4$ to reduce to the former case (noting that $\Delta_2 |t|/4 = 1/2$); so we may assume without loss of generality that $\Delta_2 P = 0$, thus $P$ descends to the reduced system $\Fw \times_{\rho \mod 2} \Z/2\Z$.  

From Lemma \ref{verto} again, $\Delta_1 P$ is linear, and thus takes the form
$$ \Delta_1 P(x,t) = \sum_i d_i \frac{|x_i|}{2} + \beta$$
for some integers $d_i$ (of which only finitely many are non-zero) and $\beta \in \T$.  With the cocycle identity
\begin{equation}\label{coc1}
\Delta_1 P(x,t) + \Delta_1 P(x,t+1) = \Delta_2 P(x,t) = 0
\end{equation}
we see that $\beta$ is either $0$ or $1/2$, thus $\beta = d/2$ for some integer $d$.  If we then subtract off
$\sum_i d_i \frac{|x_i| |S_1|}{2} + d \frac{|S_1|}{4}$ from $P$ (noting that $\Delta_1 S_1 = 1$) we can reduce to the case $\Delta_1 P = 0$, thus $P$ now descends to a function of $\F^\N$.

For any $\eps$, we may approximate $P$ in measure to error $\eps$ by a function depending only on finitely many $x_1,\ldots,x_n$ of the coefficients.  For any $n' > n$, the linear polynomial $\ader_{e_{n'}} P$ is then within $O(\eps)$ of zero in measure, and is thus constant (by the classification of linear polynomials); using the cocycle identity $0 = \ader_{e_{n'}} P + T_{e_{n'}} \ader_{e_{n'}} P$ we see that $\ader_{e_{n'}} P$ is in fact identically zero, thus $P$ in fact descends to a function of just a finite number of coordinates $x_1,\ldots,x_n$.  The claim now follows from Lemma \ref{polybasic}(ii).
\end{proof}

In a similar vein, we can classify cubics:

\begin{lemma}[Classification of cubics] Let $P: X \to \T$ be of degree $\leq 3$.  Then $P$ is an integer linear combination of a finite number of the following functions:
\begin{itemize}
\item[(i)] Constants $\alpha \in \T$;
\item[(ii)] $|x_i| |x_j| |x_k| / 2$ for natural numbers $i < j < k$;
\item[(iii)] $|x_i| |x_j| / 4$ for natural numbers $i < j$;
\item[(iv)] $|x_i| / 8$ for a natural number $i$;
\item[(v)] $|x_i| |x_j| |S_1| / 2$ for natural numbers $i < j$;
\item[(vi)] $|x_i| |S_1| / 4$ for a natural number $i$;
\item[(vii)] $|S_1| / 8$ for a natural number $i$;
\item[(viii)] $|x_i| |S_2|/2$ for a natural number $i$; 
\item[(ix)] $|S_1| |S_2|/2$;
\item[(x)] $|S_2|/2$.
\end{itemize}
Here $S_2: X \to \F$ is the function such that $S_2(x,t)=1$ when $t=2,3 \mod 4$ and $S_2(x,t) = 0$ otherwise.
\end{lemma}

\begin{remark} The polynomials $S_1, S_2$ can be viewed as the ergodic limit as $n \to \infty$ (using the Furstenberg correspondence principle) of the symmetric polynomials $S_1, S_2$ considered in Example \ref{exam}, where the coordinate functions $x_i$ correspond to the usual coordinate functions on $\F_2^n$, and $t$ corresponds to $L \mod 4$.  The cubic polynomial $\iota_4(t) = |t|/4$ can be expressed as $|t|/4 = |S_2|/2 + |S_1|/4$; cf. \eqref{lam}.
\end{remark}

\begin{proof} By a (somewhat tedious) computation we see that all the above functions are polynomials of degree $\leq 3$ (and that $S_2$ is a polynomial of degree $\leq 2$).

From Lemma \ref{verto}, $\Delta_2 P$ is of degree $\leq 1$, so we can express 
$$\Delta_2 P = \sum_i c_i |x_i|/2 + c |S_1|/2 + \alpha$$
for some integers $c_i, c$ (only finitely many of which are non-zero) and $\alpha \in \T$.  Using \eqref{coc2} we can write $\alpha = d/2$ for some integer $d$.  By subtracting $\sum_i c_i |x_i| |S_2| / 2 + c |S_1| |S_2|/2 + d |S_2|/2$ from $P$ we may thus assume that $\Delta_2 P = 0$.

Next, $\Delta_1 P$ is of degree $\leq 2$, and is annihilated by $\Delta_2$, so by Lemma \ref{quadrat}, we have
$$
\Delta_1 P = \sum_{i<j} c_{ij} \frac{|x_i| |x_j|}{2} + \sum_i c_i \frac{|x_i|}{4} + \sum_i d_i \frac{|x_i| |S_1|}{2} + d \frac{|S_1|}{4} + \alpha \mod 1$$
for some integers $c_{ij}, c_i, d_i, d$ (only finitely many of which are non-zero) and $\alpha \in \T$.  Using \eqref{coc1}, we conclude that the $\frac{c_i+d_i}{2}$ must vanish, and that $\frac{d}{4} + 2\alpha = 0$; thus we can simplify the above expression to
$$
\Delta_1 P = \sum_{i<j} c_{ij} \frac{|x_i| |x_j|}{2} + \sum_i d_i (\frac{|x_i| |S_1|}{2} - \frac{|x_i|}{4}) + d (\frac{|S_1|}{4} - \frac{1}{8} ) + \frac{e}{2} \mod 1$$
for some integer $e$.  If we then subtract off the cubic polynomial
$$ \sum_{i<j} c_{ij} \frac{|x_i| |x_j| |S_1|}{2} + \sum_i d_i \frac{|x_i| |S_1|}{4} + d \frac{|S_1|}{8} + \frac{e |S_1|}{2} $$
from $P$, we can reduce to the case $\Delta_1 P = 0$, thus $P$ descends to $\F^\N$.  Arguing as in Lemma \ref{quadrat}, we conclude that $P$ is a function of finitely many coordinates $x_1,\ldots,x_n$, and the claim follows from Lemma \ref{polybasic}(iii).
\end{proof}

From the above lemma, we see that if $P$ is a polynomial of degree $\leq 3$, then $2P$ is an integer linear combination of the following types of functions:
\begin{itemize}
\item[(i)] Constants $\alpha \in \T$;
\item[(ii)] $|x_i| |x_j| / 2$ for natural numbers $i < j$;
\item[(iii)] $|x_i| / 4$ for a natural number $i$;
\item[(iv)] $|x_i| |S_1| / 2$ for a natural number $i$;
\item[(v)] $|S_1| / 4$ for a natural number $i$.
\end{itemize}
In particular, $\Delta_2(2P)$ must vanish.  On the other hand, $\Delta_2 \iota_4(t) = \frac{1}{2} \neq 0$.  Thus $\iota_4(t)$ is not of the form $2P$ for a polynomial of degree $\leq 3$, thus establishing Proposition \ref{root-counter-prop}.

\begin{remark} One can avoid the full classification of polynomials of degree $\leq 3$ in proving the above proposition.
Indeed, if $P$ is cubic, then from Lemma \ref{verto} $\Delta_2 \Delta_2 P = 0$, and then by \eqref{coc2} $\Delta_2(2P)=0$, and we can conclude as above.  
\end{remark}

\end{document}